\documentclass[reqno,10pt]{amsart}
\usepackage{amssymb}
\usepackage{amssymb, amsmath, graphicx}
\usepackage{bbm}
\usepackage{color}
\usepackage{verbatim}
\usepackage[margin=0.8in]{geometry}
\usepackage{comment}
\usepackage{import}
\usepackage{tikz}
\usepackage{latexsym}
\usepackage[makeroom]{cancel}
\usepackage{caption}

\newcommand{\andf}{\quad\hbox{and}\quad}
\newcommand{\with}{\quad\hbox{with}\quad}
\newcommand{\newcom}{\newcommand}

\def\inte#1{
\displaystyle\mathop{#1\kern0pt}^\circ }

\newcom{\al}{\alpha}
\newcom{\de}{\delta}
\newcom{\Th}{\Theta}
\newcom{\be}{\beta}
\newcom{\s}{\sigma}
\newcom{\eps}{\epsilon }
\newcom{\ve}{\varepsilon}
\newcom{\ga}{\gamma}
\newcom{\Ga}{\Gamma}
\newcom{\ka}{\kappa}
\newcom{\Lam}{\Lambda}
\newcom{\lam}{\lambda}
\newcom{\bp}{\Phi}
\newcom{\om}{\omega}
\newcom{\Sig}{\Sigma}
\newcom{\sig}{\sigma}
\newcom{\tht}{\theta}
\newcom{\tri}{\triangle}
\newcom{\oo}{\infty}
\newcom{\h}{{\rm h}}
\newcom{\rmv}{{\rm v}}
\newcom{\hs}{\hslash}
\newcom{\vphi}{\varphi}
\def\dive{\mathop{\rm div}\nolimits}
\newcom{\grad}{\nabla}
\newcom{\lap}{\Delta}
\newcom{\ca}{{\mathcal a}}
\newcom{\cB}{{\mathcal B}}
\newcom{\cC}{{\mathcal C}}
\newcom{\cD}{{\mathcal D}}
\newcom{\cE}{{\mathcal E}}
\newcom{\cF}{{\mathcal F}}
\newcom{\cG}{{\mathcal G}}
\newcom{\cH}{{\mathcal H}}
\newcom{\cI}{{\mathcal I}}
\newcom{\cJ}{{\mathcal J}}
\newcom{\cL}{{\mathcal L}}
\newcom{\cK}{{\mathcal K}}
\newcom{\cM}{{\mathcal M}}
\newcom{\cN}{{\mathcal N}}
\newcom{\cP}{{\mathcal P}}
\newcom{\cS}{{\mathcal S}}
\newcom{\cO}{{\mathcal O}}
\newcom{\cQ}{{\mathcal Q}}
\newcom{\cT}{{\mathcal T}}
\newcom{\cY}{{\mathcal Y}}
\newcom{\cZ}{{\mathcal Z}}
\newcom{\cR}{{\mathcal R}}
\newcom{\cU}{{\mathcal U}}
\newcom{\cV}{{\mathcal V}}
\newcom{\cW}{{\mathcal W}}
\newcom{\cX}{{\mathcal X}}
\newcom{\R}{\Bbb R}
\newcom{\T}{\Bbb T}
\newcom{\N}{\Bbb N}
\newcom{\Z}{\Bbb Z}
\newcom{\C}{\Bbb C}
\newcom{\E}{\Bbb E}
\newcom{\pr}{\Bbb P}
\let\wh=\widehat
\newcom{\bfp}{\mathbf{p}}
\newcom{\bfq}{\mathbf{q}}
\newcom{\bfr}{\mathbf{\wh{R}}}
\newcom{\bfn}{\mathbf{\wh{N}}}
\newcom{\f}{\frac}
\newcom{\dint}{\displaystyle\int}
\newcom{\dsum}{\displaystyle\sum}
\newcom{\dlim}{\displaystyle\lim}
\newcom{\ov}{\overline}
\newcom{\wt}{\widetilde}
\newcom{\pt}{\partial_t}
\newcom{\p}{\partial}
\newcom\na{\nabla}
\newcom\rto{\rightarrow}
\newcom\lto{\leftarrow}
\newcom\mto{\mapsto}
\newcom{\disp}{\displaystyle}
\newcom{\non}{\nonumber}
\newcom{\no}{\noindent}
\newcom{\QED}{$\square$}

\def\eqdefa{\buildrel\hbox{\footnotesize def}\over =}

\newcommand{\beq}{\begin{equation}}
\newcommand{\eeq}{\end{equation}}
\newcommand{\beqo}{\begin{equation*}}
\newcommand{\eeqo}{\end{equation*}}
\newcommand{\ben}{\begin{eqnarray}}
\newcommand{\een}{\end{eqnarray}}
\newcommand{\beno}{\begin{eqnarray*}}
\newcommand{\eeno}{\end{eqnarray*}}

\newcom{\is}{{i^\star}}
\newcom{\sM}{\mathscr{M}}
\newcom{\eo}{{\mathbf{e}_1}}
\newcom{\et}{\mathbf{e}_2}

\newcom{\Ein}{\text{Ein}}
\newcom{\nl}{\text{NL}}
\newcom{\lt}{{L^2}}
\newcom{\sa}{\mathscr{a}}
\newcom{\sF}{\mathscr{F}}
\newcom{\sN}{\mathscr{N}}
\newcom{\sI}{\mathscr{I}}
\newcom{\lla}{\left\langle}
\newcom{\rra}{\right\rangle}

\newcom{\lbk}{\left\{}
\newcom{\rbk}{\right\}}

\newcom{\lamet}{\Lam^{E,\star}}
\newcom{\spn}{\text{span}}

\newcom{\oma}{{\Omega_a}}
\newcom{\omia}{{\Omega_{i,a}}}
\newcom{\pia}{{\Phi_{i,a}}}
\newcom{\pisa}{{\Phi_{\is,a}}}
\newcom{\ala}{{\al_a}}
\newcom{\alpa}{{\dot{\al}_a}}
\newcom{\thpa}{{\dot{\theta}_a}}
\newcom{\thpp}{{\dot{\theta}_p}}
\newcom{\dtea}{\dot{\theta}^E_a}
\newcom{\wiv}{{W_{i,\ve}}}
\newcom{\civ}{{\cX_{i,\ve}}}
\newcom{\Div}{{\cD_{i,\ve}}}
\newcom{\ive}{{\mathrm{I}_{i,\ve}}}
\newcom{\iive}{{\mathrm{II}_{i,\ve}}}
\newcom{\iiive}{{\mathrm{III}_{i,\ve}}}

\newtheorem{Theorem}{Theorem}[section]

\newtheorem{Proposition}[Theorem]{Proposition}
\newtheorem{Lemma}[Theorem]{Lemma}

\newtheorem{Remark}[Theorem]{Remark}

\begin{document}
\title[Long time evolution of viscous point vortices]
{Long time evolution of a pair of 2D viscous point vortices}

\author[P. Zhang]{Ping Zhang}%
\address[P. Zhang]
{State Key Laboratory of Mathematical Sciences, Academy of Mathematics $\&$ Systems Science, The Chinese Academy of
	Sciences, Beijing 100190, China, and School of Mathematical Sciences,
	University of Chinese Academy of Sciences, Beijing 100049, CHINA.}
\email{zp@amss.ac.cn}

\author[Y. Zhang]{Yibin Zhang}
\address[Y. Zhang]
 {academy of Mathematics $\&$ Systems Science, The Chinese academy of
Sciences, Beijing 100190, CHINa.  } \email{zhangyibin22@mails.ucas.ac.cn}

\date{\today}

\begin{abstract}
This paper studies the long-time evolution of two  point vortices under the 2D Navier–Stokes equations. Starting from initial data given by a pair of Dirac measures, we derive an asymptotic expansion for the vorticity  over time scales significantly longer than the advection time, yet shorter than the diffusion time. Building on previous works \cite{GS24-1, DG24}, we construct suitable approximate solutions $\Omega_a$ and employ Arnold’s method to define a nonlinear energy functional $E_\ve[\om]$, with respect to which the linearized operator $\Lambda^{E,\star}$ around $\Omega_a$ is nearly skew-adjoint. A key innovation in this work is the introduction of ``pseudo-momenta'': $\varrho^e_a, \varrho^o_a,\varrho^{te}_a, \varrho^{to}_a$, which correspond to eigenfunctions or other nontrivial elements in invariant subspaces of $\Lambda^E$, derived from the Lie structure of the 2D Euler equations.
In particular, by transitioning to a co-moving frame that tracks the trajectories of the vortex centers and introducing the ``frame stream function" $\hat{X}$ and  its associated ``frame derivative" $X=\{\hat{X}, \cdot\}_V$, we obtain the relation $\{\varrho^e_a, \Omega^E_a\}_V \approx X\Omega^E_a$, which  is shown to be an important eigenfunction of $\Lam^E_a$ of eigenvalue zero.
Imposing orthogonality of the perturbation $\omega$ with respect to these carefully chosen ``pseudo-momenta" enables us to prove the coercivity of $E_\ve[\omega]$.
 Finally, we estimate the error between $\Omega_a$ and the true solution  by projecting it onto the directions spanned by these ``pseudo-momenta'' and their orthogonal complement, and then analyzing the evolution of each projection. Our analysis reveals how the vortex centers evolve and how the solution deviates from a simple superposition of Oseen vortices, generalizing previous result \cite{DG24} and providing a refined description of metastable vortex dynamics in the small viscosity regime.
\end{abstract}

\maketitle

\renewcommand{\theequation}{\thesection.\arabic{equation}}
\setcounter{equation}{0}

\section{introduction}
The study of coherent vortex structures in two-dimensional incompressible flows, which refers to organized, swirling patterns of fluid motion that persist in time and space,  has long been a central topic in fluid dynamics, both from a theoretical and applied perspective. At high Reynolds numbers, such structures often persist over long time intervals and dominate the transport and mixing properties of the flow.
Far from being mere theoretical concepts, these structures are ubiquitous in nature and technology. They are the invisible architects behind the familiar: from the swirling eddies that shed off a bridge pier in a river, influencing erosion and sediment transport, to the powerful wingtip vortices generated by aircraft, which dictate safe separation distances during takeoff and landing. In the atmosphere, large-scale coherent vortices define weather systems, including hurricanes and tornadoes, while in engineering, controlling them is key to enhancing the efficiency of turbines and reducing drag on vehicles. Essentially, these coherent structures are the fundamental building blocks of fluid flow, connecting abstract theory to the dynamic world we live in.

A natural entry into this field is the study of vortex interactions, which are mathematically more tractable yet more singular than vortex patches or sheets. Indeed, according to numerical studies \cite{DV02} and nonrigorous asymptotic analysis \cite{TK91,TT65}, the interaction of vortices commences with a rapid relaxation phase. In this stage, each vortex undergoes shape adaptation in response to the velocity field produced by surrounding vortices. This initial phase is sensitive to the detailed initial conditions and features oscillatory behavior in the vortex cores that diminish within an inviscid time scale. Following this, the vortices gradually approach Gaussian-type profiles at a diffusion-controlled rate. The system eventually attains a “metastable regime” that no longer reflects the initial data and remains stable until vortices come into close proximity, triggering a merging process.
During this ``metastable regime", the dynamic of the vortex centers is characterized  by the following ``point vortex system", which originally derived by Helmholtz \cite{H58} and Kirchhoff \cite{K76}:
\begin{equation}\label{eq 1.1a}
    \dot{z}_i(t) = \frac{1}{2\pi} \sum_{j\neq i} \Ga_j \frac{(z_i(t)-z_j(t))^\perp}{|z_i(t)-z_j(t)|^2}.
\end{equation}
Here $z_i \in \R^2$ and $\Ga_i\in \R$ designate the positions and circulations of the $i$-th vortice.   This ODE system describes the motion of idealized singular vortices in an inviscid fluid,  and arises naturally as a formal limit of the 2D Euler or Navier–Stokes equations when the vorticity is highly concentrated.

To render this approximation mathematically rigorous, we start from the 2D Navier-Stokes equations in vorticity formulation:
\begin{equation}\label{eq 1.1}
\pt w + u\cdot \grad_x w-\nu \lap_x w=0, \quad  (t,x) \in \R^+\times \R^2,
\end{equation}
where $\nu\ll1$ represents the viscous coefficient, $\omega$ designates the vorticity of the fluid and the  corresponding velocity $u$ can be expressed by the Biot-Savart formula
\begin{equation*}
    u\eqdefa\grad_x^\perp \varphi, \quad \text{where} \ \varphi\eqdefa \lap_x^{-1} \om = \frac{1}{2\pi} \int_{\R^2} \log|x-y|  \om(y)\, dy.
\end{equation*}
 To investigate the long time evolution of  a pair of point vortices, we implement \eqref{eq 1.1} with
\begin{equation}\label{eq 1.2}
w(t=0)=\Ga_1 \delta_{P_1} + \Ga_2 \delta_{P_2}.
\end{equation}
Here  $\Ga_i\in \R$ and $P_i \in\R^2 (i=1,2)$, are the strength and initial position of the $i$-th Dirac measure, respectively. Provided the total circulation  $\Ga\eqdefa \Ga_1+\Ga_2 \neq 0$, we may apply a rotation and translation to assume, without loss of generality:
\begin{equation}\label{eq 1.3}
  \Ga_1\geq |\Ga_2|>0,\quad   P_i= (\ell_i,0) ,\quad     \Ga_1\ell_1 +\Ga_2\ell_2=0 \andf d\eqdefa\ell_1 -\ell_2 >0,
\end{equation}
which together with the conservation laws implies
\begin{equation}\label{eq 1.4}
    \int_{\R^2} w(t)dx =\Ga  \andf \int_{\R^2}x_jw(t)dx\equiv 0, \quad \text{for}\ j=1,2.
\end{equation}

Among the numerous parameters governing this complex system, three dimensionless quantities are paramount: the Reynolds number, the advection time, and the diffusion time, defined respectively by
\begin{equation}
   {\rm Re}\eqdefa \Ga_1/\nu   \gg1,\quad    T_{{\rm adv}} \eqdefa   d^2/\Ga_1, \quad T_{{\rm dif}}\eqdefa d^2/\nu=T_{{\rm adv}}\times {\rm Re}.
\end{equation}
Introducing the dimensionless variables
\begin{equation}\label{eq 1.4a}
    x=d\Tilde{x}, \quad \ell_i=d\Tilde{\ell}_i, \quad t=T_{{\rm adv}}\Tilde{t}, \quad w(t,x)= T_{{\rm adv}}^{-1}\Tilde{w}(\Tilde{t},\Tilde{x}),\quad  \Ga_2 = \Ga_1 \Tilde{\Ga}_2 \andf \nu= \Ga_1 \Tilde{\nu},
\end{equation}
we obtain the rescaled system:
\begin{equation*}
\left\{
\begin{aligned}
    &\p_{\Tilde{t}} \Tilde{w} + \Tilde{u}\cdot \grad_{\Tilde{x}} \Tilde{w}-\Tilde{\nu} \lap_{\Tilde{x}} \Tilde{w}=0,\quad \Tilde{u}=\grad^\perp_{\Tilde{x}}\lap_{\Tilde{x}}^{-1}\Tilde{w},\quad  (\Tilde{t},\Tilde{x}) \in \R^+\times \R^2, \\
    &\Tilde{w}(\Tilde{t}=0)=\delta_{(\Tilde{\ell}_1,0)} + \Tilde{\Ga}_2 \delta_{(\Tilde{\ell}_2,0)}.
\end{aligned}
\right.
\end{equation*}
From \eqref{eq 1.3}  and \eqref{eq 1.4a}, we deduce:
\begin{equation*}
    |\Tilde{\Ga}_2| \leq 1, \quad \Tilde{d}\eqdefa \Tilde{\ell}_1-\Tilde{\ell}_2 =1.
\end{equation*}
Hence, unless stated otherwise, we assume the normalized parameter regime:
\begin{equation}\label{eq assumption}
    \Ga_1=1, \quad |\Ga_2|\leq 1, \quad d=1.
\end{equation}

For the reader's convenience, and to contextualize our work, we first briefly survey prior results on  well-posedness, asymptotic behavior, enhanced dissipation, long time evolution  of \eqref{eq 1.1} with measure data.

\subsection{Well-posedness of \eqref{eq 1.1} and enhanced dissipation of Oseen vortices}
The well-posedness and uniqueness of \eqref{eq 1.1} with initial data in $\mathcal{M}(\R^2)$ were established in a series of literatures, see \cite{GMO88,GG05,GGL05, BM14}. Notably, the exact solution with $w(t=0)=\Ga \delta_0$ is a self-similar solution and  known as the ``Oseen vortice":
\begin{equation}
    w(t,x) = \frac{\Ga}{\nu t}G\Bigl(\frac{x}{\sqrt{\nu t}}\Bigr) \quad \text{with} \quad  G(\xi)\eqdefa \frac{1}{4\pi} e^{-|\xi|^2/4} ,
\end{equation}
which is widely used to analyze phenomena like tornadoes, industrial swirling flows, two-dimensional turbulence, and many other   applications.

What particularly distinguishes the Oseen vortex is the fact that any solution of \eqref{eq 1.1} with measure initial data (due to the smoothing effect of \eqref{eq 1.1}, it suffices to consider initial data in $L^1(\R^2)$) will eventually converge to  Oseen vortice of some strength.

\begin{Theorem}[\cite{GW05}]
   {\sl For any $w_0\in L^1(\R^2)$, the solution $w(t,x)$ of \eqref{eq 1.1} satisfies
\begin{equation*}
    \lim_{t\rto +\oo} t^{1-\frac{1}{p}}\Bigl\|w(t,x)- \frac{\Ga}{\nu t}G\Bigl(\frac{x}{\sqrt{\nu t}}\Bigr) \Bigr\|_{L^p_x} =0 \quad  \text{for } \ 1\leq p\leq +\oo,
\end{equation*}
where $\Ga\eqdefa \int_{\R^2} w_0(x) dx$.}
\end{Theorem}

Moreover, motivated by the self-similar structure of Oseen vortice,
we recall the self-similar variables:
\begin{equation}
    \xi\eqdefa \frac{x}{\sqrt{\nu t}},\quad \tau \eqdefa \log t \andf w(t,x)\eqdefa \frac{1}{\nu t} H(\tau,\xi),
\end{equation}
and then find that  $H$ solves:
\begin{equation}\label{eq 1.13}
    \p_\tau H + \nu^{-1}V\cdot \grad_\xi H - \cL H =0, \quad V(\xi)\eqdefa \grad^\perp_\xi \lap^{-1}_\xi H , \quad (\tau,\xi)\in \R\times\R^2,
\end{equation}
where $\cL\eqdefa\lap_\xi +\frac{\xi}{2}\cdot\grad_\xi+1$.  In self-similar coordinates, Oseen vortice becomes a stationary solution to \eqref{eq 1.13}.
The classical framework for studying Oseen vortices is the following weighted $L_\xi^2$ space:
\begin{equation*}
    \cY\eqdefa L^2(\R^2; G^{-1}d\xi), \quad \text{with}\quad  \lla f,g\rra_\cY\eqdefa \int_{\R^2} f\Bar{g} G^{-1}\, d\xi.
\end{equation*}
 Let $\xi=\rho(\cos\vartheta, \sin\vartheta)$, by  expanding $f\in \cY$ into  Fourier series with respect to $\vartheta$ variable, we obtain an orthogonal decomposition:
\begin{equation*}
    \cY= \bigoplus_{n=0}^\oo \cY_n,  \quad \text{with} \ \cY_n \eqdefa \bigl\{ f\in \cY \ :\ f=a(\rho) \cos(n\vartheta) + b(\rho) \sin(n\vartheta)  \bigr\}.
\end{equation*}
and denote by $\cP_n$ the orthogonal projection onto $\cY_n$.

\begin{Theorem}[\cite{GW05}]
    For any $H_0\in \cY$, \eqref{eq 1.13} has a unique solution  $H(\tau,\xi)\in C([0,+\oo),\cY)$ so that $H(0,\xi)=H_0$. And we have
    \begin{equation}
\lim_{\tau\rto +\oo}\|H(\tau,\xi) -\Ga G(\xi)\|_{\cY} = 0 , \quad \text{with}\quad  \Ga\eqdefa \int_{\R^2} H_0(\xi) d\xi.
    \end{equation}
\end{Theorem}

Nonlinear stability and decay rates were further studied in \cite{G12,GR08}:

\begin{Theorem}[\cite{G12,GR08}]
{\sl There exists $\epsilon>0$ so that, for all $\Ga\in \R$ and $H_0\in \cY$ satisfying
\begin{equation*}
 \int_{\R^2} H_0(\xi) d\xi =\Ga \andf \|H_0-\Ga G\|_{\cY} \leq \epsilon,
\end{equation*}
the unique solution $H(\tau,\xi)$ of \eqref{eq 1.13} with $H(0,\xi)=H_0$ verifies
\begin{equation}\label{eq 1.16}
   \|H(\tau,\xi) -\Ga G(\xi)\|_{\cY} \leq \min(1,e^{-\tau/2})  \|H_0-\Ga G\|_{\cY}, \quad \text{for} \ \tau \geq 0.
\end{equation}}
\end{Theorem}

The above asymptotic behavior stems from the spectrum of self-adjoint operator $\cL$ on $\cY$. Except for this mechanism, Oseen vortice also exhibits enhanced dissipation,
which was first discovered by Kelvin \cite{K87}, Orr \cite{O07} and then mathematically justified by a series of modern papers (see for example \cite{V09,BMV16,W21}), and refers to the accelerated decay of fluid disturbances due to the interplay between advection and diffusion in high-Reynolds-number flows. Unlike pure molecular diffusion, which happens on slow timescales, enhanced dissipation appears when shear flows or vortical structures stretch and deform fluid elements, effectively amplifying the gradient of passive scalars  and accelerating their homogenization. Mathematically, enhanced dissipation manifests as a dissipation timescale shorter than the purely diffusive prediction. The results in \cite{LWZ20,G18}  indicated  that the non-radial part of the perturbation near $\Ga G$ will  decay after $\tau= \cO(|\Ga/\nu|^{-\f13})$, which is much faster than \eqref{eq 1.16}.

 \begin{Theorem}
 {\sl There exists $C_1, C_2,\ka>0$ so that for all $\Ga\in \R$ and $ H_0 = \Ga G + \Tilde{H}_0 $ satisfying
\begin{equation*}
\int_{\R^2} \Tilde{H}_0 d\xi= \int_{\R^2} \xi_i \Tilde{H}_0 d\xi=0, \ i=1,2, \andf  \|\Tilde{H}_0\|_\cY \leq C_1  \frac{\nu (1+|\Ga/\nu|)^\f16}{\log(2+ |\Ga/\nu|)},
\end{equation*}
the unique solution $H(\tau,\xi)$ of \eqref{eq 1.13} with $H(0,\xi)= H_0$ verifies
\begin{align*}
    &\|H(\tau, \xi)- \Ga G\|_\cY \leq C_2 e^{-\tau} \|\Tilde{H}_0\|_\cY, \\
    &\|(1-\cP_0) (H(\tau, \xi)- \Ga G)\|_\cY \leq C_2 \|\Tilde{H}_0\|_\cY \exp\Bigl(- \frac{\ka (1+|\Ga/\nu|)^\f13}{\log(2+ |\Ga/\nu|)}\tau \Bigr) .
\end{align*}}
 \end{Theorem}

\subsection{Point vortex system}\label{subsec 1.3}
Setting $\nu=0$ in \eqref{eq 1.1}, we obtain the classical 2D incompressible Euler equations:
\begin{equation}\label{eq Euler}
    \pt w + u\cdot \grad_x w=0,\quad u=\grad^\perp_x\lap_x^{-1}w,\quad  (t,x) \in \R^+\times \R^2.
\end{equation}
A natural conjecture is that for small $\nu$,  the dynamical behavior of the solutions to \eqref{eq 1.1} will sufficiently approximate that of the Euler equations.
However, unlike  the case of 2D Naveier-Stokes equations, the measure initial data
\begin{equation}\label{eq 1.11}
    w_0= \sum_{i=1}^N \Ga_i \delta(\cdot - p_i)
\end{equation}
is too singular to construct  even a weak solution of Euler equations  \eqref{eq Euler} (see \cite{DM87,D91}).
To study such evolutions rigorously, three approximation approaches are commonly employed (see \cite{MB02} for instance):
\begin{itemize}
\item using vortex blob method to construct the  approximate  solutions $\{w^\epsilon\}_{0<\epsilon<1}$,

\item  regularizing initial data $w^\epsilon(t=0) \eqdefa \varphi_\ve \ast w_0 $ and solving \eqref{eq Euler} for the  approximate  solutions $\{w^\epsilon\}_{0<\epsilon<1}$,

    \item  solving \eqref{eq 1.1} for a sequence of  approximate  solutions $\{w^\nu\}_{0<\nu<1}$.
\end{itemize}

Along these lines, the inviscid motion of point vortices \eqref{eq 1.11} in the plane has been investigated in many literatures, initiated by Helmholtz \cite{H58} and Kirchhoff \cite{K76}, who derived the following formal solution of \eqref{eq Euler}:
\begin{equation}\label{eq 1.18}
    w(t,x) = \sum_{i=1}^N \Ga_i \delta(x-z_i(t)), \quad  \text{where} \
    z_i\  \text{solving \eqref{eq 1.1a} with }\  z_i(0)=p_i.
\end{equation}
For relevant literatures, readers may refer to \cite{M88, MP93}. In this paper we pursue the third approach to mathematically justify \eqref{eq 1.18}, in a relatively long region $t\ll \nu^{-1}$. Over such extended time scales, the primary imperative is to ensure the global well-posedness of the equations.
We list below some known results:
\begin{itemize}
    \item  for $N=2$ or same-sign circulations,  \eqref{eq 1.1a} is globally well-posed;
    \item the set of initial data leading to finite-time collisions has Lebesgue measure zero;
    \item for $N\geq 3$ with mixed signs, finite-time collisions can occur \cite{MP94,N01}.
\end{itemize}
Even if the solution of general $N$-body ($N\geq 3$) point vortex system is global, concise descriptions of the trajectories of $z_i(t)$ remain elusive.

\subsection{Previous studies on the long time evolution of Dirac measures}\label{subsec 1.2}
The case $N=2$ and $\Ga\eqdefa \Ga_1+\Ga_2=0$ has already been studied lately by Dolce and Gallay in \cite{DG24}. More precisely, their result can be stated as follows:

\begin{Theorem}[\cite{DG24}]
 {\sl   Let $s\in(0,1)$, $\ve\eqdefa\sqrt{\nu t}/d\ll1$, $\delta\eqdefa\nu /\Ga \ll1$ and let $w$ be the unique solution of \eqref{eq 1.1} with initial data $w_0= \Ga(\delta_{(-d/2,0)}- \delta_{(d/2,0)})$.
There exists constant $C$ independent of $\nu$, so that
\begin{equation*}
    \int_{\R^2} \Bigl| w(t,x)-\frac{\Ga}{4\pi\nu t} \Bigl(\exp\Bigl(-\frac{|x-x_\ell(t)|^2}{4\nu t}\Bigr) -\exp\Bigl(-\frac{|x-x_r(t)|^2}{4\nu t} \Bigl)\Bigr)\Bigr| \,dx \leq C \frac{|\Ga|}{d^2} \nu t, \quad \forall\ t \leq \frac{d^2}{|\Ga|} \left(\frac{|\Ga|}{\nu}\right)^{s},
\end{equation*}
where
\begin{equation}
    x_\ell(t)=\Bigl(-\frac{d}{2},Z(t)\Bigr),\quad x_r(t)=\Bigl(\frac{d}{2},Z(t)\Bigr),
\end{equation}
with $Z(0)=0$ and $Z'(t)= \frac{\Ga}{2\pi d}\left( 1-2\pi \al \ve^4 +  \cO(\ve^5+\delta^2 \ve + \delta \ve^2)\right) $ for some $\al \approx 22.24$.}
\end{Theorem}

Given the dynamical complexity of the N-body point vortex system, fine-grained studies on the long-term evolution of multiple point vortices under general conditions remain scarce. Currently, only short-time approximations such as the following are available:

\begin{Theorem}[\cite{G11}]
  {\sl  Let $w$ be the unique solution of \eqref{eq 1.1} with initial data
    $w_0= \sum_{i=1}^N \Ga_i \delta_{p_i}$.
We assume that the point vortex system \eqref{eq 1.18} with circulation $\Ga_i$ and initial position $p_i$ is well-posed on $[0,T]$, and we denote the solution by $\{z_i(t)\}_{1\leq i\leq N}$,
\begin{equation*}
    d\eqdefa \min_{t\in[0,T]}\min_{i\neq j} |z_i(t)-z_j(t)|, \qquad |\Ga|\eqdefa \sum_{i=1}^N|\Ga_i|>0.
\end{equation*}
Then there exists constant $C$, which depends on the normalized time $\frac{|\Ga|T}{d^2}$ and viscosity $\nu$,  so that
\begin{equation}
    \int_{\R^2} \Bigl|w(t,x)- \sum_{i=1}^N \frac{\Ga_i}{4\pi \nu t} e^{-\frac{|x-z_i(t)|^2}{4\nu t}}\Bigr|\, dx \leq C  \frac{|\Ga|}{d^2}\nu t, \quad \forall \ t \leq T.
\end{equation}}
\end{Theorem}

\subsection{Main results.}
Based on the known results presented in Subsections \ref{subsec 1.2} and \ref{subsec 1.3}, this paper focuses exclusively on the case where $N=2$ and $\Gamma\eqdefa \Ga_1+\Ga_2 \neq 0$, and obtains a complete characterization.
\begin{Theorem}\label{Thm 1.1}
  {\sl  Let $w$ solve \eqref{eq 1.1}-\eqref{eq 1.2} and $\Ga \neq0$. Then under the assumption \eqref{eq assumption}, there exists a real sequence of $\lbk\beta_k \rbk_{k\geq 2}$, so that for any $s\in(0,1)$,
    \begin{equation}\label{eq 1.5}
\int_{\R^2}\Bigl|  w(t,x)-    \sum_{i=1}^2\frac{\Ga_i}{4\pi \nu t} \exp\Bigl(-\frac{|x-x_i(t)|^2}{4\nu t}\Bigr)  \Bigr|\,dx \lesssim C_s\left(\Ga_2\right)\nu t , \quad \forall \  t \leq \nu^{-s},
    \end{equation}
    where
\begin{equation}\label{eq 1.8}
    x_i(t) \eqdefa \ell_i\left(\cos\theta(t)  ,\  \sin\theta(t)\right), \quad \text{for}\  i=1,2,
\end{equation}
with $\theta(0)=0$ and
\begin{equation}\label{eq 1.8a}
    \dot{\theta}(t) = \frac{\Ga}{2\pi}\Bigl( 1 +  \sum_{k=2}^{M_s} \beta_k \left(\nu t\right)^{k/2} \Bigr),  \quad  \text{ for some $M_s$ large enough and depending on $s$} ,
\end{equation}
where $C_s(z)$ is a positive function on $(-1,1]\backslash\{0\}$, which depends on $s$ and may tend to infinity as $z$ approaching $-1,0$.}
\end{Theorem}

\begin{Remark}
 Owing to the relation given in \eqref{eq 1.4a}, we are able to treat general parameter configurations without relying on assumption \eqref{eq assumption}, and thereby obtain the following:
   \begin{equation}\label{eq 1.9}
   \int_{\R^2}\Bigl| w(t,x)-    \sum_{i=1}^2\frac{\Ga_i}{4\pi \nu t} \exp\Bigl(-\frac{|x-x_i(t)|^2}{4\nu t}\Bigr)  \Bigr|\,dx\lesssim  C_s\left(\frac{\Ga_2}{\Ga_1}\right)   T_{{\rm adv}}^{-1} \nu t, \quad \forall \ t \leq  T_{{\rm adv}}\times {\rm Re}^s,
   \end{equation}
   where  $ x_i(t) \eqdefa \ell_i\left(\cos\theta(t)  ,\  \sin\theta(t)\right), i=1,2$,
with $\theta(0)=0$ and
\begin{equation}\label{eq 1.8b}
    \dot{\theta}(t) = \frac{\Ga}{2\pi d^2}\Bigl( 1 +  \sum_{k=2}^{M_s} \beta_k \Bigl(\frac{\sqrt{\nu t}}{d}\Bigr)^{k} \Bigr).
\end{equation}
\end{Remark}

\begin{Remark}
   In fact, the solution $
w(t,x)$ exhibits more refined localization properties than those indicated in \eqref{eq 1.5}. Specifically,  $w(t,x)$ remains highly concentrated near the trajectories $x_1(t)$ and $x_2(t)$, while away from these points it decays rapidly at the rate $\exp\Bigl(-\frac{|x-x_i(t)|^{2\ga}}{4(\nu t)^{\ga}}\Bigr)$ for some $\ga\ll1$.
\end{Remark}
\begin{Remark}
  The bound $\cO(\nu t)$ can be improved by incorporating higher-order corrections. And as noted in Remark \ref{Rmk 4.4}, we have $\beta_2=\beta_3=0$, $\beta_4\neq 0$. Integrating \eqref{eq 1.8b} with respect to $t$ then yields
 \begin{equation}\label{eq 1.24}
        \theta(t)= \frac{\Ga}{2\pi d^2}\Big(t+ \frac{\beta_4}{3d^4}\nu^2t^3+ \sum_{k=5}^{M_s} \frac{2\beta_k}{(k+2)d^k}\nu^{\frac{k}{2}}t^{\frac{k}{2}+1}\Big),
    \end{equation}
while the solution of the point vortex system is given by
\begin{equation}\label{eq 1.25a}
    \begin{split}
z_{i}(t)= \ell_i (\cos\theta_{{\rm p}} (t), \sin\theta_{{\rm p}} (t)), \quad \text{with }\ \dot{\theta}_{{\rm p}}(t)=\frac{\Ga}{2\pi d^2}, \quad  \theta_{{\rm p}} (t) = \frac{\Ga}{2\pi d^2} t.
    \end{split}
\end{equation}
By comparing \eqref{eq 1.8b}, \eqref{eq 1.24}, and \eqref{eq 1.25a}, we conclude that although the angular velocities differ only slightly, the accumulated phase shift becomes significant over the long time scale  $t=\cO( \nu^{-2/3})$.
\end{Remark}

\no {\bf Remarks on the constant $C_s(\Ga_2)$ in \eqref{eq 1.5}:}
    There are two interpretations for the blow-up of $C_s(\Ga_2)$ as $\Ga_2 \rto 0$. The first is related to the enhanced dissipation mechanism. When $\Ga_2\ll1$, the second point vortex can be treated as a small perturbation to the first. According to the results of Gallay et al., such a vortex undergoes nonlinear enhanced dissipation well before the diffusion time $t=\cO(\nu^{-1})$. As a result, the dynamical behavior of the solution in this regime can no longer be accurately captured by the point vortex system \eqref{eq 1.18}, which leads to the blow-up of the constant. The second interpretation arises from the sequence of approximate solutions $\Omega_{a}$ constructed in Section \ref{sect3}. In particular, we have $\Omega^{NS}_{2,2} = \cO (1/\Ga_2)$, which also implies singular behavior as $\Ga_2 \rto 0$.

Regarding the blow-up of $C_s(\Ga_2)$ as $\Ga_2\rto-1$,  the key point is that the case $\Ga=0$ is entirely different. In this scenario, one cannot use a rigid transformation to set the initial momentum to zero, i.e., to achieve $\Ga_1\ell_1+\Ga_2\ell_2=0$. Instead, the two point vortices undergo translational motion while maintaining an approximately fixed separation. Consequently, the rotational motion described in \eqref{eq 1.8} fails to accurately represent the long-time dynamics, leading to the blow-up of $C_s(\Ga_2)$ as $\Ga_2\rto -1$, i.e., $\Ga=0$. In our proof, the condition $\Ga\neq 0$ is used to ensure that the four ``pseudo-momenta" constructed in Section \ref{sect5} are linearly independent at leading order.

Despite the difference between the physical scenarios of cases $\Ga=0$ and $\Ga\neq 0$, they still share some similarity, since the motion of the case $\Ga=0$ can be viewed as a  rotation  around infinity. In particular, the leading-order linear velocities of the two point vortices are given by
\begin{equation}
    v_1 \eqdefa  \ell_1 \times \dot{\theta}(t) = \frac{\Ga_2}{2\pi d} + \cO(\ve^2)
     \andf v_2\eqdefa \ell_2 \times \dot\theta(t) = -\frac{\Ga_1}{2\pi d} + \cO(\ve^2) ,
\end{equation}
which resembles the result for the case $\Ga=0$. For higher-order corrections, we refer to Remark \ref{Rmk 4.4}.

\subsection{Sketch of the proof and new ingredients.} The proof of Theorem \ref{Thm 1.1} proceeds in several steps:

\no {\bf Step 1. The derivation of the resacled equations.}
By virtue of the fundamental solution of transport-diffusion equations, we decompose the solution of \eqref{eq 1.1}-\eqref{eq 1.2} into two parts $w=w_1+w_2$ (see \eqref{eq 2.1}), which interplay with each other and have fast decay away from two moving centroids.  Furthermore,  taking advantage of the self-similar structure, we introduce co-moving frame \eqref{eq 2.2} and  self-similar coordinates with different centroid \eqref{eq 2.2a}, to get a vector-valued equation (see \eqref{eq 2.6a}, \eqref{eq 2.7}):
\begin{equation}\label{eq 1.27}
    (t\p_t -\cL) \Omega + \frac{1}{\nu} \Bigl\{ \cB\Omega  +\frac{\ve^2}{2}\dot{\theta}|\xi|^2 Y_1 + \frac{\ve}{\Ga}(\dot{\theta}\al \xi_1 + \dot{\al}\xi_2) Y_2, \Omega \Bigr\}_V =0,
\end{equation}
where we use $\dot{f}$ to denote $\p_t f$ and  $(\Omega , \theta,\al)$ are unknowns.

\no {\bf Step 2. Construction of approximate solutions.}
 Using  finite power series expansion in small parameters $\ve\eqdefa\sqrt{\nu t}$ and $\nu$, we then construct  approximate solutions of \eqref{eq 1.27} by
\begin{equation*}
    \Omega_{a} = \sum_{k=0}^M \ve^k \bigl(\Omega^E_{k}+\nu \Omega^{NS}_{k}\bigr), \quad \dot{\theta}_a= \sum_{k=0}^{M-1} \ve^k \bigl(\dot{\theta}^E_k + \nu \dot{\theta}^{NS}_k\bigr), \quad \dot{\al}_a = \sum_{k=0}^{M-1}  \nu \ve^k \dot{\al}^{NS}_k,
\end{equation*}
so that the residual
\begin{equation*}
\begin{split}
  {\rm R}_a&\eqdefa  (t\p_t -\cL) \Omega_a + \frac{1}{\nu} \Bigl\{ \cB_a\Omega_a  +\frac{\ve^2}{2}\dot{\theta}_a|\xi|^2 Y_1 + \frac{\ve}{\Ga}(\dot{\theta}_a\al_a \xi_1 + \dot{\al}_a\xi_2) Y_2, \Omega_a \Bigr\}_V \\
  &=\cO(\nu^{-1}\ve^{M+1}+\nu \ve^2).
\end{split}
\end{equation*}
With these approximate solutions, we set
\begin{equation*}
    \Omega = \Omega_{a} +\om, \quad \dot{\theta}=\dot{\theta}_a + \dot{\theta}_p \andf
    \dot{\al}=\dot{\al}_a
\end{equation*}
and derive the linearized equation
\begin{equation}\label{eq 1.28}
    (t\p_t-\cL)\om + \frac{1}{\nu} \Lam^E\om  + \frac{1}{\nu} \frac{\ve }{\Ga} \al_a \thpp \{\hat{X},\Omega^E_a\}_V + \text{l.s.t.}= -{\rm R}_a + \text{n.l.t.},
\end{equation}
where $\hat{X}\eqdefa\xi_1 Y_2+ \frac{\ve \Ga}{2\ala} |\xi|^2 Y_1$ is called ``frame stream function", and  ``l.s.t." and ``n.l.t." designate ``lower size term" and ``nonlinear term", respectively.
For detailed definitions, one may check \eqref{eq 5.4}-\eqref{eq 5.5} below. Here we do not use the modulation parameter $\dot{\al}$, which might appear to reduce the available freedom. But {\bf Step 5} will show that this freedom is in fact unnecessary.

\no {\bf Step 3. Functional relationship and weighted spaces.}
Following \cite{GS24-1, DG24}, we denote
\begin{equation*}
    \Psi^{E}_a \eqdefa \cB_a \Omega^{E}_a + \frac{\ve^2}{2} \dot{\theta}^{E}_a |\xi|^2 Y_1 +\frac{\ve }{\Ga}\dot{\theta}^{E}_a \al_a \xi_1 Y_2-\frac{1}{2\pi} \log|\ve /\al_a| (\Ga_2,\Ga_1)^T \quad \text{with} \quad  \Omega^E_a \eqdefa \sum_{k=0}^M \ve^k \Omega^E_{k},
\end{equation*}
and then recover the functional relationship between $\Psi^E_a$ and $\Omega^E_a$, i.e., there exist $F_i$ so that
$\Psi^E_{i,a} \approx F_i(\Omega^E_{i,a})$. Inspired by Arnold's Method \cite{A66} and works in \cite{GS24-2}, we define
\begin{equation*}
    W_{i,\ve}(\xi) \approx F'_i\left(\Omega^E_{i,a}(\xi)\right), \quad \|f\|_{\cX_{i,\ve}}^2\eqdefa \int_{\R^2} |f(\xi)|^2 W_{i,\ve}(\xi) d\xi, \quad E_\ve[\om] \eqdefa\frac{1}{2}\lla \om ,\om \dot\otimes W_\ve  + \cB_a \om\rra_V.
\end{equation*}
Here for the detailed definitions, see \eqref{eq 5.4}-\eqref{eq 5.4a}, \eqref{eq 5.7}-\eqref{eq 5.10}. With this energy functional,
\begin{equation*}
    \Lam^E \om= \lbk\Psi^{E}_a, \om \rbk_V + \lbk\cB_a \om, \Omega^{E}_a\rbk_V  \approx \lbk \om \dot\otimes W_\ve + \cB_a \om, \ \Omega^{E}_a\rbk_V ,
\end{equation*}
which leads to
\begin{equation*}
    \lla \Lam^E \om, \ \om \dot\otimes W_\ve + \cB_a \om  \rra_V \approx 0 ,\quad \text{i.e., $\Lam^E$ is almost skew-symmetric with respect to $E_\ve$. }
\end{equation*}

\no {\bf Step 4. Coercivity of $E_\ve[\om]$ and ``pseudo-momenta''.}
This step is crucial and differs from previous approaches. To restore the coercivity of $E_\ve[\om]$, earlier works exploited the freedom in modulation parameters (here, $\dot{\theta}$ and $\dot{\al}$) to enforce the momentum orthogonality condition
\begin{equation*}
    \int_{\R^2} \xi_j  \om d\xi =0, \quad j=1,2,
\end{equation*}
as  observed in \cite{GS24-2}.
However, in this paper we demonstrate that this condition is incompatible with the operator $\Lam^E$ and allows small perturbations. Instead, it can be replaced by
\begin{equation*}
    \int_{\R^2}(\xi_1+\ve P_1) \om \, d\xi=\int_{\R^2}(\xi_2+\ve P_2) \om \, d\xi  =0 , \quad \text{for some $P_1,P_2\in \cS_\star$},
\end{equation*}
while still preserving the coercivity of $E_\ve[\om]$. The key innovation of this paper lies in constructing suitable functions and-referred to as ``pseudo-momenta"-that are compatible with the linearized equation. This construction is detailed in Subsection \ref{subsec. coercivity}.

To better illustrate the concept of ``pseudo-momenta'', we consider a toy model:
\begin{equation*}
    (t\p_t -\cL )f +\nu^{-1}\Lam f={\rm R}, \quad \text{with} \quad \Lam f\eqdefa\{\lap^{-1}_\xi G, f\}+ \{\lap_\xi^{-1}f,G\} .
\end{equation*}
We denote the adjoint of $\Lam$ with respect to standard $L^2_\xi$ inner product by $\Lam^\star$. Then it's easy to check that
\begin{equation}\label{eq 1.29}
    \Lam[\p_j G]=0,  \quad \Lam^\star[\xi_j]=0,\quad   \{\xi_1, G\} = \p_2 G, \quad  \{\xi_2, G\}= -\p_1 G.
\end{equation}
Motivated by \eqref{eq 1.29}, ``pseudo-momenta" should be the eigenfunctions, or some elements in invariant spaces of $\Lam^{E,\star}$, up to an error $\cO(\ve^{M+1})$, whose leading order is $S\xi$ for some non-constant matrix $S$. Similar idea-modulating the kernel of the linearized operator to restore coercivity-has been used in  constructing blow-up solutions and stability analysis.

Let $\xi=\rho(\cos\vartheta, \sin\vartheta)$, we   recall  the ``frame stream function" $\hat{X}$ and define the ``frame derivative" $X$ by
\begin{equation*}
    \hat{X}\eqdefa \xi_1 Y_2+ \frac{\ve \Ga}{2\ala} |\xi|^2 Y_1 \andf X\Omega \eqdefa \{\hat{X},\Omega\}_V = \Bigl(Y_2 \p_2 + \frac{\ve\Ga}{\ala}Y_1 \p_\vartheta\Bigr)\Omega.
\end{equation*}
A key observation is that, up to an small error,  $X\Omega^E_a$ is an eigenfunction of $\Lam^E$ with eigenvalue zero.  More precisely, there exist four ``pseudo-momenta":
\begin{align*}
    &\Lam^{E,\star} [\xi_1 Y_1] = \ve^2  \dot{\theta}^E_a \xi_2 Y_1 , \quad \Lam^{E,\star} [\xi_2 Y_1] = -\ve^2  \dot{\theta}^E_a \xi_1 Y_1-\frac{\ve }{\Ga}\dot{\theta}^E_a \ala Y_2,\\
    &\Lam^{E,\star}[\varrho^e_a] \approx 0,\quad   \Lam^{E,\star}\varrho^o_a \approx \lam^e_a \varrho^e_a, \\ &\text{with}\quad \lam^e_a =\cO(\ve^2), \quad  \varrho^o_a = \xi_2 Y_2 +\cO(\ve) ,\quad \varrho^e_a = \xi_1 Y_2 +\cO(\ve).
\end{align*}
And enlighten by \eqref{eq 1.29}, we use the Lie structure of 2D Euler equations to  introduce
\begin{align*}
    & -\p_1 \Omega^E_a\eqdefa \{\xi_2 Y_1 , \Omega^E_a\}_V ,\quad \p_2\Omega^E_a\eqdefa\{\xi_1 Y_1 , \Omega^E_a\}_V, \\
    &  f^{o}_a\eqdefa \{\varrho^o_a , \Omega^E_a\}_V, \quad f^e_a\eqdefa \{\varrho^e_a, \Omega^E_a\}_V \approx  \{\hat{X}, \Omega^E_a\}_V,
\end{align*}
which also turn out to be the eigenfunctions, or some elements in invariant spaces of $\Lam^{E}$.
Moreover, $\varrho^o_a, f^e_a$ are $\xi_2$-odd, $ \varrho^e_a, f^o_a$ are $\xi_2$-even and $\p_1\Omega^E_a , \p_2\Omega^E_a, f^e_a,f^o_a$ enjoy  certain specific inner product structures (see \eqref{eq 6.19}).

\no {\bf Step 5. Choice of modulation parameter $\dot{\theta}_p$.}
Thanks to \eqref{eq 1.4}, $\om$ is always orthogonal to $\xi_1 Y_1$ and $\xi_2Y_1$. Thus it's natural to decompose $\om$ into
\begin{equation*}
    \begin{split}
\om \approx \mu^e f^e_a + \mu^o f^o_a + \om^R,  \quad \text{with} \quad \lla\om^R, Z\rra_V=0 \quad \text{for}\  Z =\xi_1 Y_1 , \xi_2 Y_1,\varrho^{o}_a, \varrho^{e}_a.
    \end{split}
\end{equation*}
Inserting the above decomposition into \eqref{eq 1.28}, we find
\begin{align*}
    &(t\p_t+1/2 )\mu^o \times f^o_a +   (t\p_t+1/2 )\mu^e \times f^e_a  + (t\p_t-\cL +\nu^{-1}\Lam^E ) \om^R \\
    &\quad +\nu^{-1}\left( \frac{\ve }{\Ga}\ala \dot{\theta}_p - \lam^e_a \mu^o  \right)f^e_a  = -{\rm R}_a+ \textrm{l.s.t. + n.l.t.}.
\end{align*}
Thus it suffices to  choose $\frac{\ve }{\Ga}\ala \dot{\theta}_p - \lam^e_a \mu^o=0$, to eliminate the growth of solutions induced by the large factor $\nu^{-1}$, which completely stems from the structure of the equations and the linearized operator $\Lam^E$, as well as a simple and universal perspective of linear algebra.

Finally by testing  above  equality by  $\varrho^o_a,\varrho^e_a, \om^R \dot\otimes W_\ve + \cB_a \om^R$ to get the estimate of $\mu^e,\mu^o,\om^R$, respectively, we can bound $\om$ by ${\rm R}_a$ and thus complete the proof of Theorem \ref{Thm 1.1}.

\subsection{Notations:}
We denote $\lla f,g\rra$ as the  $L_\xi^2(\R^2)$ inner product. For  $\xi=\rho(\cos\vartheta, \sin\vartheta)$, the Poisson bracket is:
$$\{f,g\}\eqdefa \nabla^\perp_\xi f\cdot\nabla_\xi g= \p_{\xi_1}f \p_{\xi_2}g-\p_{\xi_1}g \p_{\xi_2}f =\frac{1}{\rho}(\p_\rho f\p_\vartheta g- \p_\vartheta f \p_\rho g).$$ For 2D vector function $\Omega$, we denote its $i$-th component by $\Omega_i$ and $\Omega \dot\otimes \varrho \eqdefa (\Omega_1 \varrho_1, \Omega_2\varrho_2)^T$,
\begin{equation*}
    \lla \Omega, \varrho \rra_V \eqdefa \int_{\R^2} \bigl(\Omega_1 \varrho_1 + \Omega_2 \varrho_2\bigr) \, d\xi, \quad \lbk \Omega, \varrho \rbk_V
     \eqdefa \bigl(
        \{\Omega_1, \varrho_1\}, \lbk \Omega_2 ,\varrho_2\rbk \bigr)^T.
\end{equation*}

We denote the dual space of $\cY$ by
$\cY^\star\eqdefa L^2(\R^2; G d\xi) = G^{-1}\cY $, and  the multiplier space of  $\cS(\R^2)$ and $\cS'(\R^2)$ by  $\cS_\star(\R^2)$,  which consists of smooth functions $f(\xi)$ so that for any $\beta= (\beta_1,\beta_2) \in \N^2$,
\begin{equation*}
    |\p^\beta f| \leq C_{\beta,N} (1+|\xi|)^N , \quad \text{for some $N\in\N$ and constant $C_{\beta,N}$ .}
\end{equation*}
We  define a dense subspace $\cZ\in\cY$ by
\begin{equation*}
    \cZ= G \cS_\star\eqdefa \bigl\{ f\in \cY\ : \ f= G h \  \text{for some}\ h\in \cS_\star  \bigr\},
\end{equation*}
and  class $\cK$ by
\begin{equation*}
    \cK\eqdefa\bigl\{F\in \cC^\oo(\R^+)\ :\  F(G) \in \cS_\star(\R^2)\bigr\}.
\end{equation*}

Let $f \in \cS_\star(\R^2)$ depend on $0< \ve\ll1$, we denote
\begin{equation*}
     \Xi_{k} [f] \eqdefa \frac{1}{k!}\frac{d^k f}{d\ve^k}\bigg|_{\ve=0} \andf \Pi_M[f]\eqdefa \sum_{k=0}^M \ve^k \Xi_k[f] =  \sum_{k=0}^M  \frac{\ve^k}{k!} \frac{d^k f}{d\ve^k}\bigg|_{\ve=0}.
\end{equation*}
We say $f = \cO_{\cS_\star}\left(\nu^{M_1} \ve^{M_2}\right)$ if
\begin{equation*}
    \forall \beta \in \N^2, \exists C>0, N\in \N, \quad \text{so that} \ |\p^\beta f(\xi)| \leq C(1+|\xi|)^N \nu^{M_1} \ve^{M_2}, \quad \forall \xi \in \R^2.
\end{equation*}
Similarly, we write $f = \cO_\cZ\left(\nu^{M_1} \ve^{M_2}\right)$ if $G^{-1}f = \cO_{\cS_\star}\left(\nu^{M_1} \ve^{M_2}\right)$.

And we say $f = \cO_{\cZ}\left(\nu^{M_1}\ve^{M_2} + \nu^{M_3} \ve^{M_4}\right)$ if
$$\text{  $f = f_1+ f_2$ with $f_1=\cO_\cZ\left(\nu^{M_1}\ve^{M_2}\right)$ and $f_2= \cO_\cZ\left(\nu^{M_3} \ve^{M_4}\right)$.}   $$

\renewcommand{\theequation}{\thesection.\arabic{equation}}
\setcounter{equation}{0}
\section{Preliminaries}\label{sect2}

In this section, we
 shall first derive the rescaled vorticity equations and then discuss its basic conservation laws. Finally we present some basic operators, which will be frequently used throughout this paper.

\subsection{The rescaled vorticity equations }
We first introduce $\{f,g\}_x\eqdefa \grad_x^\perp f \cdot \grad_x g$ and rewrite \eqref{eq 1.1} as
\begin{equation}
    \p_t w +\{\varphi,w\}_x-\nu \lap_x w =0, \quad  (t,x) \in \R^+\times \R^2.
\end{equation}
Then by the uniqueness result in \cite{GG05}, it's equivalent to consider  $w=w_1+w_2$ with $w_i$ satisfying
\begin{equation}\label{eq 2.1}
\left\{
\begin{aligned}
&\pt w_i + \{\varphi_1+\varphi_2,w_i\}_x-\nu \lap_x w_i=0, \\
&w_i(t=0)=\Ga_i \delta_{(\ell_i,0)},
\end{aligned}
\right.
\end{equation}
where $\varphi_i\eqdefa\lap_x^{-1}w_i $. Motivated by point vortex system and previous work \cite{DG24},
we seek the solution of \eqref{eq 2.1} of the form
\begin{equation}\label{eq 2.2}
    w_i(t,x)=\frac{1}{\nu t} \Omega_i\Bigl(t,\frac{\cR(\theta(t)) x - r_i(t) \eo }{\sqrt{\nu t}}\Bigr) \quad \text{with} \ \Ga_1 r_1 +\Ga_2 r_2=0,
\end{equation}
where we define $\cR(\theta):\R^2 \rightarrow \R^2$ via
\begin{equation*}
    \cR({\theta}) \begin{pmatrix}
        x_1\\x_2
    \end{pmatrix} \eqdefa \begin{pmatrix}
    \cos\theta ,&  -\sin\theta \\
    \sin\theta ,& \cos\theta
    \end{pmatrix} \begin{pmatrix}
        x_1\\x_2
    \end{pmatrix},
\end{equation*}
and thus
\begin{equation*}
    \frac{d}{dt} \big(\cR(\theta(t))x\big)=\dot{\theta}(t)\begin{pmatrix}
    -\sin\theta(t), & -\cos\theta(t)\\
    \cos\theta(t), &  -\sin\theta(t)
    \end{pmatrix} \begin{pmatrix}
        x_1\\x_2
    \end{pmatrix} = \dot{\theta}(t)\big(\cR(\theta(t))x\big)^{\perp} .
\end{equation*}
Here $\dot{\theta}(t),\dot{r}_i(t)$ designate the angular velocity and radius contraction speed  of the centroid of $w_i$, respectively,  with $\theta(0)=0, r_i(0)=\ell_i$. Via direct computation, we find
\begin{equation}
    \varphi_i(t,x)=\frac{1}{2\pi}\int_{\R^2} \log|x-y| w_i(y)\, dy=\Phi_i\Bigl(t,\frac{\cR(\theta(t)) x - r_i(t) \eo }{\sqrt{\nu t}}\Bigr) + \frac{\log(\nu t)}{4\pi}\int_{\R^2} \Omega(\xi) \, d\xi,
\end{equation}
where $\Phi_i \eqdefa \lap_\xi^{-1} \Omega_i$.

To proceed, we recall that $d=1,$ $\Ga=\Ga_1+\Ga_2,$ and introduce $i=1, \is=2,$ while $i=2, \is=1,$  $\ka_1=-\ka_2=1$,
\begin{equation}\label{eq 2.2a}
\begin{split}
    &\xi^{(i)} \eqdefa \frac{\cR(\theta(t)) x - r_i(t) \eo }{\sqrt{\nu t}},  \quad \al(t)\eqdefa r_1(t) - r_2(t), \quad  \ve=\sqrt{\nu t} \quad  \text{and thus}\   r_i(t) = \frac{\ka_i \Ga_\is}{\Ga} \al(t).
\end{split}
\end{equation}
Throughout this paper, we will use variable $\xi^{(i)}$ when considering the $\Omega_i$-equation, and  omit the superscript in $\xi^{(i)}$ in what follows, if no ambiguity arises.
Then we find
\begin{equation}\label{eq 2.3}
\begin{split}
&t\p_t \Omega_i  - \cL \Omega_i + \frac{1}{\nu} \Big\{ \Phi_i(\xi) +  \Phi_\is\left(\xi+\frac{\ka_i \al}{\ve}\eo\right)+ \frac{\ve^2}{2}\dot{\theta}|\xi|^2  + \frac{\ve \ka_i\Ga_\is}{\Ga} (\dot{\theta} \al \xi_1+\dot{\al}\xi_2), \  \Omega_i  \Big\} =0.
\end{split}
\end{equation}
According to \cite{GW05,GG05}, we have
\begin{equation}\label{eq 2.4}
    \Omega_i(t=0)=\Ga_i G.
\end{equation}

Furthermore, for any $2$D vector
$  \Omega=(\Omega_1, \Omega_2)^T$, we introduce:
\begin{equation}\label{eq 2.6a}
\begin{split}
  &\qquad \cB\Omega(\xi)\eqdefa \begin{pmatrix}\lap_\xi^{-1}\Omega_1(\xi) + \lap_\xi^{-1}\Omega_2\left(\xi+\ve^{-1}\al\eo\right)\\ \lap_\xi^{-1}\Omega_2(\xi) + \lap_\xi^{-1}\Omega_1\left(\xi-\ve^{-1}\al\eo\right) \end{pmatrix} , \quad  Y_1\eqdefa\begin{pmatrix}
      1\\1
  \end{pmatrix}, \quad Y_2 \eqdefa \begin{pmatrix}
      \Ga_2 \\ -\Ga_1
  \end{pmatrix}.
  \end{split}
\end{equation}
Equation \eqref{eq 2.3} can then be written in vector form as
\begin{equation}\label{eq 2.7}
    (t\p_t -\cL) \Omega + \frac{1}{\nu} \Bigl\{ \cB\Omega  +\frac{\ve^2}{2}\dot{\theta}|\xi|^2 Y_1 + \frac{\ve}{\Ga}(\dot{\theta}\al \xi_1 + \dot{\al}\xi_2) Y_2, \Omega \Bigr\}_V =0.
\end{equation}

\subsection{Basic conservation}
To begin with, we define
\begin{equation*}
    m[\Omega]\eqdefa\int_{\R^2} \Omega(\xi) d\xi \andf M[\Omega] \eqdefa \begin{pmatrix}
        M_1[\Omega] \\ M_2[\Omega]
    \end{pmatrix}=\int_{\R^2} \Omega(\xi) \xi d\xi.
\end{equation*}

\begin{Proposition}\label{Prop 3.1}
 {\sl Let $(\Omega, \dot{\al}, \dot{\theta}) $ solve \eqref{eq 2.7}.  Then there hold
    \begin{equation*}
        m[\Omega_i] = \Ga_i \andf M[\Omega_1]+M[\Omega_2]=0 \quad \text{for }\ i=1,2.
    \end{equation*}}
\end{Proposition}
\begin{proof}
Due to $\dive_x u=0$, we get, by integrating \eqref{eq 2.1} over $\R^2,$ that
\begin{equation*}
    \int_{\R^2}  w_i(t,x)\, dx = \Ga_i,
\end{equation*}
which is equivalent to $m[\Omega_i]=\Ga_i$ by \eqref{eq 2.2}.
And thanks to \eqref{eq 1.4}, we have
 \begin{equation*}\begin{split}
     M[\Omega_1(t)]+M[\Omega_2(t)]
     &=\sum_{i=1}^2\int_{\R^2}  \Omega_i\Bigl(\frac{\cR(\theta)x-r_i\eo}{\sqrt{\nu t}} \Bigr) \frac{\cR(\theta)x-r_i\eo}{\sqrt{\nu t}} \frac{dx}{\nu t}\\
   & =(\nu t)^{-\f12} \int_{\R^2} w(x) \cR(\theta)x dx -\sum_{i=1}^2   \frac{r_i\eo}{\sqrt{\nu t}} \int_{\R^2}  \Omega_i\left(\frac{\cR(\theta)x-r_i\eo}{\sqrt{\nu t}} \right)  \frac{dx}{\nu t} \\
   & = -(\nu t)^{-\f12}(\Ga_1 r_1 + \Ga_2 r_2)\eo =0.
\end{split}\end{equation*}
This completes the proof of Proposition \ref{Prop 3.1}
\end{proof}

\subsection{Basic operators}\label{subsec.basic operators}
We consider the following operators on $\cY$:
\begin{equation*}
    \cL f= \lap_\xi f+ \frac{\xi}{2}\cdot \grad_\xi f + f \quad \text{with} \ \cD(\cL)\eqdefa \bigl\{f\in \cY\ :\ \lap_\xi f\in \cY, \, \xi\cdot\grad_\xi f\in \cY \bigr\},
\end{equation*}
and
\begin{equation*}
    \Lam f \eqdefa \{\Upsilon, f\} + \{\lap_\xi^{-1}f, G\} \quad  \text{with} \ \cD(\Lam)\eqdefa \bigl\{f\in \cY \ : \ \{\Upsilon, f\}  \in \cY  \bigr\}.
\end{equation*}
where we denote
\begin{equation*}
    \begin{split}
    \Upsilon(\xi)&\eqdefa \lap_\xi^{-1}G     =\frac{1}{2\pi} \int_{\R^2} \log|\xi-\eta| G(\eta) d\eta  = \frac{1}{4\pi}\left(\textrm{Ein}(|\xi|^2/4) -\ga_E\right),
    \end{split}
\end{equation*}
with $ \textrm{Ein}(x)\eqdefa\int_0^x (1-e^{-t})/tdt$ and $\ga_E$ is the Euler constant.
We denote  the adjoint of $\cL$, $\Lam$ in $L^2_\xi$ by $\cL^\star,\Lam^\star$, respectively, i.e.,
\begin{equation*}
    \lla  f, \cL^{\star}g\rra \eqdefa \lla \cL f, g  \rra \andf  \lla  f, \Lam^{\star}g\rra \eqdefa \lla \Lam f, g  \rra,
\end{equation*}
from which it follows directly that
\begin{equation*}
    \cL^{\star} f = \lap_\xi f -\frac{\xi}{2}\cdot \grad_\xi f \quad \text{with}\ \cD(\cL^\star)\eqdefa \bigl\{f\in \cY^\star \ :\ \lap_\xi f, \  \xi\cdot\grad_\xi f \in \cY^\star  \bigr\},
\end{equation*}
and
\begin{equation*}
    \Lam^\star f =-\{\Upsilon , f\} - \lap_\xi^{-1}\{f,G\}   \quad \text{with} \ \cD(\Lam^\star)\eqdefa \bigl\{f\in \cY^\star \ :\ \{\Upsilon, f\} \in \cY^\star \bigr\}.
\end{equation*}

Below we list some fundamental properties of $\cL$ and $\Lam$.
\begin{Proposition}\label{Prop 3.2}
 {\sl   There hold
\begin{enumerate}
    \item[(1)]  $\cL$ is self-adjoint in $\cY$ with compact resolvent and purely discrete spectrum:
    \begin{equation*}
        \sig(\cL) = \lbk-\frac{n}{2} \ : \ n\in \N\rbk.
    \end{equation*}
    Moreover, $\ker\cL= \spn\{G\}$, $\ker (\f12+\cL) = \spn\{\p_1 G,\p_2G\}$. And the eigenvalue $-\frac{n}{2}$ has multiplicity $n+1$, with eigenfunctions being Hermite functions of degree $n$.

    \item[(2)] For any $\ka>0$ and $f\in \cZ$, one has $(\ka- \cL)^{-1}f \in \cZ$.

    \item[(3)] $\Lam$ is skew-adjoint in $\cY$ with
    \begin{equation*}
        \ker\Lam = \cY_0 \oplus \spn\bigl\{\p_1 G, \p_2 G\bigr\}.
    \end{equation*}

    \item[(4)] Let
    \begin{equation*}
        \cY'_1 \eqdefa \cY_1 \cap (\ker{\Lam})^\perp = \lbk f\in \cY_1 \ : \ \int_{\R^2} \xi_j f(\xi)\, d\xi=0 ,\  j=1,2   \rbk.
    \end{equation*}
    Then for $f\in \cY_n \cap \cZ$ if $n\geq 2$,  or $f\in \cY'_1\cap \cZ$ if $n=1$, there exists a unique $\Omega $ in the same space so that  $\Lam \Omega = f$.
     Furthermore, if $f=b(\rho) \sin(n\vartheta)$ (respectively, $f=b(\rho)\cos(n\vartheta)$), then $\Omega= w(\rho)\cos(n\vartheta)$ (respectively, $\Omega=-w(\rho)\sin(n\vartheta)$), where
\begin{equation}\label{eq 3.1}
    w(\rho)\eqdefa - \frac{\rho^2/4}{e^{\rho^2/4}-1} \varphi(\rho) - \frac{2\pi\rho^2 }{n(1-e^{-\rho^2/4})} b(\rho),
\end{equation}
with $\varphi$ solving
\begin{equation}\label{eq 3.2}
    -\varphi'' - \frac{1}{\rho} \varphi' +\left(\frac{n^2}{\rho^2} - \frac{\rho^2/4}{e^{\rho^2/4}-1} \right) \varphi = \frac{2\pi\rho^2 }{n(1-e^{-\rho^2/4})} b(\rho), \quad \rho>0.
\end{equation}
In particular, if $f$ is $\xi_2$-even (odd), then $\Omega$ is $\xi_2$-odd (even).

\item[(5)] For any $h\in G^{-1}\cY_n \cap \cS_\star$ if $n\geq 2$, or $h\in G^{-1}\cY'_1 \cap \cS_\star$ if $n=1$, there exists a unique $\varrho$ in the same space so that  $\Lam^\star \varrho = h$. In particular, if $h$ is $\xi_2$-even (odd), then $\varrho$ is $\xi_2$-odd (even).
\end{enumerate}}
\end{Proposition}
\begin{proof}
   For the assertions (1)–(4), see for instance \cite{GW02,DG24,GS24-1}. Using (3), we compute for any $f,g\in \mathcal{C}^\oo_c(\R^2)$:
\begin{equation*}
    \lla f, \Lam^\star g \rra = \lla \Lam f, g\rra = \lla \Lam f, Gg \rra_\cY = - \lla f, \Lam[G g]\rra_\cY = -\lla f, G^{-1}\Lam[Gg]\rra,
\end{equation*}
 which gives rise to
 \begin{equation*}
     \Lam^\star g = - G^{-1} \Lam[G g].
 \end{equation*}
In particular, we arrive at
\begin{equation*}
     \Lam^\star \varrho =  h \iff \Lam[G \varrho] = - Gh,
\end{equation*}
from which and (4), we  deduce (5) and thus complete the proof of Proposition \ref{Prop 3.2}.
% Notice that
%    \begin{equation*}
%        \lbk f, G\rbk = \frac{-1}{\rho} \p_\vartheta f \p_\rho G = \frac{1}{2}G \p_\vartheta f ,
%    \end{equation*}
% we have
% \begin{equation*}
%    \lap_\xi^{-1}\{f,G\} = \frac{1}{2} \lap_\xi^{-1}(G\p_\vartheta f) = \frac{1}{2} \p_\vartheta \lap_\xi^{-1}(G f)= G^{-1} \lbk\lap_\xi^{-1}(G f), G  \rbk.
% \end{equation*}
% Using the above equality and  $ \lbk \Upsilon , f\rbk = G^{-1} \lbk \Upsilon , Gf\rbk $,
% we achieve
% \begin{equation*}
%     \begin{split}
% \Lam^\star f = -G^{-1} \lbk \Upsilon , Gf\rbk - G^{-1} \lbk\lap_\xi^{-1}(G f), G  \rbk =-G^{-1}\Lam[G f],
%     \end{split}
% \end{equation*}
% which in particular implies
% \begin{equation*}
%     \Lam^\star \varrho =  h \iff \Lam[G \varrho] = - Gh.
% \end{equation*}
% from which and (4), we  deduce (5) of Proposition \ref{Prop 3.2},
\end{proof}

\subsection{Expansion of the interacting velocity}
Let $\xi=(\rho \cos\vartheta, \rho\sin\vartheta)$, we recall the   following n-homogeneous  polynomials in $\R^2$ from \cite{DG24}:
\begin{equation*}
    Q^c_n(\xi_1,\xi_2)\eqdefa {\rm Re}(\xi_1+i\xi_2)^n=\rho^n\cos(n\vartheta) \andf Q^s_n(\xi_1,\xi_2)\eqdefa {\rm Im}(\xi_1+i\xi_2)^n=\rho^n\sin(n\vartheta).
\end{equation*}

\begin{Lemma}\label{Lem 3.3}
{\sl There hold
\begin{equation}
    \p_1 Q^c_n(\xi_1,\xi_2)= nQ^c_{n-1}(\xi_1,\xi_2) \andf \p_2 Q^c_n(\xi_1,\xi_2)=-nQ^s_{n-1}(\xi_1,\xi_2).
\end{equation}}
\end{Lemma}
\begin{proof}
    Notice that
    $$\p_1 = \cos\vartheta \p_\rho -\sin\vartheta \frac{\p_\vartheta}{\rho},\quad \p_2 = \sin\vartheta \p_\rho +\cos\vartheta \frac{\p_\vartheta}{\rho}, $$
we get
\begin{equation*}
\begin{split}
    \p_1 Q^c_n &= \Bigl(\cos\vartheta \p_\rho -\sin\vartheta \frac{\p_\vartheta}{\rho}\Bigr) (\rho^n \cos (n\vartheta))=n \rho^{n-1} \left(\cos\vartheta \cos(n\vartheta) + \sin\vartheta \sin(n\vartheta)   \right)\\
    &=n \rho^{n-1} \cos((n-1)\vartheta)  = n Q^c_{n-1},\\
    \p_2 Q^c_n &= \Bigl(\sin\vartheta \p_\rho +\cos\vartheta \frac{\p_\vartheta}{\rho}\Bigr) (\rho^n \cos (n\vartheta)) = n \rho^{n-1} \left(\sin\vartheta \cos(n\vartheta) - \cos\vartheta \sin(n\vartheta)   \right)\\
    &=-n \rho^{n-1} \sin((n-1)\vartheta) = -n Q^s_{n-1},
\end{split}
\end{equation*}
which completes the proof of Lemma \ref{Lem 3.3}.
\end{proof}

\begin{Lemma}\label{Lem 3.4}
{\sl   Let $0 < |\lam| \ll1$,  $N\in \N$, $\Omega = \cO_\cZ(1)$ and $\Phi=\lap_\xi^{-1}\Omega$. There hold
\begin{equation}\begin{split}\label{eq 3.4}
   &\Phi(\xi+\lam^{-1} \eo) = \frac{1}{2\pi}\log|\lam|^{-1} \int_{\R^2}\Omega(\eta) d\eta \\
   &\quad +  \sum_{n=1}^N \frac{(-1)^{n-1}}{2\pi n} \lam^{n} \int_{\R^2}   Q^c_n\big( \xi-\eta\big) \Omega(\eta) d\eta + \cO_{\cS_\star}(\lam^{N+1}).
\end{split}
\end{equation}}
\end{Lemma}
\begin{proof}
First of all, in view of
    \begin{equation*}
        \Phi (\xi) = (\lap_\xi^{-1}\Omega) (\xi) = \frac{1}{2\pi} \int_{\R^2} \log|\xi-\eta| \Omega(\eta) d\eta,
    \end{equation*}
we have $\Phi\in \cS_\star$ and thus $\Phi (\xi+\lam^{-1}\eo) \in \cS_\star$. Notice that any $f\in\cS_\star$ supported on $|\xi|\geq  \frac{1}{8|\lam|}$ belongs to $\cO_{\cS_\star}(\lam^M)$ for any $M$, it suffices to consider the expansion of $ \Phi (\xi+\lam^{-1}\eo)$ for $|\xi|\leq \frac{1}{8|\lam|}$. We write
\begin{equation*}
\begin{split}
      \Phi (\xi+\lam^{-1}\eo) &=  \frac{1}{4\pi} \int_{\R^2} \log\bigl|\xi+\lam^{-1}\eo -\eta\bigr|^2 \Omega(\eta) d\eta\\
     &= \frac{1}{2\pi}\log|\lam|^{-1} \int_{\R^2}\Omega(\eta) d\eta +    \frac{1}{4\pi} \int_{\R^2} \log\left|\eo + \lam(\xi-\eta)\right|^2 \Omega(\eta) d\eta.
\end{split}
\end{equation*}
Thanks to
\begin{equation*}
\begin{split}
    &\Bigl|\frac{1}{2}\log|\eo+x|^2 - \sum_{n=1}^N \frac{(-1)^{n-1}}{n}Q^c_n(x)\Bigr|  \\
    &\lesssim |x|^{N+1} \mathbbm{1}_{|x|\leq \f14} + \left|\log|x+\eo| \right| \mathbbm{1}_{|x+\eo|\leq \f14} + |x|^N \mathbbm{1}_{|x|\geq \f14},
    \end{split}
\end{equation*}
we have
\begin{equation*}
\begin{split}
& \Bigl| \frac{1}{4\pi} \int_{\R^2} \log\left|\eo + \lam(\xi-\eta)\right|^2 \Omega(\eta) d\eta - \sum_{n=1}^N \frac{(-1)^{n-1}}{2\pi n} \lam^{n} \int_{\R^2}   Q^c_n\big( \xi-\eta\big) \Omega(\eta) d\eta \Bigr| \\
 &\lesssim  \int_{\{|\xi-\eta|\leq \frac{1}{4|\lam|}\}}\lam^{N+1}|\xi-\eta|^{N+1}   |\Omega(\eta)| d\eta +    \int_{\{|\xi-\eta|\geq \frac{1}{4|\lam|}\}}\lam^{N}|\xi-\eta|^{N}   |\Omega(\eta)| d\eta \\
 &\quad + \int_{\{ |\eo+ \lam(\xi-\eta)  |\leq \frac{1}{4}\}} \Big|\log|\eo+\lam(\xi-\eta) | \Big|  |\Omega(\eta)| d\eta\eqdefa \cJ_1 + \cJ_2 +\cJ_3.
\end{split}
\end{equation*}
We can easily bound $\cJ_1$ by
\begin{equation}\label{eq 3.7}
    |\cJ_1| \lesssim \lam^{N+1}\lla\xi\rra^{N+1}.
\end{equation}
Now for $\cJ_2$, we get by $|\xi|\leq \frac{1}{8|\lam|}$ and $\Omega= \cO_\cZ(1)$ that
\begin{equation}
\begin{split}
    &|\cJ_2| \lesssim \int_{\{|\eta|\geq \frac{1}{8|\lam|}\}}\lam^{N}  |\xi-\eta|^{N}   |\Omega(\eta)| d\eta \lesssim \int_{\{|\eta|\geq \frac{1} {8|\lam|}\}}\lam^{N}  |\eta|^{N}   |\Omega(\eta)| d\eta  \lesssim \lam^{N} e^{-c|\lam|^{-2}}.
\end{split}
\end{equation}
And similarly, we have
\begin{equation}\label{eq 3.9}
    \begin{split}
&|\cJ_3|\lesssim   \int_{\{|\eta|\geq \frac{1}{8|\lam|}\}\cap \{ |\eo+ \lam(\xi-\eta)  |\leq \frac{1}{4}\}} \big|\log|\lam(\xi-\eta) +\eo |\big|  |\Omega(\eta)| d\eta   \lesssim |\lam|^{-2}e^{-c|\lam|^{-2}}.
    \end{split}
\end{equation}
Finally by summarizing \eqref{eq 3.7}-\eqref{eq 3.9}, we obtained the desired estimate for $\Phi(\xi+\lam^{-1}\eo)$, and the estimates of $\grad_\xi^\beta \Phi(\xi+\lam^{-1}\eo)$ follow along the same line. This completes the proof of  Lemma \ref{Lem 3.4}.
\end{proof}

\renewcommand{\theequation}{\thesection.\arabic{equation}}
\setcounter{equation}{0}
\section{The construction of approximate solutions}\label{sect3}
In this section, we shall construct approximate solutions of (\ref{eq 2.3}-\ref{eq 2.4}).
To measure the error of the approximation solutions and its true solution, we introduce
\begin{equation}\label{S4eq1}
\begin{split}
    &{\rm R}_i[\Omega, \dot{\al}, \dot{\theta}] \eqdefa  t\p_t \Omega_i  - \cL \Omega_i + \frac{1}{\nu} \Big\{\Phi_i(\xi) + \Phi_\is\left(\xi+\ka_i \ve^{-1}\al\eo\right) + \frac{ \ve^2}{2}\dot{\theta}|\xi|^2  \\
&\qquad \qquad \qquad \qquad \quad \qquad \quad \qquad   + \frac{\ve \ka_i\Ga_\is}{\Ga} (\dot{\theta} \al \xi_1+\dot{\al}\xi_2), \  \Omega_i  \Big\}, \quad i=1,2.
\end{split}\end{equation}
For any fixed $M\in \N$, we will seek approximate solutions $(\Omega_{i,a}, \dot{\al}_{a},\dot{\theta}_{a})$, so that
\begin{equation*}
    m[\Omega_{i,a}] \equiv \Ga_i, \quad   M[\Omega_{i,a}]  \equiv 0 \andf  {\rm R}_i[\Omega_a, \dot{\al}_a, \dot{\theta}_a] =\cO_\cZ(\nu^{-1}\ve^{M+1} + \nu \ve^2) ,\quad i=1,2.
\end{equation*}

\subsection{Analysis of the momentum.}

\begin{Lemma}\label{Lem 4.1}
{\sl Let $m[\Omega_i]=\Ga_i$ and $M[\Omega_i]=0$ for $i =1,2$. Then we have
\begin{equation}\label{eq 4.1}
   m[{\rm R}_i[\Omega, \dot{\al}, \dot{\theta}]] =0,
\end{equation}
and
\begin{equation}\label{eq 4.2}
\begin{split}
   \nu
       M[{\rm R}_1[\Omega, \dot{\al}, \dot{\theta}]] &= -\nu
       M[{\rm R}_2[\Omega, \dot{\al}, \dot{\theta}]]
   \\
   &=    \frac{\ve \Ga_1\Ga_2}{\Ga}   \begin{pmatrix}
    \dot{\al} \\-\dot{\theta} \al
\end{pmatrix} - \int_{\R^2}   \grad_\xi^\perp   \Phi_2\left(\xi+\ve^{-1}\al\eo\right) \Omega_1 (\xi) d\xi  .
\end{split}\end{equation}}
\end{Lemma}
\begin{proof} In view of \eqref{S4eq1}, we get, by
using integration by parts and Proposition \ref{Prop 3.1}, that
\begin{equation*}
\begin{split}
    m[{\rm R}_i[\Omega, \dot{\al}, \dot{\theta}]] = \int_{\R^2} \bigl(t\p_t-\lap_\xi -\frac{\xi}{2}\cdot\grad_\xi-1\bigr) \Omega_i \, d\xi = t\p_t m[\Omega_i] = 0,
\end{split}
\end{equation*}
which gives \eqref{eq 4.1}. Similarly, we have
\begin{equation*}
\int_{\R^2} (t\p_t \Omega_i  - \cL \Omega_i) \xi_j d\xi= t\p_t M_j[\Omega_i] + \f12  M_j[\Omega_i]=0,
\end{equation*}
from which, Proposition \ref{Prop 3.1}, and
\begin{equation*}
\begin{split}
    \int_{\R^2}\big\{ \lap_\xi^{-1} f, f \big\} \xi d\xi
     &= -\int_{\R^2}\big\{ \lap_\xi^{-1} f, \xi \big\} f d\xi =  -\int_{\R^2} f \grad_\xi^\perp \lap_\xi^{-1} f     d\xi\\
    &=\frac{-1}{2\pi}\iint \frac{(\xi-\eta)^\perp}{|\xi-\eta|^2} f(\eta) f(\xi) \, d\xi d\eta =0,
\end{split}
\end{equation*}
we infer
\begin{equation*}
    \begin{split}
    &\nu M[{\rm R}_i[\Omega, \dot{\al}, \dot{\theta}]]\\
&=\int_{\R^2}  \Big\{\Phi_i(\xi) +  \Phi_\is\left(\xi+\ka_i \ve^{-1}\al\eo\right) + \frac{ \ve^2}{2}\dot{\theta}|\xi|^2  + \frac{\ve \ka_i\Ga_\is}{\Ga} (\dot{\theta} \al \xi_1+\dot{\al}\xi_2), \  \Omega_i  \Big\}\xi d\xi\\
&= \int_{\R^2}  \Big\{  \Phi_\is\left(\xi+\ka_i \ve^{-1}\al\eo\right) + \frac{ \ve^2}{2}\dot{\theta}|\xi|^2  + \frac{\ve \ka_i\Ga_\is}{\Ga} (\dot{\theta} \al \xi_1+\dot{\al}\xi_2), \  \Omega_i  \Big\}\xi d\xi\\
&= -\int_{\R^2}  \Big\{  \Phi_\is\left(\xi+\ka_i \ve^{-1}\al\eo\right) + \frac{ \ve^2}{2}\dot{\theta}|\xi|^2  + \frac{\ve \ka_i\Ga_\is}{\Ga} (\dot{\theta} \al \xi_1+\dot{\al}\xi_2), \  \xi \Big\} \Omega_i  d\xi \\
&= -  \ve^2\dot{\theta}\int_{\R^2}    \Omega_i \xi^\perp d\xi  -  \int_{\R^2}  \left( \grad_\xi^\perp  \Phi_\is\left(\xi+\ka_i \ve^{-1}\al\eo\right)  + \frac{\ve \ka_i\Ga_\is}{\Ga}  \begin{pmatrix}
    -\dot{\al} \\\dot{\theta} \al
\end{pmatrix} \right)   \Omega_i(\xi)  d\xi \\
&=    - \int_{\R^2}   \grad_\xi^\perp  \Phi_\is\left(\xi+\ka_i \ve^{-1}\al\eo\right) \Omega_i (\xi) d\xi  -\frac{\ve \ka_i\Ga_\is}{\Ga} m[\Omega_i]  \begin{pmatrix}
    -\dot{\al} \\\dot{\theta} \al
\end{pmatrix} \\
&=   \ka_i\frac{\ve \Ga_1\Ga_2}{\Ga}   \begin{pmatrix}
    \dot{\al} \\-\dot{\theta} \al
\end{pmatrix}  - \int_{\R^2}   \grad_\xi^\perp  \Phi_\is\left(\xi+\ka_i \ve^{-1}\al\eo\right) \Omega_i (\xi) d\xi.
\end{split}\end{equation*}
This together with
\begin{equation*}
    \int_{\R^2}   \grad_\xi^\perp  \Phi_1\big(\xi- \ve^{-1}\al\eo\big) \Omega_2 (\xi) d\xi = - \int_{\R^2}   \grad_\xi^\perp  \Phi_2\big(\xi+ \ve^{-1}\al\eo\big) \Omega_1 (\xi) d\xi,
\end{equation*}
implies \eqref{eq 4.2}. We  thus completes the proof of Lemma \ref{Lem 4.1}.
\end{proof}

\subsection{Construction of approximate solutions} Let us recall that $\varepsilon\eqdefa (\nu t)^{\f12}.$ Motivated by
\cite{DG24}, we construct the approximate solutions of (\ref{eq 2.3}-\ref{eq 2.4}) via
\begin{equation}\label{eq 4.4}
    \begin{split}
\Omega_{i,a}^{(M)}(t,\xi)= \sum_{k=0}^M \varepsilon(t)^k \Omega_{i,k}(\xi), \quad \dot{\al}_{a}^{(M)}(t) =   \sum_{k=0}^{M-1} \varepsilon(t)^k \dot{\al}_{k}, \quad \dot{\theta}_{a}^{(M)}(t)= \sum_{k=0}^{M-1} \varepsilon(t)^k \dot{\theta}_{k},
    \end{split}
\end{equation}
where
\begin{equation*}
    \begin{split}
\Omega_{i,k}(\xi)= \Omega_{i,k}^E(\xi) + \nu \Omega_{i,k}^{NS}(\xi) ,\quad  \dot{\al}_{k}=  \nu \dot{\al}_{k}^{NS}, \quad \dot{\theta}_{k}= \dot{\theta}_{k}^E + \nu \dot{\theta}_{k}^{NS},
    \end{split}
\end{equation*}
with each component being independent of $\nu, t$. Hence, we get
\begin{equation*}
    \Phi_{i,a}^{(M)}(t,\xi)= \sum_{k=0}^M \varepsilon(t)^k \Phi_{i,k}(\xi)= \sum_{k=0}^M \varepsilon(t)^k \lap_\xi^{-1}\Omega_{i,k}(\xi) \andf  \al_{a}^{(M)}(t)  = 1+ \sum_{k=2}^{M+1}  \frac{2\varepsilon(t)^k }{k} \dot{\al}_{k-2}^{NS}.
\end{equation*}

The main result of this subsection states as follows:

\begin{Proposition}\label{Prop 4.2}
  {\sl  For any $M\geq 2$, there exists $(\Omega_{a}^{(M)}, \dot{\al}_{a}^{(M)}, \dot{\theta}_{a}^{(M)})$ of the form \eqref{eq 4.4}    satisfying
\begin{itemize}
\item[(a1)]    ${\rm R}_i[\Omega_{a}^{(M)}, \dot{\al}_{a}^{(M)},\dot{\theta}_{a}^{(M)}]= \cO_{\cZ}(\nu^{-1}\ve^{M+1} + \nu \ve^2)$;
    \item[(a2)] $m[\Omega_{i,0}]=\Ga_i$ and  $m[\Omega_{i,k}] =0$ for  $k\geq 1$;
    \item[(a3)] $M[\Omega_{i,k}] =0$ for   $k\geq0$;
    \item[(a4)]   $\Omega_{i,k}^E,\Omega_{i,k}^{NS} \in \cZ$ and  $\Omega_{i,k}^E$ is $\xi_2$-even for  $k\geq 0$;
    \item[(a5)] $\Omega^E_{i,k}$ and $\Omega^{NS}_{i,k}$ have finite Fourier series expansion with respect to variable $\vartheta$;
    \item[(a6)] $\Omega^E_{i,k}, \Omega^{NS}_{i,k}=\cO_{\cZ}(1),$ $ \dot{\al}_k^{NS},\dot{\theta}^E_k, \dot{\theta}^{NS}_k=\cO(1)$ for $k\geq 0$.
\end{itemize}
Moreover, there hold
$$\Omega^E_{i,0}=\Ga_i G,\quad \dot{\theta}^E_0=-\frac{\Ga}{2\pi},\quad \Omega^E_{i,1}=\Omega_{i,0}^{NS}=\Omega_{i,1}^{NS}=\dot{\al}^{NS}_0=0.$$}
\end{Proposition}

\begin{proof} We shall prove this proposition
 by induction. In view of \eqref{eq 2.4}, we take
\begin{equation*}
    \Omega_{i,0}(\xi)=\Ga_i G(\xi) \andf \Phi_{i,0}(\xi)=\Ga_i\Upsilon(\xi).
\end{equation*}

\no {\bf Step 1. The first order approximation.}
Let us take
\begin{equation*}
    \Omega_{i,a}^{(1)}=\Ga_i G+\ve \Omega_{i,1}, \quad \dot{\al}_a^{(1)} = \dot{\al}_0=\nu \dot{\al}^{NS}_0, \quad \dot{\theta}_a^{(1)}= \dot{\theta}_0 \quad
     \text{and thus}\quad  \al_a^{(1)}= 1+ \ve^2\dot{\al}^{NS}_0.
\end{equation*}
Then in view of \eqref{S4eq1},
we find
\begin{equation}\label{eq 4.5a}
\begin{split}
{\rm R}_i[\Omega_a^{(1)},\dot{\al}_a^{(1)},\dot{\theta}_a^{(1)}]
= &\ve \left(1/2-\cL\right) \Omega_{i,1} \\
&+ \frac{1}{\nu} \Big\{ \Ga_i \Upsilon +  \Ga_\is \Upsilon\left(\xi+\ka_i \ve^{-1}\al_a\eo\right) + \ve  \Phi_{i,1} +\ve   \Phi_{\is,1}\left(\xi+\ka_i \ve^{-1}\al_a\eo\right)  \\
& \qquad + \frac{\ve^2}{2}\dot{\theta}_0|\xi|^2   +  \frac{\ve\ka_i\Ga_\is}{\Ga}(\dot{\theta}_0 \al_a^{(1)} \xi_1 +\dot{\al}_0 \xi_2), \ \Ga_i G+\ve \Omega_{i,1} \Big\}.
\end{split}
\end{equation}
While it follows from  Lemma \ref{Lem 3.4} that $\Phi_{\is,1}\left(\xi+\ka_i \ve^{-1}\al_a\eo\right) = C(t)+\cO_{\cS_\star}(\ve)$ and
\begin{equation}\label{eq 4.5}
\begin{split}
     \Upsilon\left(\xi+\ka_i\frac{\al_a}{\ve}\eo\right)  =C(t)+  \frac{\ka_i \ve}{2\pi}\xi_1- \frac{\ve^2}{4\pi}  Q^c_2(\xi) + \cO_{\cS_\star}(\ve^3),
\end{split}
\end{equation}
from which and \eqref{eq 4.5a}, we infer
\begin{equation*}
    \begin{split}
{\rm R}_i[\Omega_a^{(1)},\dot{\al}_a^{(1)},\dot{\theta}_a^{(1)}]
&= \ve \left(1/2-\cL + \Ga_i/\nu\Lam\right) \Omega_{i,1}\\
&\quad + \frac{\Ga_1\Ga_2}{\nu} \Big\{    \frac{\ka_i \ve}{2\pi}\xi_1 + \frac{\ka_i\ve}{\Ga}(\dot{\theta}_0  \xi_1 +\dot{\al}_0 \xi_2)  ,\ G \Big\} +\cO_\cZ(\nu^{-1}\ve^2),
    \end{split}
\end{equation*}
where $\Lambda f=\bigl\{\Delta_\xi^{-1}G, f\bigr\}+\bigl\{\Delta_\xi^{-1}f, G\bigr\}.$

Thus by choosing
\begin{equation}\label{eq 4.6}
    \dot{\theta}_0 = -\frac{\Ga}{2\pi }, \quad \dot{\al}_0 =0 \andf \Omega_{i,1}=0,
\end{equation}
we have ${\rm R}_i=\cO_\cZ(\nu^{-1}\ve^2)$ and $(\Omega_a, \dot{\al}_a, \dot{\theta}_a)$ satisfying (a1)-(a6), and  thus complete the construction for $M=1$.

\no {\bf Step 2. $(M+1)$-th approximation by induction.}
We assume that we have constructed $(\Omega_{a}^{(M)}, \dot{\al}_{a}^{(M)}, \dot{\theta}_{a}^{(M)})$ which
%\begin{equation}\label{eq 4.7} \begin{split}
%\Omega_{i,a}= \sum_{k=0}^M \varepsilon(t)^k \Omega_{i,k}, \quad \dot{\al}_{a} =   \sum_{k=0}^{M-1} \varepsilon(t)^k \dot{\al}_{k}, \quad %\dot{\theta}_{a}= \sum_{k=0}^{M-1} \varepsilon(t)^k \dot{\theta}_{k},
  %  \end{split}\end{equation}
satisfies (a1)-(a6). We are going to seek
\begin{equation}\label{eq 4.8}
    \begin{split}
\Omega_{i,a}^{(M+1)}= \Omega_{i,a}^{(M)}+ \ve^{M+1} \Omega_{i,M+1}, \quad \dot{\al}_{a}^{(M+1)} =   \alpa^{(M)} + \varepsilon^M \dot{\al}_{M},
\quad \dot{\theta}_{a}^{(M+1)}= \thpa^{(M)} +\varepsilon^M \dot{\theta}_{M},
    \end{split}
\end{equation}
satisfying (a2)-(a6) so that ${\rm R}_i[\Omega_{a}^{(M+1)},\dot{\al}_{a}^{(M+1)},\dot{\theta}_{a}^{(M+1)}]= \cO_\cZ(\nu^{-1}\ve^{M+2}+\nu\ve^2).$
In order to do so,
we write
\begin{equation}\label{eq 4.9}
    \begin{split}
&{\rm R}_i[{\Omega}_{a}^{(M+1)}, \dot{\al}_{a}^{(M+1)}, \dot{\theta}_{a}^{(M+1)}] =  (t\p_t-\cL) \big(\Omega_{i,a}^{(M)} + \ve^{M+1} \Omega_{i,M+1}\big)   \\
& \quad + \nu^{-1}\Big\{ \Phi_{i,a}^{(M)}(\xi) +    \Phi_{\is,a}^{(M)}\big(\xi + \ka_i \ve^{-1}\al_a^{(M+1)}\eo\big)
+  \frac{ \ve^2}{2}\dot{\theta}_{a}^{(M+1)} |\xi|^2\\
& \qquad\qquad    +\frac{\ve\ka_i\Ga_\is}{\Ga} \big(\dot{\theta}_{a}^{(M+1)} \al_{a}^{(M+1)} \xi_1+\dot{\al}_{a}^{(M+1)}\xi_2\big)
  \\
 &\qquad \qquad  + \ve^{M+1}  \Phi_{i,M+1} +  \ve^{M+1} \Phi_{\is,M+1}\big(\xi + \ka_i \ve^{-1}\al_a^{(M+1)}\eo\big)  ,\ \Omega_{i,a}^{(M)} + \ve^{M+1} \Omega_{i,M+1} \Big\}\\
&= {\rm R}_i[\Omega_a^{(M)}, \dot{\al}_{a}^{(M+1)}, \dot{\theta}_{a}^{(M+1)}] +\ve^{M+1} \Bigl(\frac{M+1}{2}-\cL + \frac{\Ga_i}{\nu} \Lam \Bigr) \Omega_{i,M+1} + \cO_\cZ(\nu^{-1}\ve^{M+2}+\nu \ve^2).
    \end{split}
\end{equation}
By virtue of the assumptions on $(\Omega_{a}^{(M)}, \dot{\al}_a^{(M)}, \dot{\theta}_a^{(M)})$, we have
\begin{equation}\label{eq 4.10}
    {\rm R}_i[\Omega_a^{(M)}, \dot{\al}_{a}^{(M+1)}, \dot{\theta}_{a}^{(M+1)}] = \nu^{-1}\ve^{M+1} \cH_0^i + \ve^{M+1} \cH_1^i + \cO_\cZ(\nu^{-1}\ve^{M+2} + \nu \ve^2),
\end{equation}
where $\cH_0^i, \cH_1^i\in \cZ$ are independent of $\nu, t$, and have finite Fourier series expansion with respect to variable $\vartheta$.

\no {\bf Step 2.1. Choices of $\dot{\al}_{M}$  and $\dot{\theta}_{M}$.} In this step,  we shall  choose $\dot{\al}_{M}$  and $\dot{\theta}_{M}$  to
guarantee that $M[\cH_k^i]=0$ for $i=1,2$ and $k=0,1$. Since
\begin{equation*}
  \al_{a}^{(M+1)} =   \al_a^{(M)} +   \frac{2\varepsilon^{M+2}}{(M+2)} \dot{\al}_{M}^{NS}\andf \dot{\al}_k =\nu \dot{\al}^{NS}_k\   \text{for} \ k\leq M-1,
\end{equation*}
 we deduce from Lemma \ref{Lem 4.1} that
\begin{equation}
\begin{split}\label{S4eq3}
&\nu \Xi_{M+1} \big[
      M[{\rm R}_1[\Omega_a^{(M)}, \dot{\al}_{a}^{(M+1)}, \dot{\theta}_{a}^{(M+1)}]] \big]  \\
&=\frac{\Ga_1\Ga_2}{\Ga}   \begin{pmatrix}
    \dot{\al}_{M} \\ - \Xi_{M}[\dot{\theta}_{a}^{(M+1)}\al_a^{(M+1)}]
\end{pmatrix} - \Xi_{M+1} \Big[\int_{\R^2}   \grad_\xi^\perp \Phi_{2,a}^{(M)} \big(\xi+\ve^{-1}\al_{a}^{(M+1)}\eo\big) \Omega_{1,a}^{(M)} (\xi) d\xi \Big] \\
&=\frac{\Ga_1\Ga_2}{\Ga}   \begin{pmatrix}
    \nu \dot{\al}_{M}^{NS} \\-\dot{\theta}_{M}-  \sum_{k=0}^{M-2} \frac{2}{M-k}\dot{\theta}_k  \dot{\al}^{NS}_{M-k-2}
\end{pmatrix} \\
&\quad - \Xi_{M+1}\Big[\int_{\R^2}   \grad_\xi^\perp \Phi^{(M),E}_{2,a} \big(\xi+\ve^{-1}\al_{a}^{(M)}\eo\big) \Omega^{(M),E}_{1,a} (\xi) d\xi\Big] \\
&\quad - \nu \Xi_{M+1}\Big[\int_{\R^2}   \grad_\xi^\perp \Phi^{(M),E}_{2,a} \big(\xi+\ve^{-1}\al_{a}^{(M)}\eo\big) \Omega^{(M),NS}_{1,a} (\xi) d\xi\Big] \\
 &\quad- \nu \Xi_{M+1}\Big[\int_{\R^2}   \grad_\xi^\perp \Phi^{(M),NS}_{2,a} \big(\xi+\ve^{-1}\al_{a}^{(M)}\eo\big) \Omega^{(M),E}_{1,a} (\xi) d\xi\Big] + \cO(\nu^2).
    \end{split}
\end{equation}
By using the above equality, we can    uniquely determine $\dot{\theta}_M$ and $ \dot{\al}_M^{NS}$
by $(\Omega_a^{(M)}, \dot{\al}_a^{(M)}, \dot{\theta}_a^{(M)})$ so that
\begin{equation}\label{S4eq2}
   \nu \Xi_{M+1}\big[
      M[{\rm R}_1[\Omega_a^{(M)}, \dot{\al}_{a}^{(M+1)}, \dot{\theta}_{a}^{(M+1)}]]\big]  = \cO(\nu^2 ),
\end{equation}

On the other hand, we observe from \eqref{eq 4.10} that
\begin{equation*}\label{eq 4.11}
\begin{split}
    &\nu
       M[{\rm R}_i[\Omega_a^{(M)}, \dot{\al}_{a}^{(M+1)}, \dot{\theta}_{a}^{(M+1)}]] = \ve^{M+1}
           M[\cH^i_0]+ \nu\ve^{M+1}
           M[\cH^i_1]  +\cO(\ve^{M+2} + \nu^2 \ve^2),
\end{split}
\end{equation*}
which along with \eqref{S4eq2} and Lemma \ref{Lem 4.1} ensures that
\begin{equation}\label{eq 4.12}
    M[\cH^i_0] = M[\cH^i_1]=0 \quad \text{for} \ i=1,2.
\end{equation}

 It remains to show $\dot{\al}_M^{NS}=\cO(1)$. Indeed,  using assumption (a4), we have
\begin{equation*}
    \begin{split}
\int_{\R^2}   \p_2 \Phi^{(M),E}_{2,a} \big(\xi+\ve^{-1}\al_{a}^{(M)}\eo\big) \Omega^{(M),E}_{1,a} (\xi) d\xi = 0,
    \end{split}
\end{equation*}
from which and \eqref{S4eq3}, we deduce that $\dot{\al}_M^{NS}=\cO(1)$.

\no {\bf Step 2.2. Choice of $\Omega_{i,M+1}$.} In this step, formally we shall choose $\Omega_{i,M+1}$ by solving
\begin{equation}\label{eq 4.14}
    \Bigl(\frac{M+1}{2}-\cL + \frac{\Ga_i}{\nu} \Lam \Bigr) \Omega_{i,M+1} + \nu^{-1} \cH^i_0 + \cH^i_1=0.
\end{equation}
In order to do so, we have already chosen $\dot{\al}_{M},\dot{\theta}_{M}$ so  that \eqref{eq 4.12} holds.
Now in view of the definition of ${\rm R}_i$ and (a2), we have
\begin{equation}\label{eq 3.15a}
    m[\cH_0^i]= m[\cH_1^i]=0 \quad \text{for} \ i=1,2.
\end{equation}
Moreover, thanks to  $\dot{\al}_{k}=\nu \dot{\al}_k^{NS}$, we write
\begin{equation*}\begin{split}
    {\rm R}_i[\Omega_a^{(M)}, \dot{\al}_{a}^{(M+1)}, \dot{\theta}_{a}^{(M+1)}]
    &=  \frac{1}{\nu} \Big\{ \Phi^{(M),E}_{i,a}(\xi) +\Phi^{(M),E}_{\is,a}\big(\xi+\ka_i \ve^{-1}\al_a^{(M+1)}\eo\big) +  \frac{ \ve^2}{2}\dot{\theta}_{a}^{(M+1),E}|\xi|^2  \\
&\qquad \quad  + \frac{\ve\ka_i\Ga_\is  }{\Ga} \dot{\theta}_{a}^{(M+1),E} \al_{a}^{(M+1)} \xi_1, \  \Omega^{(M),E}_{i,a}  \Big\} + \cO_\cZ(1),
\end{split}
\end{equation*}
which together with the fact that $\Omega_{i,a}^{(M),E}$, $\Phi_{i,a}^{(M),E}$, $|\xi|^2$ and $\xi_1$ are all $\xi_2$-even functions, implies that
$$  \cH_0^i \ \text{is a $\xi_2$-odd function, in particular,}\ \cP_0 \cH_0=0.  $$

Now we set
\begin{equation}\label{eq 4.16}
    \Omega_{i,M+1} = \Omega_{i,M+1}^{E,0} + \Omega_{i,M+1}^{E,1} + \nu \Omega_{i,M+1}^{NS},
\end{equation}
where $\Omega_{i,M+1}^{E,0}, \Omega_{i,M+1}^{E,1}, \Omega_{i,M+1}^{NS}$ are defined via
\begin{align*}
   &\Bigl(\frac{M+1}{2}-\cL\Bigr) \Omega_{i,M+1}^{E,0} + \cP_0 \cH^i_1=0,\quad  \Ga_i \Lam \Omega_{i,M+1}^{E,1} + \cH_0^i=0 ,\\
    &\Ga_i \Lam \Omega_{i,M+1}^{NS}
 + (1-\cP_0)\cH_1^i + \Bigl(\frac{M+1}{2}-\cL\Bigr)\Omega_{i,M+1}^{E,1}=0.
 \end{align*}
 Then we deduce from \eqref{eq 4.12}, \eqref{eq 3.15a} and Proposition \ref{Prop 3.2}  that the above equations have a unique solution $\Omega^{E,0}_{i,M+1},\Omega^{E,1}_{i,M+1}, \Omega^{NS}_{i,M+1}\in \cZ$  which have  finite Fourier series expansion with respect to variable $\vartheta$, and satisfy
\begin{equation*}
    \Omega_{i,M+1}^{E}\eqdefa\Omega_{i,M+1}^{E,0} + \Omega_{i,M+1}^{E,1} \ \text{is a $\xi_2$-even function},
\end{equation*}
and
\begin{equation*}
   m[\Omega_{i,M+1}]=0, \quad  M[\Omega_{i,M+1}] = 0 \ \text{for } \ i=1,2.
\end{equation*}

Finally by inserting \eqref{eq 4.10} and \eqref{eq 4.16}  into \eqref{eq 4.9}, we find
\begin{equation*}
{\rm R}_i[\Omega_{a}^{(M+1)}, \dot{\al}_{a}^{(M+1)}, \dot{\theta}_{a}^{(M+1)}] =\cO_\cZ(\nu^{-1}\ve^{M+2}+\nu \ve^2),
\end{equation*}
and  (a1)-(a6) are all satisfied by our approximate solutions $(\Omega_{a}^{(M+1)}, \dot{\al}_{a}^{(M+1)}, \dot{\theta}_{a}^{(M+1)}).$ This completes the proof of Proposition \ref{Prop 4.2}.
\end{proof}

\begin{Lemma}\label{Lem 4.3}
    For $M\geq 2$ and the approximate solutions constructed in Proposition \ref{Prop 4.2}, we have
    \begin{equation}
\dot{\theta}_1= \dot{\al}_1=0, \quad
    \Omega^{E}_{i,2}= \Ga_\is \overline{\Omega}^{E}_{2}(\rho) \cos(2\vartheta), \quad
    \Omega^{NS}_{i,2}= \frac{\Ga_\is}{\Ga_i} \overline{\Omega}^{NS}_{2}(\rho) \sin(2\vartheta),
    \end{equation}
where $\overline{\Omega}^E_2(|\xi|)>0,$ $\overline{\Omega}^{NS}_2(|\xi|) $ are two universal functions in $\cY_0$.
\end{Lemma}
\begin{proof}
We inherit the notations in the proof of Proposition \ref{Prop 4.2}, and  take  $M=2$ in \eqref{eq 4.4} to get
\begin{equation*}
    \Omega_{i,a}^{(2)} = \Ga_i G + \ve^2 \Omega_{i,2}, \quad \dot{\al}_{a}^{(2)} = \ve \dot{\al}_1, \quad \dot{\theta}_{a}^{(2)}=-\frac{\Ga}{2\pi } + \ve \dot{\theta}_1, \quad {\al}_{a}^{(2)}= 1+ \frac{2\dot{\al}_1}{3\nu} \ve^3.
\end{equation*}
Then in view of \eqref{S4eq1}, we have
\begin{equation*}
\begin{split}
    {\rm R}_i[\Ga_i G,\dot{\al}_{a}^{(2)},\dot{\theta}_{a}^{(2)}]
    &= \frac{1}{\nu}\Bigr\{\Ga_i \Upsilon + \Ga_{\is} \Upsilon\big(\xi+\ka_i \ve^{-1}\al_{a}^{(2)}\eo\big) \\
    &\qquad \quad   + \frac{\ve^2 }{2} \dot{\theta}_{a}^{(2)} |\xi|^2 + \frac{ \ve \ka_i \Ga_{\is} }{\Ga} \big(\dot{\theta}_{a}^{(2)}\al_{a}^{(2)} \xi_1+ \dot{\al}_{a}^{(2)} \xi_2\big), \ \Ga_i G \Bigr\}\\
    &= \frac{\Ga_1 \Ga_2}{\nu} \Bigl\{   \Upsilon\big(\xi+\ka_i \ve^{-1}\al_{a}^{(2)}\eo\big)  + \frac{ \ve\ka_i  }{\Ga} \big(\dot{\theta}_{a}^{(2)}\al_a^{(2)} \xi_1+ \dot{\al}_{a}^{(2)} \xi_2\big),\   G \Bigr\}.
\end{split}
\end{equation*}
By inserting \eqref{eq 4.5} into above  equality and setting $\dot{\theta}_1=\dot{\al}_1=0$, we obtain
\begin{equation*}
    {\rm R}_i[\Ga_i G,\dot{\al}_{a}^{(2)},\dot{\theta}_{a}^{(2)}] = -\nu^{-1}\ve^2\frac{\Ga_1\Ga_2}{4\pi} \lbk Q^c_2, \, G  \rbk + \cO_{\cZ}(\nu^{-1}\ve^3),
\end{equation*}
which implies $
    \cH^i_0= -\frac{\Ga_1\Ga_2}{4\pi} \lbk Q^c_2, \, G  \rbk$ and $ \cH^i_1=0$. Thus by our strategy of construction of the approximate
    solutions, we get
\begin{equation*}
    \begin{split}
  &  \Omega_{i,2}= \Omega^E_{i,2} + \nu \Omega^{NS}_{i,2} \quad \text{with} \ \Lam \Omega^E_{i,2} =\frac{\Ga_\is}{4\pi} \{ Q^c_2, G\}
 \andf  \Lam\Omega^{NS}_{i,2}= \frac{1}{\Ga_i} (\cL-1) \Omega^E_{i,2}.
    \end{split}
\end{equation*}

Noticing
$ \frac{4}{\rho^2} - \frac{\rho^2/4}{ e^{\rho^2/4}-1} >0$,
we deduce from  \eqref{eq 3.1}, \eqref{eq 3.2} and comparison principle that the equations have a  unique $\overline{\Omega}^{E}_{2}(\rho)>0$ and $\overline{\Omega}^{NS}_{2}(\rho)$ so that
\begin{equation*}
   \Lam\left[\overline{\Omega}^{E}_{2}(\rho) \cos(2\vartheta)\right] =\frac{1}{4\pi} \{ Q^c_2, G\} = -\frac{\rho^2}{16\pi^2}e^{-\rho^2/4} \sin(2\vartheta),
\end{equation*}
and
\begin{equation*}
    \Lam\left[\overline{\Omega}^{NS}_{2}(\rho) \sin(2\vartheta)\right]=(\cL-1) \left[\overline{\Omega}^{E}_{2}(\rho) \cos(2\vartheta)\right].
\end{equation*}
Then we have $
    \Omega^{E}_{i,2}= \Ga_\is \overline{\Omega}^{E}_{2}(\rho) \cos(2\vartheta)$ and $
    \Omega^{NS}_{i,2}= \frac{\Ga_\is}{\Ga_i} \overline{\Omega}^{NS}_{2}(\rho) \sin(2\vartheta)$, which completes the proof of Lemma \ref{Lem 4.3}.
\end{proof}

\begin{Remark}\label{Rmk 4.4}
  Indeed  by more detailed  calculations, we have
  \begin{equation}\begin{split}
    \dot{\theta}^E_a &= -\frac{\Ga}{2\pi} - \ve^4  \frac{\Ga(\Ga_1^2+\Ga_2^2)}{2\pi \Ga_1\Ga_2} \int_{\R^2} (\xi_1^2-\xi_2^2) \overline{\Omega}^E_2(\rho) \cos(2\vartheta) d\xi +\cO(\ve^5) \\
    &= -\frac{\Ga}{2\pi } - \ve^4  \frac{\Ga(\Ga_1^2+\Ga_2^2)}{2  \Ga_1\Ga_2} \int_0^\oo  \overline{\Omega}^E_2\rho^3 d\rho +\cO(\ve^5),
    \end{split}
  \end{equation}
which coincides with the result in \cite{DG24}, which we shall not pursue here and thus
 omit the details.
\end{Remark}

\subsection{Functional relationship of  the Eulerian approximation}
By our construction in Proposition \ref{Prop 4.2}, we have for any $M\geq 1$,
\begin{equation}\label{eq 4.19}
    \Omega^{(M),E}_{i,a}= \sum_{k=0}^M \ve^k\Omega^E_{i,k},\quad \Phi^{(M),E}_{i,a}= \sum_{k=0}^M \ve^k\Phi^E_{i,k}, \quad \dot{\theta}^{(M),E}_a= \sum_{k=0}^{M-1}\ve^k \dot{\theta}^E_k, \quad  \dot{\al}_a^{(M)}= \sum_{k=0}^{M-1}\nu \ve^k \dot{\al}_k^{NS}.
\end{equation}
which gives
\begin{equation}\label{eq 4.20}
    \begin{split}
\al_a^{(M)}=1+ \sum_{k=2}^{M+1} \frac{2 }{k}\ve^k \dot{\al}^{NS}_{k-2}.
    \end{split}
\end{equation}
We introduce
\begin{equation}\label{eq 4.21}
\begin{split}
    \Psi_{i,a}^{(M),E} &\eqdefa \Phi_{i,a}^{(M),E}(\xi) +\Phi_{\is,a}^{(M),E}\big(\xi+  \ka_i\ve^{-1}\al_a^{(M)}\eo\big) + \frac{ \ve^2}{2}\dot{\theta}^{(M),E}_a|\xi|^2    \\
    &\quad + \ve\frac{\ka_i\Ga_\is}{\Ga} \dot{\theta}_a^{(M),E} \al_a^{(M)}\xi_1- \frac{\Ga_{\is}}{2\pi}\log\bigl|\ve/\al_a^{(M)}\bigr| \\
    &\eqdefa \sum_{k=0}^M \ve^k \Psi^E_{i,k} + \cO_{\cS_\star}(\ve^{M+1}),
\end{split}
\end{equation}
we shall  study in this subsection the functional relationship between $\Omega_{i,a}^E$ and $\Psi_{i,a}^E$.

\begin{Proposition}\label{Prop 4.5}
 {\sl   Let $M\geq 2$. There exist $F_i^{(M)}\in \cK$ of the form
$ F_i^{(M)} = \sum_{k=0}^M \ve^k F_{i,k},$
so that
\begin{equation}\label{eq 4.22}
    \Pi_M\big[\Psi^{(M),E}_{i,a} + F_i^{(M)}(\Omega_{i,a}^{(M),E})\big] =0.
\end{equation}}
\end{Proposition}

\begin{Remark}\label{Rmk 4.6}
    We remark that, since $\Omega^{(M),E}_{i,a} = \Ga_i G+ \cO_{\cS_\star}(\ve^2)$, we only have (see \eqref{eq 4.24})
    $$\left|\Omega^{(M),E}_{i,a}-\Ga_i G\right| \ll  G \quad \text{ for $|\xi|\lesssim \ve^{-\sig_1}$}.$$
 This is the reason why we can only define $F_i(\Omega^{(M),E}_{i,a})$ over $|\xi|\lesssim \ve^{-\sig_1}$ instead of whole $\R^2$. However, $\Pi_M (F_i^{(M)}(\Omega^E_{i,a}))$ is well-defined, since it only involves terms like $\f{d^kF^{(M)}_i}{d\xi^k}(\Ga_i G)\bigl(\Omega^{(M),E}_{i,a}-\Ga_i G\bigr)^k$.
\end{Remark}

\begin{proof} Once again we shall prove
   by induction, and to simplify the notations, we drop in this proof the superscript $(M),E$ and subscripts $i,a$ in \eqref{eq 4.19}-\eqref{eq 4.21}.
Since
$$\Omega_0= \Ga G(\xi) \andf \Psi_0 = \Phi_0 = \Ga \Upsilon, $$
we define (without loss of generality, we assume $\Ga>0$)
\begin{equation*}
    F_0(s) = \frac{\Ga}{4\pi}\left( \ga_E- \Ein\log\left((4\pi s)^{-1} \Ga\right)\right) \quad  \text{for} \ 0 < s \leq \frac{\Ga}{4\pi},
\end{equation*}
from which we have $\Psi_0 + F_0(\Omega_0) =0$.

Now we assume that for $M$-th approximation
\begin{equation*}
    \Psi = \sum_{k=0}^M \ve^k \Psi_k ,\qquad \Omega = \sum_{k=0}^M \ve^k \Omega_k\eqdefa \Omega_0 + \Omega_r,
\end{equation*}
we have found $F=\sum_{k=0}^M \ve^k F_k$ so that
$\Pi_M[\Psi + F(\Omega)]=0$. And for refined $(M+1)$-th approximation, we write
\begin{equation}\label{S4eq4}
    \widetilde{\Psi} = \Psi + \ve^{M+1}\Psi_{M+1} ,\qquad \widetilde{\Omega} = \Omega+ \ve^{M+1}\Omega_{M+1},
\end{equation}
we shall seek $\widetilde{F} = F + \ve^{M+1}F_{M+1}$
 so that $\Pi_{M+1}[\widetilde{\Psi} + \widetilde{F}(\widetilde{\Omega})]=0$. Thanks to $\Pi_M[\Psi + F(\Omega)]=0$, we have
\begin{equation*}
    \Pi_{M+1}[\Psi + F(\Omega)] = \ve^{M+1} \cH_{M+1} \quad \text{for some }\ \cH_{M+1} \in \cS_\star,
\end{equation*}
from which and the fact that for any $N\in\N$,
\begin{equation*}
    \begin{split}
&\quad \Pi_N\{F(\Omega) ,\  \Omega\} = \Pi_N\{F(\Omega_0+\Pi_N \Omega_r),\  \Omega_0+\Pi_N \Omega_r\} \\
&= \Pi_N \Big\{  \sum_{\ell \leq N}  \ve^\ell F_\ell(\Omega_0),\  \Omega_0+\Pi_N \Omega_r\Big\} + \Pi_N \Big\{\sum_{k+\ell \leq N, 1\leq k}  \frac{\ve^\ell}{k!} F^{(k)}_\ell(\Omega_0)(\Pi_N \Omega_r)^k,\  \Omega_0+\Pi_N \Omega_r\Big\} \\
&= \Pi_N \Big\{\sum_{\ell \leq N}  \ve^\ell F_\ell(\Omega_0),\ \Pi_N \Omega_r\Big\} +
\Pi_N \Big\{\sum_{k+\ell \leq N, 1\leq k}  \frac{\ve^\ell}{k!} (\Pi_N \Omega_r)^k,\  F^{(k-1)}_\ell(\Omega_0)\Big\} \\
&\quad +\Pi_N \Big\{\sum_{k+\ell \leq N, 1\leq k}  \frac{\ve^\ell}{(k+1)!} F^{(k)}_\ell(\Omega_0),\  (\Pi_N \Omega_r)^{k+1}\Big\}\\
&=\Pi_N \Big\{  \frac{1}{(N+1)!} F^{(N)}_0(\Omega_0),\  (\Pi_N \Omega_r)^{N+1}\Big\}=0,
    \end{split}
\end{equation*}
we deduce that
\begin{equation}\label{eq 4.23}
\begin{split}
    \Pi_{M+1}\left\{ \Psi, \Omega \right\}&= \Pi_{M+1}\left\{ \Psi + F(\Omega), \Omega \right\} \\
    &= \Pi_{M+1}\left\{ \Pi_{M+1}[\Psi + F(\Omega)], \Omega \right\} = \ve^{M+1} \{ \cH_{M+1}, \ \Omega_0 \}.
\end{split}
\end{equation}
As a result, it follows from \eqref{S4eq4} and the strategy of constructing $(\Omega_{M+1},\dot{\theta}_{M},\dot{\al}_{M})$, that
\begin{equation*}
    0=\Pi_{M+1} \{\widetilde{\Psi},\ \widetilde{\Omega}\} = \Pi_{M+1} \{\Psi,\ \Omega\} + \ve^{M+1}\{\Psi_{M+1},\ \Omega_0\} + \ve^{M+1} \{\Psi_{0},\ \Omega_{M+1}\},
\end{equation*}
 which together with \eqref{eq 4.23} implies
 \begin{equation*}
 \begin{split}
     0&=\{ \cH_{M+1}, \ \Omega_0 \} +  \{\Psi_{M+1},\ \Omega_0\} +  \{\Psi_{0},\ \Omega_{M+1}\} \\
     &= \{ \cH_{M+1}, \ \Omega_0 \} +  \{\Psi_{M+1},\ \Omega_0\} - \{F_{0}(\Omega_0),\ \Omega_{M+1}\} \\
     &= \{ \cH_{M+1}, \ \Omega_0 \} +  \{\Psi_{M+1},\ \Omega_0\} + \{F'_{0}(\Omega_0)\Omega_{M+1},\ \Omega_0 \}\\
     &=\{ \cH_{M+1}+\Psi_{M+1}+F'_{0}(\Omega_0)\Omega_{M+1}, \ \Omega_0 \}.
\end{split}
 \end{equation*}
Hence by defining
\begin{equation*}
  F_{M+1}(\Omega_0)\eqdefa -\left(   \cH_{M+1} + \Psi_{M+1} + F'_0(\Omega_0) \Omega_{M+1} \right),
\end{equation*}
and then  using the definitions of $\widetilde{\Psi}, \widetilde{\Omega}$, $\widetilde{F}$ and Leibniz's chain rule, we find
\begin{equation*}
\begin{split}
    \Pi_{M+1}[\widetilde{\Psi}+\widetilde{F}(\widetilde{\Omega})] &=  \Pi_{M+1}[\Psi+F(\Omega)] + \Pi_{M+1}[\widetilde{\Psi}-\Psi] + \Pi_{M+1}[\widetilde{F}(\widetilde{\Omega})-F(\Omega)]\\
    & = \ve^{M+1}\left(\cH_{M+1} + \Psi_{M+1} + F'_0(\Omega_0) \Omega_{M+1} + F_{M+1}(\Omega_0) \right) =0.
\end{split}
\end{equation*}
 This completes the proof of Proposition \ref{Prop 4.5}.
\end{proof}

Notice that \eqref{eq 4.22} is just a pointwise expansion, it doesn't imply any estimates of the remainder terms in any topology, which will be handled in next proposition. We denote $\chi_{\sig_1}(\xi)\eqdefa \chi(\ve^{\sig_1}|\xi|)$, where $\chi$ is a smooth cut-off function satisfying $\chi(r)=1$ for $r\leq 1/2$ and $\chi(r)=0$ for $r\geq 1$.

\begin{Proposition}\label{Prop 4.7}
    {\sl Let $M\geq 2$, $i=1,2$ and let $F_i^{(M)}$ be determined by  Proposition \ref{Prop 4.5},  there exist $\sig_1>0$, $N\in \N$ depending only on $M$ and constant $C$ depending only on $M,\Ga_2$, so that for $\ve$ small enough and $ |\xi| \leq 4 \ve^{-\sig_1}$, there hold
 \begin{subequations} \label{eqProp 4.7}
\begin{align}
     &|\Omega_{i,a}^{(M)}(\xi) - \Ga_i G(\xi)|  \leq C\ve (1+|\xi|)^N G(\xi), \label{eq 4.24} \\
     &|F_i'(\Omega^{(M),E}_{i,a}) -  W_0(\xi)| \leq C\ve (1+|\xi|)^N W_0(\xi),\label{eq 4.25}\\
     & \Theta_i(\ve,\xi) \chi_{\sig_1}(\xi/4) =\cO_{\cS_\star}(\ve^{M+1})   \label{eq 4.26},
\end{align}
\end{subequations}
where $ \Theta_i(\ve,\xi)\eqdefa \Psi^{(M),E}_{i,a} + F_i^{(M)}(\Omega^{(M),E}_{i,a})$ and  $W_0(\xi)\eqdefa 4|\xi|^{-2}\bigl(e^{|\xi|^2/4}-1\bigr)$.}
\end{Proposition}
\begin{proof} Once again,
we  drop the superscripts $(M),E$ and subscripts $i,a$ in this proof. Since $\Omega_k \in \cZ$, we have
\begin{equation*}
    \ve^k |\Omega_k(\xi)| \leq C_k \ve^k (1+|\xi|)^{N_k} G(\xi) \quad \text{for some $C_k>0$ and $N_k\in \N$ }.
\end{equation*}
As a result, it comes out
\begin{equation}\label{eq 4.27}
    |\Omega(\xi) - \Ga G(\xi)| \leq \sum_{k=1}^M \ve^k |\Omega_k(\xi)| \leq C \ve (1+|\xi|)^{N} G(\xi),
\end{equation}
which leads to \eqref{eq 4.24} and in particular implies that $\Theta(\ve,\xi)$ and $F(\Omega)$ are well-defined and smooth over $|\xi| \leq 4 \ve^{-\sig_1}$ when $\sig_1 \ll 1$.

Notice that for  $F\in\cK$,  $G^\ell F^{(\ell)}\in \cK$ for any $\ell \in \N$, we then  deduce from  $F_k\in \cK$ and  \eqref{eq 4.27} that
\begin{equation*}
\begin{split}
  |F'(\Omega) - W_0(\xi)|&= |F'(\Omega) - F'_0(\Ga G)| \\
  &\leq \sum_{k=1}^M |\ve^k F_k'(\Omega)| + |F'_0(\Omega)- F'_0(\Ga G)| \lesssim C\ve (1+|\xi|)^{N'} W_0(\xi),
\end{split}
\end{equation*}
which gives \eqref{eq 4.25}.

 Now for \eqref{eq 4.26}, by using $\Pi_M\Theta=0$ and Taylor expansion, we find
\begin{equation*}
    \Theta(\ve,\xi) =\frac{\ve^{M+1}}{M!} \int_0^1 (1-\tau)^M \p_\ve^{M+1} \Theta(\tau\ve,\xi) d\tau \quad \text{for} \ |\xi| \leq  2 \ve^{-\sig_1},
\end{equation*}
which along with  \eqref{eq 4.27}, $\Omega_k \in \cZ$, $\Psi_k \in \cS_\star$ and $F_k\in \cK$, gives rise to \eqref{eq 4.26} and thus  we  complete the proof of Proposition \ref{Prop 4.7}.
\end{proof}

\renewcommand{\theequation}{\thesection.\arabic{equation}}
\setcounter{equation}{0}
\section{Analysis of the linearized equation}\label{sect4}

\subsection{derivation of the perturbed equation}
Given $(\Omega_a^{(M)}, \dot{\al}_a^{(M)}, \dot{\theta}^{(M)}_a)$ constructed in the previous section with  ${\rm R}_i[\Omega_a^{(M)}, \dot{\al}_a^{(M)}, \dot{\theta}_a^{(M)}]=\cO_{\cZ}(\nu^{-1}\ve^{M+1}+\nu\ve^2)$, we set
\begin{equation}\label{eq 5.1}
\begin{split}
\Omega_i = \Omega_a^{(M)} + \omega_i, \quad \Phi_i =\Phi_{i,a}^{(M)} + \phi_i , \quad \dot{\theta}= \dot{\theta}_a^{(M)}
 +  \thpp,  \quad \dot{\al} = \dot{\al}_a^{(M)} + \dot{\al}_p,
\end{split}
\end{equation}
where the modulation parameters $\dot{\theta}_p$ and $\dot{\al}_p$ remain to be determined.
Then by virtue of Proposition \ref{Prop 3.1} and the properties of $\Omega_{i,a}$, one has
\begin{equation}\label{eq 5.1a}
    m[\om_i]=0 \andf M[\om_1]+M[\om_2]=0   \quad \text{for}\ i=1,2.
\end{equation}

 For simplicity, we neglect the superscript
$(M)$ below.
Motivated by \cite{GS24-1,DG24}, a primary consideration is to choose $\dot{\theta}_p, \dot{\al}_p$ so that $M[\om_i]\equiv 0$ for $t>0$, thereby recovering the coercivity of the energy functional $E_\ve$ defined in \eqref{eq 5.10}. However, as already emphasized in the introduction, the orthogonal condition
\begin{equation}\label{eq 5.2}
   \int_{\R^2}\xi_1\om_i \, d\xi=\int_{\R^2}\xi_2 \om_i \, d\xi  =0
\end{equation}
exhibits incompatibility with the linearized equation around $\Omega_a$ and  allows an ``$\cO(\ve)$ perturbation"(see subsection \ref{subsec. coercivity}). More precisely, \eqref{eq 5.2} can be replaced by
\begin{equation*}
    \int_{\R^2}(\xi_1+\ve P_1) \om_i \, d\xi=\int_{\R^2}(\xi_2+\ve P_2) \om_i \, d\xi  =0 , \quad \text{for some $P_1,P_2\in \cS_\star$},
\end{equation*}
which still keeps the coercivity of $E_\ve$. Seeking suitable $P_1,P_2$ compatibility with the linearized equation  is the key imgredient in this paper, and this will be carried out in subsection \ref{subsec. coercivity}.

For a concise presentation, we  bypass this step for the time being and set $\al_p=0$, then we get, by inserting \eqref{eq 5.1} into \eqref{eq 2.7}, that
\begin{equation}\label{eq 5.3}
\begin{split}
    &(t\p_t - \cL) \om +  \frac{1}{\nu} \Lam^E \om + \frac{1}{\nu} \frac{\ve }{\Ga} \al_a \thpp \{\hat{X}, \Omega^E_{a}\}_V  + \Lam^{NS}\om + \frac{\ve }{\Ga} \al_a \thpp \{\hat{X}, \Omega^{NS}_a\}_V \\
    &= -{\rm R}_a - \frac{1}{\nu} \{\cB_a \om , \om\}_V - \frac{1}{\nu}\frac{\ve }{\Ga} \al_a \thpp \{\hat{X}, \om\}_V.
\end{split}
\end{equation}
where we define
\begin{equation}\label{eq 5.4}
    \begin{split}
& {\rm R}_a\eqdefa({\rm R}_1[\Omega_a,\dot{\al}_a,\dot{\theta}_a],\ {\rm R}_2[\Omega_a,\dot{\al}_a,\dot{\theta}_a])^T  = \cO_{\cZ}(\nu^{-1}\ve^{M+1}+\nu\ve^2),  \\
& \Lam^{E} \om \eqdefa \{\Psi^{E}_a, \om \}_V + \{\cB_a \om, \Omega^{E}_a\}_V ,\quad \Lam^{NS} \om \eqdefa \{\Psi^{NS}_a, \om \}_V + \{\cB_a \om, \Omega^{NS}_a\}_V,
\end{split}\end{equation}
with
\begin{equation}\label{eq 5.5}
\begin{split}
& \Omega_{a}\eqdefa \Omega^E_a + \nu \Omega^{NS}_a, \quad  \hat{X}\eqdefa  \xi_1 Y_2+\frac{\ve\Ga }{2 \al_a}|\xi|^2 Y_1\\
&\cB_a\Omega(\xi)\eqdefa \begin{pmatrix}\lap^{-1}_\xi\Omega_1(\xi) + \lap^{-1}_\xi\Omega_2(\xi+\ve^{-1}\ala\eo)\\ \lap^{-1}_\xi\Omega_2(\xi) + \lap^{-1}_\xi\Omega_1(\xi-\ve^{-1}\ala\eo) \end{pmatrix},\\
 & \Psi^{E}_a \eqdefa (\Psi^E_{1,a}, \Psi^E_{2,a})^T  = \cB_a \Omega^{E}_a + \frac{\ve^2}{2} \dot{\theta}^{E}_a |\xi|^2 Y_1 +\frac{\ve }{\Ga}\dot{\theta}^{E}_a \al_a \xi_1 Y_2-\frac{1}{2\pi}\log\big|\ve/\ala\big| (\Ga_2, \Ga_1)^T \\
 &\qquad = \cB_a \Omega^{E}_a + \frac{\ve }{\Ga}\dot{\theta}^{E}_a \al_a \hat{X}-\frac{1}{2\pi}\log\big|\ve/\ala\big| (\Ga_2, \Ga_1)^T ,\\
&\Psi^{NS}_a \eqdefa\cB_a \Omega^{NS}_a + \frac{\ve^2}{2} \dot{\theta}^{NS}_a |\xi|^2 Y_1 + \frac{\ve }{\Ga}(\dot{\theta}^{NS}_a \al_a \xi_1 + \dot{\al}^{NS}_a\xi_2) Y_2   .
    \end{split}
\end{equation}

Momentarily disregarding the precise topology and derivative structures,
and assuming $|\dot{\theta}_p|\thicksim |\omega|$, we observe from
$\Omega_a^{NS}=\cO_{\cZ}(\ve^2)$ and $\Psi_a^{NS}=\cO_{\cS_\star}(\ve^2)$ that
\begin{equation*}
    (t\p_t-\cL)\om + \frac{1}{\nu} \Lam^E\om  + \frac{1}{\nu} \frac{\ve }{\Ga} \al_a \thpp \{\hat{X}, \Omega^E_{a}\}_V + \text{l.s.t.}= -{\rm R}_a + \text{n.l.t.}.
\end{equation*}
Consequently, studying the linearized operator $\Lam^E$ is  the most crucial task.

\subsection{Weight function and energy functional}
Heuristically, in view of Propositions \ref{Prop 4.5} and \ref{Prop 4.7}, we have
\begin{equation*}
    \Lam^E \om= \lbk\Psi^{E}_a, \om \rbk_V + \lbk\cB_a \om, \Omega^{E}_a\rbk_V  \approx \lbk \om \dot\otimes \begin{pmatrix}
        F_1'(\Omega^E_{1,a}) , F_2'(\Omega^E_{2,a})
    \end{pmatrix}^T + \cB_a \om, \ \Omega^{E}_a\rbk_V ,
\end{equation*}
which gives rise to
\begin{equation*}
    \lla \Lam^E \om, \ \om \dot\otimes \begin{pmatrix}
        F_1'(\Omega^E_{1,a}) , F_2'(\Omega^E_{2,a})
    \end{pmatrix}^T + \cB_a \om  \rra_V \approx 0.
\end{equation*}

However, as mentioned in Remark \ref{Rmk 4.6}, $F_i(\Omega^E_{i,a})$ can only be defined in the region $|\xi|\lesssim \ve^{-\sig_1}$ for  some $\sig_1\ll1$. So in this subsection, we will follow the idea in \cite{DG24} and construct weighted functions via truncation techniques to rigorously realize the aforementioned theoretical framework and recover the skew-adjointness of $\Lam^E$.

 Let $\sigma_1$, $\sigma_2$,
and $\gamma$ be three real numbers satisfying
\begin{equation*}
  0 < \sigma_1 \ll 1 , \qquad \sigma_2 \gg 1, \qquad
  \gamma \eqdefa \sigma_1/\sigma_2\ll1.
\end{equation*}
In particular, $\sig_1$ is chosen small enough to ensure the validity of Proposition \ref{Prop 4.7}.

We now decompose $\R^2$ into three disjoint regions:
\begin{align}
  &\label{inner}\tag{Inner}\mathrm{I}_{i,\ve} \eqdefa \bigl\{\xi\in \R^2 :\, |\xi| < 2\ve^{-\sigma_1},
  \, F_i'(\Omega_{i,a}^E) < \exp(\ve^{-2\sigma_1}/4)\bigr\},\hspace{0.8cm}\\
  &\label{intermediate}\tag{Intermediate}\mathrm{II}_{i,\ve} \eqdefa \bigl \{\xi\in \R^2: \,
  \xi \notin \ive, \, |\xi| \leq \ve^{-\sigma_2}\bigr\},\\
  &\label{outer}\tag{Outer}\mathrm{III}_{i,\ve} \eqdefa \bigl\{\xi\in \R^2: \, |\xi| > \ve^{-\sigma_2}\bigr\},
\end{align}
which depend on time through the parameter $\ve(t)$. Our weighted functions are defined via
\begin{equation}\label{eq 5.4a}
W_\ve\eqdefa \begin{pmatrix}
    W_{1,\ve} \\ W_{2,\ve}
\end{pmatrix} \quad  \text{with} \   W_{i,\ve}(\xi) \eqdefa \begin{cases}
  F_i'(\Omega_{i,a}^E) \quad &\text{in } \,\mathrm{I}_{i,\ve}\,,\\[1mm]
  \exp(\ve^{-2\sigma_1}/4) \quad &\text{in } \,\mathrm{II}_{i,\ve}\,,\\[1mm]
  \exp(|\xi|^{2\gamma}/4) \quad &\text{in } \,\mathrm{III}_{i,\ve}\,.
\end{cases}
\end{equation}
For any fixed $\ve > 0$, the weight functions $W_{i,\ve} >0$ is a positive, locally Lipschitz and piecewise
smooth function, but  $\nabla W_{i,\ve}$ has a discontinuity at the boundaries
of the regions $\mathrm{I}_{i,\ve}, \mathrm{II}_{i,\ve}, \mathrm{III}_{i,\ve}$, see Fig.~\ref{fig1}.

\begin{figure}[ht]
  \begin{center}
 \begin{picture}(200,160)% width and height of the picture
  \put(0,5){\includegraphics[width=1.2\textwidth]{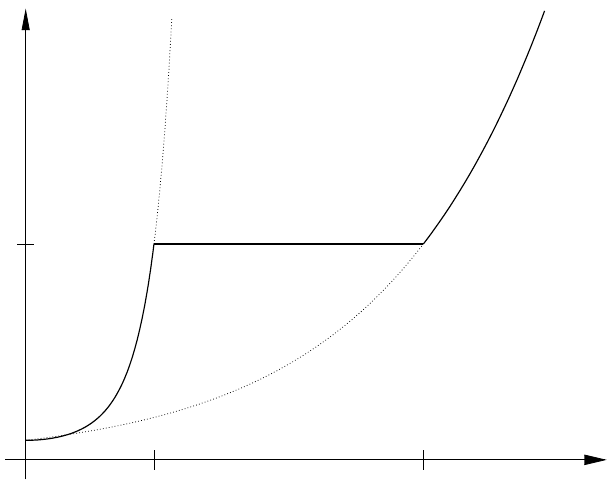}}
  \put(46,17){$\ve^{-\sigma_1}$}
  \put(129,17){$\ve^{-\sigma_2}$}
  \put(170,18){$|\xi|$}
  \put(1,16){$1$}
  \put(-60,77){$\exp(\ve^{-2\sigma_1}/4)$}
  \put(80,85){$W_{i,\ve}(\xi)$}
  \put(56,130){$W_0(\xi)$}
  \put(168,130){$\exp(|\xi|^{2\gamma}/4)$}
  \put(23,0){$\mathrm{I}_{i,\ve}$}
  \put(85,0){$\mathrm{II}_{i,\ve}$}
  \put(150,0){$\mathrm{III}_{i,\ve}$}
  \end{picture}
  \captionsetup{width=0.9  \linewidth}
 \caption{{\small
  In the inner region $\mathrm{I}_{i,\ve}$, the weight is close for
  $\ve > 0$ small to the radially symmetric function $W_0(\xi) = 4|\xi|^{-2}
  \bigl(e^{|\xi|^2/4}-1\bigr)$. It then takes constant values in the intermediate
  region $\mathrm{II}_{i,\ve}$, and grows like $\exp(|\xi|^{2\gamma}/4)$ in
  the outer region $\mathrm{III}_{i,\ve}$. The dashed lines illustrate the bounds
  \eqref{eq 5.8a}, where the constants $C_1, C_2$ are independent of $\ve$. This picture is borrowed from \cite{DG24}. }}\label{fig1}.
  \end{center}
 \end{figure}

\begin{Proposition}\label{Prop 5.1}
  {\sl  Let $\ve,\sig_1\ll1$. Then there hold
\begin{itemize}
    \item[(1)] $\ive$ is diffeomorphic to an open disk  and
\begin{equation}
    \{|\xi|\leq \ve^{-\sig_1}\} \subset \ive \subset \{|\xi|^2 \leq \ve^{-2\sig_1} + \ka |\log\ve|\} \quad \text{for some $\ka>0$}.
\end{equation}
\item[(2)] There exists $C_1,C_2>0$ so that
\begin{equation}\label{eq 5.8a}
    C_1 \exp(|\xi|^{2\ga}/4) \leq \wiv \leq C_2 W_0(\xi).
\end{equation}
\item[(3)] There exists $C_3>0$ depending only on $\sig_1$ and $M$,  so that
\begin{equation}
    |\wiv(\xi)-W_0(\xi)| + |\grad_\xi\wiv(\xi)-\grad_\xi W_0(\xi)|\leq C_3 \ve^{1/2} W_0(\xi), \quad \text{for $\xi \in \ive$}.
\end{equation}
\end{itemize}}
\end{Proposition}
These are some elementary properties of $\wiv$, and the proof was given in \cite{GS24-1,DG24}. For scalar function $f$, we define the weighted spaces
\begin{equation}\label{eq 5.7}
  \cX_{i,\ve} \eqdefa \Bigl\{f\in L_\xi^2(\R^2):\, \|f\|_{\cX_{i,\ve}}^2 \eqdefa \int_{\R^2}|f(\xi)|^2
   W_{i,\ve}(\xi) \,d\xi < \infty\Bigr\},
\end{equation}
and
\begin{equation}\label{eq 5.8}
   \|f\|_\Div\eqdefa  \|\varrho_\ve f\|_\civ + \|f\|_\civ + \|\grad_\xi f\|_\civ,
\end{equation}
with \begin{equation}\label{eq 5.9}
    \varrho_\ve(\xi)\eqdefa \begin{cases}
        |\xi|, &  \text{if}\ |\xi|\leq \ve^{-\sig_1}, \\
        \ve^{-\sig_1}, & \text{if}\  \ve^{-\sig_1}\leq  |\xi|\leq \ve^{-\sig_2},\\
        |\xi|^\ga,  & \text{if}\  \ve^{-\sig_2} \leq |\xi|.
    \end{cases}
\end{equation}
%Then it's easy to find from Proposition \ref{Prop 5.1} and the definition of $\wiv$, that
%\begin{equation}
%\begin{split}
 %   \|W_{i,\ve}^\f12 \om\|_{H^1} &\lesssim \|\grad_\xi (W_{i,\ve}^\f12) \om\|_{L^2} + \|\om\|_\civ + \|\grad_\xi \om\|_\civ \\
 %   &\lesssim \|(W_{i,\ve}^{-\f12}) \grad_\xi \wiv (\mathbbm{1}_\ive + \mathbbm{1}_\iiive) \om\|_{L^2} + \|\om\|_\civ + \|\grad_\xi \om\|_\civ \\
 %   &\lesssim  \|\om\|_{\Div}.
%\end{split}
%\end{equation}

And for 2D vector $\om=(\om_1,\om_2)^T$, we define
 \begin{equation}\label{eq 5.10}
  \|\om\|_{\cX_\ve}\eqdefa\sum_{i=1}^2 \|\om_i\|_{\cX_{i,\ve}},\quad  \|\om\|_{\cD_\ve}\eqdefa\sum_{i=1}^2 \|\om_i\|_{\cD_{i,\ve}}, \quad    E_\ve[\om]\eqdefa \f12 \lla\om, \om \dot\otimes W_\ve + \cB_a \om \rra_V,
 \end{equation}
and
\begin{equation}\label{eq 5.11}
    D_\ve[\om]\eqdefa -\f12 \lla  \om, \ \om \dot\otimes (t\p_t W_{\ve})\rra_V - \lla\cL \om, \ \om\dot\otimes W_\ve+ \cB_a \om \rra_V.
\end{equation}

\subsection{Coercivity of $E_{\ve}$ and $D_{\ve}$.}\label{subsec. coercivity}
In this subsection, we present some fundamental
estimates concerning the estimate of Biot-Savart law and the coercivity of $E_{\ve}$ and $D_{\ve}$.
We first recall some results from \cite{DG24}. Throughout this subsection, $f$ is a scalar function.

\begin{Proposition}[Lemmas 4.5, 4.6 in \cite{DG24}]\label{Prop 5.2}
    {\sl   Let $\ve\ll 1$, $q\in (2,\oo)$, $f\in \civ$ and let $\phi=\lap_{\xi}^{-1}f$. Then there exists a universal  constant $C>0$ so that
\begin{align*}
    &\|(1+|\cdot|)^{-1} \phi\|_{L^q_\xi} +  \|\grad_\xi \phi\|_{L^q_\xi} \leq C\|f\|_\civ,  \\
    & \|(1+|\cdot|) \grad_\xi \phi\|_{L^\oo_\xi} \leq C \|f\|_\civ^{1/2} \left(\|\grad_\xi f\|_\civ^{1/2}+\|f\|_\civ^{1/2}\right).
\end{align*}
Moreover if we assume in addition that $m[f]=0$ and $M[f]=0$, then there holds
\begin{align*}
    &\ve \|\mathbbm{1}_{\ive}(1+|\cdot|)^{-3} \cT_\ve\phi\|_{L^q_\xi} + \|\mathbbm{1}_{\ive}(1+|\cdot|)^{-3} \grad_\xi\cT_\ve\phi\|_{L^q_\xi} \lesssim  \ve^3\|f\|_\civ ,
\end{align*}
where $\cT_\ve\phi(\xi)\eqdefa \phi(-\xi_1-\ve^{-1},\xi_2).$}
\end{Proposition}

\begin{Proposition}[Propositions 4.8 and 4.12 in \cite{DG24}]\label{Prop 5.3}
    {\sl  Let $ \ve,\sig_1 \ll1, f \in \cX_{i,\ve}$ and $\phi=\lap_\xi^{-1}f$. There exist  constants $\ka_1, \ka_D>0$ depending only on $\sig_1$ and $M$, so that if $m[f]=0$ and $M[f]=0$, then there hold
    \begin{equation*}
  \ka_1 \|f\|_{\cX_{i,\ve}}^2 \leq  \|f\|_{\cX_{i,\ve}}^2 + \lla  \phi-\cT_\ve \phi, f \rra  ,
    \end{equation*}
\begin{equation*}
    \ka_D \|f\|_{\cD_{i,\ve}}^2 \leq - 1/2 \lla (t\p_t W_{i,\ve}) f, f\rra - \lla \cL f, W_{i,\ve} f + \phi-\cT_\ve \phi \rra .
\end{equation*}}
\end{Proposition}

Below we shall remove the assumption $M[f]=0$ in Propositions \ref{Prop 5.2} and \ref{Prop 5.3}.

\begin{Proposition}\label{Prop 5.4}
{\sl Let $\ve\ll 1$, $q\in (2,\oo)$, $f\in \civ$ and let $\phi=\lap_{\xi}^{-1}f$. If we assume $m[f]=0$,
then there hold
\begin{equation}\begin{split}
&\ve \left\|\mathbbm{1}_{\ive}(1+|\cdot|)^{-3} \phi(\xi+\ve^{-1}\eo)\right\|_{L^q_\xi} + \left\|\mathbbm{1}_{\ive}(1+|\cdot|)^{-3} \grad_\xi\phi(\xi+\ve^{-1}\eo)\right\|_{L^q_\xi} \\
&\qquad \lesssim  \ve^3\|f\|_\civ + \ve^2 |M[f]|.
\end{split}\end{equation}}
\end{Proposition}

\begin{proof}
  We decompose
\begin{equation}\label{eq 5.18}
    f = -M_1[f] \p_1 G - M_2[f] \p_2 G + f^R.
\end{equation}
Therefore by denoting $\phi^R\eqdefa \lap_\xi^{-1} f^R$, we have
\begin{equation*} \phi= -M_1[f] \p_1 \Upsilon -M_2[f]\p_2 \Upsilon + \phi^{R} \quad \text{with} \  \grad_\xi^\perp \Upsilon = \frac{\xi^\perp}{2\pi |\xi|^2}\left(1-e^{-\frac{|\xi|^2}{4}}\right). \end{equation*}

Let $ \zeta(\xi_1,\xi_2) \eqdefa f^R(-\xi_1,\xi_2)$, we get
\begin{equation*}
 M[\zeta]=0 \andf \phi^R(\xi+\ve^{-1}\eo) = (\lap_\xi^{-1} \zeta)(-\xi_1-\ve^{-1}, \xi_2)=\cT_\ve(\lap_\xi^{-1}\zeta)(\xi).
\end{equation*}
Noticing that
\begin{equation*}
\begin{split}
    &\ve \left\|\mathbbm{1}_{\ive}(1+|\cdot|)^{-3} \p_j \Upsilon(\xi+\ve^{-1}\eo)\right\|_{L^q_\xi}  + \left\|\mathbbm{1}_{\ive}(1+|\cdot|)^{-3} \grad_\xi \p_j \Upsilon(\xi+\ve^{-1}\eo) \right\|_{L^q_\xi} \lesssim \ve^2,
\end{split}
\end{equation*}
we get by Proposition \ref{Prop 5.2} that
\begin{equation*}
  \begin{split}
     & \ve \left\|\mathbbm{1}_{\ive}(1+|\cdot|)^{-3} \phi(\xi+\ve^{-1}\eo)\right\|_{L^q_\xi} + \left\|\mathbbm{1}_{\ive}(1+|\cdot|)^{-3} \grad_\xi\phi(\xi+\ve^{-1}\eo)\right\|_{L^q_\xi} \\
     & \lesssim \ve^2|M[f]| +\ve \left\|\mathbbm{1}_{\ive}(1+|\cdot|)^{-3} \phi^R(\xi+\ve^{-1}\eo)\right\|_{L^q_\xi}  + \left\|\mathbbm{1}_{\ive}(1+|\cdot|)^{-3} \grad_\xi\phi^R(\xi+\ve^{-1}\eo)\right\|_{L^q_\xi} \\
     &= \ve^2|M[f]| +\ve \big\|\mathbbm{1}_{\ive}(1+|\cdot|)^{-3} \cT_\ve (\lap_\xi^{-1} \zeta) \big\|_{L^q_\xi}+ \big\|\mathbbm{1}_{\ive}(1+|\cdot|)^{-3} \grad_\xi\cT_\ve (\lap_\xi^{-1} \zeta)\big\|_{L^q_\xi} \\
      &\lesssim \ve^2|M[f]| + \ve^3 \|\zeta\|_{\cX_{i,\ve}} \lesssim \ve^3 \|f\|_{\cX_{i,\ve}} +\ve^2|M[f]| ,
  \end{split}
\end{equation*}
which completes the proof  of Proposition \ref{Prop 5.4}.
\end{proof}

\begin{Proposition}\label{Prop 5.5}
    {\sl Let $\ve,\sig_1 \ll1, f \in \cX_{i,\ve}$ and $\phi=\lap_\xi^{-1}f$. There exists a  constant $\ka>0$ depending only on $\sig_1$ and $M$, so that if $m[f]=0$, then there hold
     \begin{subequations} \label{eqProp 5.5}
\begin{align}
  \label{eq 5.19}
 & \ka \|f\|_{\cX_{i,\ve}}^2 \leq  |M[f]|^2 + \bigl( \|f\|_{\cX_{i,\ve}}^2 + \lla  \phi, f \rra  \bigr),\\
   \label{eq 5.20}
   & \ka \|f\|_{\cD_{i,\ve}}^2 \leq   |M[f]|^2 + \bigl(- 1/2 \lla (t\p_t W_{i,\ve}) f, f\rra - \lla \cL f, W_{i,\ve} f + \phi \rra \bigr).
\end{align}
\end{subequations}}
\end{Proposition}
\begin{proof}In this proof, we omit the subscript $i$.
We first prove \eqref{eq 5.19} and \eqref{eq 5.20} under the assumption that $M[f]=0$.
 Let $h\eqdefa1+\sqrt{\xi_2^2 + (\xi_1+\ve^{-1})^2}\lesssim 1+ |\xi| + \ve^{-1}.$ We get,
by applying Proposition \ref{Prop 5.2}, that
\begin{equation}\label{eq 5.21}
    \begin{split}
 |\lla \cT_\ve \phi, f \rra| &\leq \left\|\mathbbm{1}_{\mathrm{I}_{\ve}}(1+|\cdot|)^{-3}\cT_\ve \phi\right\|_{L^4_\xi} \left\|(1+|\cdot|)^{3}f\right\|_{L^{4/3}_\xi}  + \left\|h^{-1} \cT_\ve\phi\right\|_{L^4_\xi} \left\| \mathbbm{1}_{\mathrm{I}^c_{\ve}} h f  \right\|_{L^{4/3}_\xi} \\
 &\lesssim \left\|\mathbbm{1}_{\mathrm{I}_{\ve}}(1+|\cdot|)^{-3}\cT_\ve \phi\right\|_{L^4_\xi} \left\|(1+|\cdot|)^{3}f\right\|_{L^{4/3}_\xi}  \\
 &\quad + \ve^{-1} \left\|(1+|\cdot|)^{-1}\phi\right\|_{L^4_\xi} \left\| \mathbbm{1}_{\mathrm{I}^c_{\ve}}(1+|\cdot|)f  \right\|_{L^{4/3}_\xi}\\
 &\lesssim \left(\ve^2+ \ve^{-1}\exp(-c\ve^{-2\sig_1})\right) \|f\|_{\cX_\ve}^2\lesssim  \ve^2 \|f\|_{\cX_\ve}^2.
    \end{split}
\end{equation}
Similarly recalling that $\cL^\star=\Delta_\xi-\f12\xi\cdot\na_\xi,$ one has
\begin{equation}\label{eq 5.22}
\begin{split}
  |\lla \cL^\star \cT_\ve \phi, f \rra| &\lesssim |\lla \cT_\ve f, f \rra| + |\lla \xi\cdot\grad_\xi \cT_\ve \phi, f \mathbbm{1}_{\mathrm{I}_\ve} \rra| + |\lla \xi\cdot\grad_\xi \cT_\ve \phi, f \mathbbm{1}_{\mathrm{I}^c_\ve} \rra|  \\
  &\lesssim \left(\ve^3+ \exp(-c\ve^{-2\sig_1})\right) \|f\|_{\cX_\ve}^2 \lesssim \ve^3 \|f\|_{\cX_\ve}^2.
\end{split}
\end{equation}
Then by virtue of \eqref{eq 5.21},\eqref{eq 5.22} and  Proposition \ref{Prop 5.3}, we obtain \eqref{eq 5.19} and \eqref{eq 5.20} for $M[f]=0$.
%\begin{align*}
%      &\ka_1' \|f\|_{\cX_{\ve}}^2 \leq \|f\|_{\cX_{\ve}}^2 + \lla  \phi, f \rra , \\
 % &\ka_D' \|f\|_{\cD_\ve}^2 \leq  - \f12 \lla (t\p_t W_{i,\ve}) f, f\rra - \lla \cL f, W_{i,\ve} f + \phi \rra .
%\end{align*}
Now for $M[f]\neq 0$, we decompose $f $ as in \eqref{eq 5.18} to get
\begin{equation*}
\begin{split}
     |M[f]|^2 + \left\|f^R\right\|_{\cX_{\ve}}^2 &\lesssim |M[f]|^2 +\left\|f^R\right\|_{\cX_{\ve}}^2 + \lla  \phi^R, f^R \rra \\
     & \lesssim |M[f]|^2 + \left( \|f\|_{\cX_{\ve}}^2 + \lla  \phi, f \rra  \right) + |M[f]| \left\|f^R\right\|_{\cX_{\ve}},
\end{split}
\end{equation*}
which together with Young inequality results in
\begin{equation}
    \|f\|_{\cX_\ve}^2 \lesssim |M[f]|^2 +  \left\|f^R\right\|_{\cX_{\ve}}^2 \lesssim |M[f]|^2 + \left( \|f\|_{\cX_{\ve}}^2 + \lla  \phi, f \rra  \right).
\end{equation}
\eqref{eq 5.20} can be proved along the same line,  we complete the proof of Proposition \ref{Prop 5.5}.
\end{proof}

\subsection{Toy model}\label{subsec. toy model}
In this subsection, we investigate a toy model to illuminate both the structure of $\Lambda^E$ and the subtle orthogonal condition \eqref{eq 4.2}.
Specifically, we examine the scalar equation obtained by linearizing \eqref{eq 2.3} around the profile $G$,
\begin{equation}
    (t\p_t -\cL)f + \nu^{-1} \Lam f ={\rm R},
\end{equation}
where $\Lam f = \lbk \Upsilon,f \rbk  + \{ \lap_\xi^{-1}f, G  \} = \{ f W_0 + \lap_\xi^{-1} f, G \}$.
It is straightforward to verify the relations
\begin{equation}\label{eq 5.24}
    \Lam[\p_j G]=0, \quad  \cL\p_j G= -1/2 \p_jG   \andf \Lam^\star[\xi_j]=0,\quad  \cL^\star \xi_j = -1/2\xi_j.
\end{equation}
In particular,  $\xi_1, \xi_2$ serve as  eigenfunctions of $\Lam^\star$, and  are related to $\p_j G$ by
\begin{equation}\label{relation G}
    \{\xi_1, G\} = \p_2 G \andf \{\xi_2, G\}= -\p_1 G.
\end{equation}
By decomposing $f = \mu_1 \p_1 G + \mu_2 \p_2 G + f^R$ with $M[f^R]=0$, we find
\begin{equation}\label{eq 5.26}
    (t\p_t+1/2)\mu_1 \times \p_1 G+  (t\p_t+1/2)\mu_2 \times \p_2 G + (t\p_t -\cL +\nu^{-1}\Lam)f^R ={\rm R}.
\end{equation}
Moreover, exploiting the structural identities
\begin{equation*}
\begin{split}
    &\big\langle \p_i G , f^R W_0 + \lap^{-1}_\xi f^R \big\rangle  = \big\langle \{\ka_\is \xi_\is, G\}, f^R W_0 + \lap_\xi^{-1}f^R \big\rangle \\
    & = - \big\langle \big\{ f^R W_0 + \lap_\xi^{-1}f^R, G\big\} , \ka_\is \xi_\is \big\rangle = - \big\langle \Lam f^R , \ka_\is \xi_\is \big\rangle = - \big\langle  f^R ,  \ka_\is \Lam^\star \xi_\is \big\rangle =0.
\end{split}
\end{equation*}
and testing  \eqref{eq 5.26} respectively with $-\xi_1$, $-\xi_2$ and $f^R W_0 + \lap_\xi^{-1}f^R$,   we are led to
\begin{equation}\label{eq 5.27}
    (t\p_t +1/2 )\mu_1  = -M_1[{\rm R}],\quad  (t\p_t +1/2 )\mu_2  = -M_2[{\rm R}] ,
\end{equation}
and
\begin{equation}\label{eq 5.29}
    t\p_t E_0[f^R] + D_0[f^R] = \big\langle {\rm R}, f^R W_0 + \lap_\xi^{-1}f^R\big\rangle,
\end{equation}
where $E_0[f^R] \eqdefa \f12 \big \langle f^R, f^R W_0 + \lap_\xi^{-1}f^R \big \rangle$ and  $D_0[f^R]\eqdefa \big \langle -\cL  f^R , f^R W_0 + \lap_\xi^{-1}f^R\big \rangle$.

On the other hand,  the orthogonality constraint $M[f^R]=0$ allows us, by invoking the estimates of \cite{GS24-2, GS24-1, DG24}, to assert
\begin{equation*}
    E_0[f^R] \approx \left\|f^R\right\|_{L^2(\R^2;W_0 d\xi)}^2 \andf D_0[f^R] \gtrsim \left\|f^R\right\|_{H^1(\R^2;W_0 d\xi)}^2,
\end{equation*}
from which, \eqref{eq 5.27}, \eqref{eq 5.29} and above inequalities, we deduce that
$$\|f(t)\|_{L^2(\R^2;W_0 d\xi)} \lesssim  |\mu_1(t)|+|\mu_2(t)|+ \left\|f^R(t)\right\|_{L^2(\R^2;W_0 d\xi)}   \lesssim    \int_0^t \frac{1}{s}\|{\rm R}(s)\|_{L^2(\R^2;W_0 d\xi)}\, ds.  $$
Here  we omit the details. The proof of Theorem \ref{Thm 1.1} basically follows along the same lines. We remark that the construction of ``pseudo-momenta" will be the most crucial part.

\renewcommand{\theequation}{\thesection.\arabic{equation}}
\setcounter{equation}{0}

\section{Construction of ``pseudo-momenta''}\label{sect5}
In this section, we shall construct ``pseudo-momenta" compatible with the operator $\Lam^E$ and orthogonal to $\om^R$, so as to obtain the coercivity of $E_\ve[\om^R]$ and $D_\ve[\om^R]$. To begin with, in view of \eqref{eq 5.4}
 and \eqref{eq 5.5}, we observe that
\begin{equation*}
    \lla \cB_a\om , \varrho \rra_V = \lla \om , \cB_a \varrho\rra_V.
\end{equation*}
Then we define
\begin{equation}\label{eq 6.1}
      \lla \om ,\  \Lam^{E,\star} \varrho \rra_V \eqdefa \lla \Lam^E \om , \ \varrho \rra_V  \quad \text{with $\Lam^{E,\star} \varrho\eqdefa -\left\{\Psi^E_a, \varrho \right\}_V - \cB_a \left\{ \varrho, \Omega^E_a \right\}_V$}.
\end{equation}

Based on the insights observed from the toy model in subsection \ref{subsec. toy model},  ``pseudo-momenta"  should be some eigenfunctions of $\Lam^{E,\star}$, or at least some elements in invariant spaces of $\Lam^{E,\star}$, whose leading order takes the form $S\xi$ for some non-zero constant $2\times 2$ matrix $S$.

\subsection{Expansion of $\Lam^{E,\star}$.}

For any $\varrho=  \cO_{\cS_\star}(1)$, the identity $m[\{\varrho,\Omega^E_a\}_V]=0$ and  Lemma \ref{Lem 3.4} lead to
\begin{equation}
    \Lam^{E,\star}\varrho = \sum_{\ell=0}^M \ve^\ell \Lam^{E,\star}_\ell\varrho + \cO_{\cS_\star}(\ve^{M+1}) \quad \text{with} \  \Lam^{E,\star}_M \varrho \eqdefa \Xi_M\left[ \Lam^{E,\star} \varrho\right].
\end{equation}
By virtue of \eqref{eq 6.1}, one has
\begin{equation*}
    \big(\Lam^{E,\star}_\ell \varrho\big)_i = -\left\{\Psi^E_{i,\ell},  \varrho_i\right\} - \lap_\xi^{-1} \left\{\varrho_i, \Omega^E_{i,\ell}\right\} -  \Xi_\ell \big[\bigl(\lap_\xi^{-1} \left\{\varrho_\is, \Omega^E_{\is,a}\right\}\bigr)\left(\xi+\ka_i \ve^{-1}\al_a\eo\right)\big].
\end{equation*}
It follows from  \eqref{eq 4.5}, \eqref{eq 4.6} and Lemma \ref{Lem 4.3} that
\begin{equation*}
\begin{split}
    &\Omega^E_{i,a} = \Ga_i G+ \ve^2\Omega^E_{i,2} +\cO_\cZ(\ve^3), \quad \dot{\theta}^E_a = -\frac{\Ga}{2\pi } + \cO(\ve^2) , \quad \al_a =1+ \cO( \ve^4), \\
    &\Psi^E_{i,a} =\Phi^E_{i,a}(\xi) +  \Phi^E_{\is,a}\big(\xi+\ve^{-1}\ka_i \al_a\eo\big)+ \frac{ \ve^2}{2}\dot{\theta}^E_a|\xi|^2  + \frac{\ve \ka_i\Ga_\is}{\Ga} \dot{\theta}^E_a \al_a \xi_1 - \frac{\Ga_\is}{2\pi}\log\big|\ve/\ala\big|\\
    &\qquad = \Ga_i \Upsilon + \ve^2\Big( \Phi^E_{i,2} - \frac{\Ga_\is }{4\pi}Q^c_2 - \frac{\Ga}{4\pi}|\xi|^2 \Big) +  \cO_{\cS_\star}(\ve^3)
    \with \Phi^E_{i,a}\eqdefa \Delta_\xi^{-1}\Omega^E_{i,a},
\end{split}
\end{equation*}
from which and Lemma \ref{Lem 3.4}, we deduce that
\begin{equation}\label{eq 6.4}
    \Lam^{E,\star}_0  =   \begin{pmatrix}
        \Ga_1 \Lam^\star ,& 0 \\
        0 , &\Ga_2 \Lam^\star
    \end{pmatrix},
\end{equation}
and
 \begin{equation}\label{eq 6.3}
 \begin{split}
\big(\Lam^{E,\star}_1 \varrho\big)_i  &= -  \Xi_1\big[ \big(\lap^{-1}_\xi \left\{\varrho_\is, \Omega^E_{\is,a}\big\}\right)\left(\xi+\ka_i \ve^{-1}\al_a\eo\right)\big] \\
& = -  \Xi_1\big[ \big(\lap^{-1}_\xi \left\{\varrho_\is, \Ga_\is G\right\}\big)\left(\xi+\ka_i \ve^{-1}\al_a\eo\right)\big]\\
&= -\frac{1}{2\pi} \frac{1}{\ka_i} \int_{\R^2} (\xi_1-\eta_1) \{\varrho_\is, \Ga_\is G\}(\eta) d\eta =C.
\end{split}
 \end{equation}
Moreover, for $\varrho = ( \mu_1 \xi_2, \mu_2 \xi_2)^T$, we have
\begin{equation}\label{eq 6.2}
\begin{split}
   \bigl( \Lam^{E,\star}_2 \varrho\bigr)_i
  &= -\left\{\Psi^E_{i,2}, \mu_i \xi_2 \right\} - \lap^{-1}_\xi\left\{ \mu_i \xi_2, \Omega^E_{i,2} \right\} - \Xi_2 \big[\big(\lap^{-1}_\xi \left\{\mu_\is \xi_2, \Omega^E_{\is,a}\right\}\big)\left(\xi+\ka_i \ve^{-1}\al_a\eo\right)\big]  \\
  & =-\frac{1}{4\pi}\left\{-\Ga_\is Q^c_2 -\Ga |\xi|^2, \mu_i\xi_2\right\}+ \frac{1}{4\pi} \int_{\R^2} Q^c_2(\xi-\eta) \left\{\mu_\is \eta_2, \Ga_\is G\right\} \, d\eta \\
  &= \frac{1}{4\pi}\left\{\Ga_\is Q^c_2
  +\Ga |\xi|^2 , \mu_i\xi_2\right\}+  \frac{\mu_\is \Ga_\is}{4\pi }\left\{\xi_2, Q^c_2 \right\}\\
& = \frac{\mu_i\Ga_\is}{4\pi}\p_1 (\xi_1^2-\xi_2^2) + \frac{\mu_i\Ga}{4\pi} \p_1(|\xi|^2) - \frac{\mu_\is \Ga_\is}{4\pi}\p_1(\xi_1^2-\xi_2^2)\\
  &=\frac{\xi_1}{2\pi} \left( \Ga_\is (\mu_i- \mu_\is)+ \mu_i \Ga\right).
\end{split}
\end{equation}

%, for $\varrho = ( \mu_1 \xi_1, \mu_2 \xi_1)^T$, there holds
%\begin{equation}\label{eq 5.47}
%\begin{split}
%  & ( \Lam^{E,\star}_2 \varrho)^i\\
%  &= -\left\{\Psi^E_{i,2}, \mu_i \xi_1 \right\} - \lap^{-1}\left\{ \mu_i \xi_1, \Omega^E_{i,2} \right\} - \Xi_2 \left[\left(\lap^{-1} \left\{\mu_\is \xi_1, \Omega^E_{\is,a}\right\}\right)\left(\xi+\ka_i \ve^{-1}\al_a\eo\right) \right]  \\
%  & =-\left\{-\frac{\Ga_\is}{4\pi} Q^c_2 -\frac{\Ga}{4\pi} |\xi|^2, \mu_i\xi_1\right\}+ \frac{1}{4\pi} \int_{\R^2} Q^c_2(\xi-\eta) \left\{\mu_\is \eta_1, \Ga_\is G\right\} \, d\eta \\
%  &= \left\{\frac{\Ga_\is}{4\pi} Q^c_2   +\frac{\Ga}{4\pi} |\xi|^2 , \mu_i\xi_1\right\}+ \mu_\is \Ga_\is \left\{\xi_1, \frac{1}{4\pi }Q^c_2 \right\}\\
%& = -\frac{\mu_i\Ga_\is}{4\pi}\p_2 (\xi_1^2-\xi_2^2) - \frac{\mu_i\Ga}{4\pi} \p_2(|\xi|^2) + \frac{\mu_\is \Ga_\is}{4\pi}\p_2(\xi_1^2-\xi_2^2)\\
%  &=\frac{\xi_2}{2\pi} \left( \Ga_\is (\mu_i- \mu_\is)- \mu_i \Ga\right) = - \frac{\mu_1\Ga_1 + \mu_2\Ga_2}{2\pi}\xi_2,
%\end{split}
%\end{equation}

From \eqref{eq 6.4} it follows that there must exist four ``pseudo-momenta’’ to be determined, whose leading order profiles are of the type $S\xi$ with $S$ a non-zero constant $2\times 2$ matrix.

\subsection{Trivial ``pseudo-momenta''}
We first recall from the proof of Proposition \ref{Prop 4.2} that
\begin{equation}\label{eq 6.5}
    {\rm R}^E_a \eqdefa  \lbk\Psi^E_{a}, \Omega^E_{a}\rbk_V = \cO_\cZ(\ve^{M+1}),
\end{equation}
and
\begin{equation}\label{eq 6.6}
\begin{split}
    \Psi^{E}_a & = \cB_a \Omega^{E}_a + \frac{\ve^2}{2} \dot{\theta}^{E}_a |\xi|^2 Y_1 +\frac{\ve }{\Ga}\dot{\theta}^{E}_a \al_a \xi_1 Y_2-\frac{1}{2\pi}\log\big|\ve/\ala\big| (\Ga_2, \Ga_1)^T \\
 & = \cB_a \Omega^{E}_a + \frac{\ve }{\Ga}\dot{\theta}^{E}_a \al_a \hat{X}-\frac{1}{2\pi}\log\big|\ve/\ala\big| (\Ga_2, \Ga_1)^T.
 \end{split}
\end{equation}

\begin{Proposition}\label{Prop 6.1}
 {\sl Let us recall that $ Y_1\eqdefa\begin{pmatrix}
      1,1
  \end{pmatrix}^T$, $ Y_2 \eqdefa \begin{pmatrix}
      \Ga_2 , -\Ga_1
  \end{pmatrix}^T$,  and denote
\begin{equation}\label{eq 6.7a}
\begin{split}
  \varrho^{te}_a\eqdefa \xi_1 Y_1, &\quad \varrho^{to}_a\eqdefa \xi_2 Y_1, \\
    f^{te}_a \eqdefa  \left\{\varrho^{te}_a, \Omega^E_a  \right\}_V=  \p_2 \Omega^E_a, &\quad f^{to}_a \eqdefa \left\{\varrho^{to}_a, \Omega^E_a  \right\}_V = -\p_1 \Omega^E_a.
\end{split}
\end{equation}
We have
\begin{equation}\label{eq 6.7}
    \begin{split}
&\Lam^E f^{te}_a = -\ve^2  \dot{\theta}^E_a f^{to}_a + \p_2 {\rm R}^E_a,\quad \Lam^E f^{to}_a = \ve^2  \dot{\theta}^E_a f^{te}_a - \p_1 {\rm R}^E_a ,  \\
&\Lam^{E,\star} \varrho^{te}_a =\ve^2  \dot{\theta}^E_a \varrho^{to}_a, \quad \Lam^{E,\star} \varrho^{to}_a  = -\ve^2  \dot{\theta}^E_a \varrho^{te}_a - \frac{\ve}{\Ga}\dot{\theta}^E_a \al_a Y_2
 .
    \end{split}
\end{equation}}
\end{Proposition}

\begin{proof} In view of \eqref{eq 5.4}, we get,
by taking $\p_j$ on \eqref{eq 6.5}, that
\begin{equation*}
\begin{split}
  & \p_j {\rm R}^E_a =  \lbk\p_j \Psi^E_{a}, \Omega^E_{a}\rbk_V+ \lbk\Psi^E_{a},\p_j \Omega^E_{a}\rbk_V \\
  & =  \lbk\cB_a \p_j\Omega^E_a , \Omega^E_{a}\rbk_V+ \lbk\Psi^E_{a},\p_j \Omega^E_{a}\rbk_V + \lbk\p_j \left(\ve^2/2 \dot{\theta}^{E}_a |\xi|^2 Y_1 +\ve /\Ga\dot{\theta}^{E}_a \al_a \xi_1 Y_2\right), \Omega^E_a \rbk_V\\
  &=  \Lam^E [\p_j \Omega^E_a]  +  \ve^2  \dot{\theta}^E_a\lbk \xi_j Y_1,  \Omega^E_a  \rbk_V = \Lam^E [\p_j \Omega^E_a] +\ka_j \ve^2  \dot{\theta}^E_a \p_{j^\star} \Omega^E_a,
\end{split}
\end{equation*}
which  proves the first line in \eqref{eq 6.7}.

While we deduce from \eqref{eq 6.1} and \eqref{eq 6.6}   that
\begin{equation*}
    \begin{split}
    \Lam^{E,\star} \varrho^{te}_a&= -\lbk\Psi^E_a ,\xi_1 Y_1 \rbk_V -\cB_a\lbk\xi_1Y_1, \Omega^E_a \rbk_V = \p_2 \Psi^E_a - \cB_a\p_2 \Omega^E_a \\
    &= \p_2 \left(\frac{\ve^2}{2} \dot{\theta}^{E}_a |\xi|^2 Y_1 +\frac{\ve }{\Ga}\dot{\theta}^{E}_a \al_a \xi_1 Y_2\right) =\ve^2  \dot{\theta}^E_a \varrho^{to}_a , \\
     \Lam^{E,\star} \varrho^{to}_a &= -\lbk\Psi^E_a ,\xi_2 Y_1 \rbk_V -\cB_a\lbk\xi_2Y_1, \Omega^E_a \rbk_V= -\p_1 \Psi^E_a + \cB_a\p_1 \Omega^E_a \\
     &= -\p_1 \left(\frac{\ve^2}{2} \dot{\theta}^{E}_a |\xi|^2 Y_1 +\frac{\ve }{\Ga}\dot{\theta}^{E}_a \al_a \xi_1 Y_2\right) = -\ve^2  \dot{\theta}^E_a \varrho^{te}_a - \frac{\ve}{\Ga}\dot{\theta}^E_a \al_a Y_2  ,
    \end{split}
\end{equation*}
and then complete the proof of Proposition \ref{Prop 6.1}.
\end{proof}

However $\varrho^{te}_a, \varrho^{to}_a$ are two trivial ``pseudo-momenta" for $\Lam^{E}$, since in light of \eqref{eq 5.1a}, they remain orthogonal to the nonlinear evolution of $\Omega$ and hence play only a subsidiary role.

\subsection{Non-trivial ``pseudo-momenta".}
To uncover some nontrivial ``pseudo-momenta",  We recall $\xi= \rho(\cos\vartheta, \sin\vartheta)$ and the ``frame stream function"
\begin{equation}\hat{X}\eqdefa  \xi_1 Y_2+\frac{\ve\Ga }{2 \al_a}|\xi|^2 Y_1  .
\end{equation}
For any 2D vector field  $Z=(Z_1,Z_2)^T$, we denote
$ Z\Omega \eqdefa (Z_1 \Omega_1, Z_2 \Omega_2)^T$.
Then using Poisson bracket, we  define ``frame derivative" $X$ via
\begin{equation}
    X\Omega \eqdefa \{\hat{X}, \Omega\}_V =
   \Bigl(   Y_2 \p_2 + \frac{\ve\Ga }{\al_a}  Y_1 \p_\vartheta   \Bigr) \Omega , \quad \text{i.e.,}\ X = \begin{pmatrix}\Ga_2 \p_2 + \frac{\ve \Ga}{\ala}\p_\vartheta\\ -\Ga_1 \p_2 + \frac{\ve \Ga}{\ala}\p_\vartheta \end{pmatrix}.
\end{equation}
Using $X\hat{X}=\{X,X\} =0$, we get
\begin{equation}\label{eq 6.10a}
\begin{split}
    X{\rm R}^E_a &= X \{\Psi^E_a, \Omega^E_a\}_V = \{\Psi^E_a, X\Omega^E_a\}_V + \{X\Psi^E_a, \Omega^E_a\}_V\\
    &= \{\Psi^E_a, X\Omega^E_a\}_V + \{X\left(\cB_a \Omega^E_a \right), \Omega^E_a\}_V.
\end{split}
\end{equation}
Notice that
\begin{equation*}
    \begin{split}
&\left(X\left(\cB_a \Omega^E_a \right)\right)_i
= X_i \Phi_{i,a}^E +X_i  \left(\Phi_{\is,a}^E\left(\xi+  \ka_i \ve^{-1}\al_a\eo\right)\right) \\
& = X_i \Phi_{i,a}^E + \frac{\ve \Ga}{\al_a} (-\xi_2 \p_1 + \xi_1 \p_2 ) \left( \Phi_{\is,a}^E\left(\xi+  \ka_i \ve^{-1}\al_a\eo\right) \right) + \ka_i \Ga_\is  \p_2 \Phi_{\is,a}^E\left(\xi+  \ka_i \ve^{-1}\al_a\eo\right)  \\
& = X_i \Phi_{i,a}^E + \frac{\ve \Ga}{\al_a} \left((-\xi_2 \p_1 + \xi_1 \p_2 )  \Phi_{\is,a}^E\right)\left(\xi+  \ka_i \ve^{-1}\al_a\eo\right)  \\
&\quad -\frac{\ve \Ga}{\al_a}  \frac{\ka_i \al_a}{\ve}\p_2 \Phi_{\is,a}^E\left(\xi+  \ka_i \ve^{-1}\al_a\eo\right)  + \ka_i \Ga_\is  \p_2 \Phi_{\is,a}^E\left(\xi+  \ka_i \ve^{-1}\al_a\eo\right)  \\
& = X_i \Phi_{i,a}^E + \frac{\ve \Ga}{\al_a} \left(\p_\vartheta  \Phi_{\is,a}^E\right)\left(\xi+  \ka_i \ve^{-1}\al_a\eo\right) - \ka_i \Ga_i \p_2\Phi_{\is,a}^E\left(\xi+  \ka_i \ve^{-1}\al_a\eo\right) \\
&= X_i \Phi_{i,a}^E + \left(X_\is \Phi_{\is,a}\right)\left(\xi+  \ka_i \ve^{-1}\al_a\eo\right),
    \end{split}
\end{equation*}
which together with $[\p_\vartheta, \lap_\xi^{-1}]=0$ implies
\begin{equation}\label{eq 6.12a}
    X\left(\cB_a \Omega^E_a \right) = \cB_a \left(X\Omega^E_a \right).
\end{equation}
We thus deduce from  \eqref{eq 6.10a} and \eqref{eq 6.12a} that
\begin{equation}\label{eq 6.12}
    \Lam^E (X\Omega^E_a)= X{\rm R}^E_a = \cO_\cZ(\ve^{M+1}),
\end{equation}
which means that up to an error of size $\cO(\ve^{M+1})$, $ X\Omega^E_a$ is an eigenfunction of $\Lam^E$ corresponding to eigenvalue zero.
The main result in this section states as follows.

\begin{Proposition}\label{Prop 6.2}
{\sl For $M\geq 3$,   there exist $\varrho^e_a, \varrho^o_a\in \cS_{\star}^2 $ and $ \lam^e_a\in \R$
so that
\begin{subequations}
\begin{align}
     &\lbk\varrho^e_a, \Omega^E_a\rbk_V = X \Omega^E_a  +{\rm R}_a^{e,z}, \quad \Lam^{E,\star}\varrho^{e}_a = {\rm R}_a^{e,s}, \quad \text{with $\varrho^e_a$ being $\xi_2$-even},  \label{eq 6.11b}  \\
       &\lamet \varrho^{o}_a = \lam^e \varrho^e_a  +{\rm R}_a^{o,s}, \quad \text{with $\varrho^o_a$ being $\xi_2$-odd},\with  \label{eq 6.11a}\\
     & {\rm R}_a^{o,s} = \cO_{\cS_\star}(\ve^{M+1}) +\Vec{C}(t), \quad {\rm R}_a^{e,s}=\cO_{\cS_\star}(\ve^{M+1}) \andf
     {\rm R}_a^{e,z} = \cO_{\cZ}(\ve^{M+1}),\label{eq 6.11c}
\end{align}
\end{subequations}
 for some vector $\Vec{C}(t)$.
 Moreover, there hold
\begin{equation}\label{eq 6.13}
    \varrho^o_a = \xi_2Y_2 + \cO_{\cS_\star}(\ve), \quad \varrho^e_a = \xi_1Y_2 + \cO_{\cS_\star}(\ve), \quad \lam^e_a  = \ve^2\Ga/\pi + \cO(\ve^3).
\end{equation}
}
\end{Proposition}
\begin{proof}
We first deduce from Proposition \ref{Prop 4.7}
that for any $J_{i,\ell}\in \cK$, $i=1,2$, $\ell =1,\cdots, M$,
\begin{equation}\label{eq 6.14}
    \varrho^e_a =\xi_1 Y_2 +  \frac{\ve\Ga}{2\ala}|\xi|^2 Y_1 +\sum_{\ell=1}^M \ve^\ell \begin{pmatrix}
    J_{1,\ell}(\Omega^E_{1,a})\chi_{\sig_1}(\xi) \\
    J_{2,\ell}(\Omega^E_{2,a}) \chi_{\sig_1}(\xi)
\end{pmatrix}
\end{equation}
with $\chi_{\sig_1}(\xi)\eqdefa \chi(\ve^{\sig_1}|\xi|)$ with $\chi(r)=1$ for $r\leq 1/2$ and $\chi(r)=0$ for $r\geq 1$,
is well-defined and always solves \eqref{eq 6.11b}. Indeed, by taking $\varrho^e_a$ of the form \eqref{eq 6.14} and
 recalling from Proposition \ref{Prop 4.7} that $ \Theta_i(\ve,\xi)\eqdefa \Psi^{E}_{i,a} + F_i(\Omega^{E}_{i,a})$, we find
\begin{equation}\label{eq 6.15}
    \begin{split}
\lbk \varrho^e_a, \Omega^E_a \rbk_V &= X\Omega^E_a + \sum_{\ell=1}^M \ve^\ell \begin{pmatrix}
    \lbk J_{1,\ell}(\Omega^E_{1,a})\chi_{\sig_1}(\xi),  \Omega^E_{1,a}\rbk  \\
   \lbk J_{2,\ell}(\Omega^E_{2,a})\chi_{\sig_1}(\xi),  \Omega^E_{2,a}\rbk
\end{pmatrix}  \\
& = X\Omega^E_a + \sum_{\ell=1}^M \ve^\ell \begin{pmatrix}
   J_{1,\ell}(\Omega^E_{1,a}) \lbk \chi_{\sig_1}(\xi),  \Omega^E_{1,a}\rbk  \\
   J_{2,\ell}(\Omega^E_{2,a}) \lbk \chi_{\sig_1}(\xi),  \Omega^E_{2,a}\rbk
\end{pmatrix} ,
    \end{split}
\end{equation}
and
\begin{equation}\label{eq 6.15a}
    \begin{split}
\lbk\Psi^E_a,  \varrho^e_a\rbk_V &= -X(\cB_a\Omega^E_a) + \sum_{\ell=1}^M \ve^\ell \begin{pmatrix}
    \lbk \Theta_1, J_{1,\ell}(\Omega^E_{1,a}) \chi_{\sig_1}(\xi) \rbk  \\
    \lbk \Theta_2, J_{2,\ell}(\Omega^E_{1,a}) \chi_{\sig_1}(\xi) \rbk
\end{pmatrix}\\
&\quad - \sum_{\ell=1}^M\ve^\ell \begin{pmatrix}
    \lbk F_1(\Omega^E_{1,a}), J_{1,\ell}(\Omega^E_{1,a}) \chi_{\sig_1}(\xi) \rbk  \\
    \lbk F_2(\Omega^E_{1,a}), J_{2,\ell}(\Omega^E_{2,a}) \chi_{\sig_1}(\xi) \rbk
\end{pmatrix}\\
&= -X(\cB_a\Omega^E_a) + \sum_{\ell=1}^M \ve^\ell \begin{pmatrix}
    \lbk \Theta_1, J_{1,\ell}(\Omega^E_{1,a}) \chi_{\sig_1}(\xi) \rbk  \\
    \lbk \Theta_2, J_{2,\ell}(\Omega^E_{1,a}) \chi_{\sig_1}(\xi) \rbk
\end{pmatrix} \\
&\quad - \sum_{\ell=1}^M\ve^\ell \begin{pmatrix}
    F'_1(\Omega^E_{1,a}) J_{1,\ell}(\Omega^E_{1,a})  \lbk \Omega^E_{1,a},  \chi_{\sig_1}(\xi) \rbk  \\
     F'_2(\Omega^E_{2,a}) J_{2,\ell}(\Omega^E_{2,a})  \lbk \Omega^E_{2,a},  \chi_{\sig_1}(\xi) \rbk
\end{pmatrix}.
    \end{split}
\end{equation}
Notice that for any $\cH \in \cS_\star$ supported on $|\xi|\gtrsim \ve^{-\sig_1},$ $\cH$ is of $\cO_{\cS_\star}(\ve^M)$ for any $M$, we obtain
\begin{equation*}
    \lbk \chi_{\sig_1}(\xi),  \Omega^E_{i,a}\rbk  = \cO_\cZ(\ve^M) \quad \text{for any $M$,}
\end{equation*}
which together with \eqref{eq 6.15}, \eqref{eq 6.15a}, \eqref{eq 6.12},  $\Theta_i \chi_{\sig_1}(\xi/4) = \cO_{\cS_\star}(\ve^{M+1})$ and the fact:  $\cB_a$ maps $\cO_{\cZ}(\ve^{M+1})$ into $\cO_{\cS_\star}(\ve^{M+1})$, ensures  \eqref{eq 6.11b}.

Below we will denote $\cR_M^o \eqdefa \Xi_M [\lamet \varrho^o_a - \lam^e_a \varrho^e_a]$ and set
\begin{equation*}
    \varrho^o_a = \xi_2 Y_2  + \sum_{\ell=1}^M \ve^\ell (\varrho^o_\ell + a_\ell \xi_2\et) \andf
  \lam^e_a = \sum_{\ell=0}^M \ve^\ell \lam^e_\ell,
\end{equation*}
with $\varrho^o_\ell$ being $\xi_2$-odd and  $\lla \left(\varrho^o_\ell\right)_i, \p_2 G\rra=0$.  Our final aim is seeking appropriate
$$\text{ $\cJ_{1,\ell}, \cJ_{2,\ell}\in \cK,  \varrho^o_\ell \in \cS_\star^2, a_\ell, \lam^e_\ell\in \R$ for $1\leq \ell \leq M $ so that  $\cR^o_\ell = \Vec{C}$ for all $0\leq \ell \leq M$}. $$
Here we set $a_0=\cJ_{1,0}=\cJ_{2,0}=\varrho^o_0=\lam^e_0=0$ for $\ell=0$.

\no {\bf Step 1. Case $M=0$.}
Using \eqref{eq 6.4}, we have
\begin{equation*}
    \cR^o_0 = \lamet_0\left[\xi_2 Y_2\right] - \lam^e_0 \xi_1 Y_2=- \lam^e_0 \xi_1 Y_2.
\end{equation*}
Thus by setting $\lam^e_0=0$, we have  $\cR^o_0=0.$

\no {\bf Step 2. Case $M=1$.}
Using \eqref{eq 6.4}-\eqref{eq 6.3}, we have
\begin{equation*}
    \begin{split}
\cR^o_1 &= \lamet_0 \left[ a_1 \xi_2 \et
    + \varrho^o_1 \right]  + \lamet_1\left[\xi_2 Y_2\right] - \lam^e_1 \xi_1 Y_2\\
    &= \lamet_0\left[ \varrho^o_1 \right] - \lam^e_1 \xi_1 Y_2 + \Vec{C}.
    \end{split}
\end{equation*}
Then by setting $\varrho^o_1 = \lam^e_1 = 0$,  we have $\cR^o_1 =\Vec{C}$.

\no {\bf Step 3. Case $M=2$.}
Using \eqref{eq 6.4}-\eqref{eq 6.2}, we have
\begin{equation*}
    \begin{split}
\cR^o_2 &= \lamet_0\left[ a_2\xi_2 \et + \varrho^o_2 \right] + \Lam^{E,\star}_1[a_1\xi_2 \et ]  + \lamet_2\left[\xi_2 Y_2\right] - \lam^e_2 \xi_1 Y_2\\
& = \lamet_0\left[ \varrho^o_2 \right]  + \lamet_2\left[\xi_2 Y_2\right] - \lam^e_2 \xi_1 Y_2 +\Vec{C}\\
    &= \lamet_0\left[ \varrho^o_2 \right] + \left(\Ga/ \pi - \lam^e_2\right) \xi_1 Y_2 +\Vec{C}.
    \end{split}
\end{equation*}
Thus by setting $\lam^e_2 = \Ga/\pi$ and $\varrho^o_2 = 0$, we have $\cR^o_2 =\Vec{C}$.

%\no {\bf Step 4. Case $M=3$.}
%Similarly, using \eqref{eq 6.4}-\eqref{eq 6.2} and the information from {\bf Steps 1-3},  we get
%\begin{equation}
%    \begin{split}
%\cR^o_3 &= \lamet_0\left[\varrho^o_3 \right] + \lamet_2\left[ a_1 \xi_2 \et\right] + \Lam^{E,\star}_3 [\xi_2 Y_2]  - \lam^e_3 \xi_1 Y_2   \\
%& \quad -\lam^e_2\left(\frac{\Ga}{2\ala}|\xi|^2 Y_1 + \Xi_1 \left[\begin{pmatrix}
%    J_{1,1}(\Omega^E_{1,a})\chi_{\sig_1}(\xi/2) \\
 %   J_{2,1}(\Omega^E_{2,a}) \chi_{\sig_1}(\xi/2)
%\end{pmatrix} \right] \right)  + \Err \\
%&= \lamet_0\left[\varrho^o_3 \right] + \Lam^{E,\star}_3 [\xi_2 Y_2] +\frac{a_1 \xi_1}{2\pi} \begin{pmatrix}
 %   -\Ga_2 \\ \Ga_1 +\Ga
%\end{pmatrix} - \lam^e_3 \xi_1 Y_2 \\
%&\quad -\frac{\Ga}{\pi}
 %   \end{split}
%\end{equation}

\no {\bf Step 4. Induction for $2\leq N\leq M-1$.}
We assume that we have constructed $\cI_N\eqdefa ( \varrho^o_j, \lam^e_j, a_k, J_{1,k}, J_{2,k})$ for $j\leq N$ and $k\leq N-2$, so that $\cR^o_{j}=\Vec{C}$ for $j \leq N$, and  we will seek $(\varrho^o_{N+1}, \lam^e_{N+1}, a_{N-1}, J_{1,N-1},J_{2,N-1})$ so that $\cR^o_{N+1}=\Vec{C}$.

By direct calculations, we have
\begin{equation*}
    \begin{split}
\cR^{o}_{N+1} &=  \Lam^{E,\star}_{0} \left[a_{N+1}\xi_2 \et + \varrho^o_{N+1}\right]  +\Lam^{E,\star}_{2} \left[a_{N-1}\xi_2 \et + \varrho^o_{N-1}\right] + \cdots +\Lam^{E,\star}_{N+1} \left[\xi_2 Y_2\right] + \Vec{C} \\
& \quad - \lam^e_{N+1} \xi_1 Y_2 -\lam^e_{N}\left( \Xi_0\left[\frac{\Ga }{2\ala}|\xi|^2 Y_1\right] + \Xi_0\left[\begin{pmatrix}
    J_{1,1}(\Omega^E_{1,a})\chi_{\sig_1}(\xi) \\
    J_{2,1}(\Omega^E_{2,a}) \chi_{\sig_1}(\xi)
\end{pmatrix}\right]\right)  - \cdots \\
&\quad -\lam^e_2 \bigg( \Xi_{N-2}\left[\frac{\Ga }{2\ala}|\xi|^2 Y_1\right]   + \Xi_{N-2}\left[\begin{pmatrix}
    J_{1,1}(\Omega^E_{1,a})\chi_{\sig_1}(\xi) \\
    J_{2,1}(\Omega^E_{2,a}) \chi_{\sig_1}(\xi)
\end{pmatrix}\right] \\
&\qquad \qquad +\cdots + \Xi_0\left[\begin{pmatrix}
    J_{1,N-1}(\Omega^E_{1,a})\chi_{\sig_1}(\xi) \\
    J_{2,N-1}(\Omega^E_{2,a}) \chi_{\sig_1}(\xi)
\end{pmatrix}\right]\bigg).
    \end{split}
\end{equation*}
Using again the fact that $\cH \in \cS_\star$ supported on $|\xi|\gtrsim \ve^{-\sig_1}$ is of $\cO_{\cS_\star}(\ve^M)$ for any $M$, we have for any $k,\ell \geq 0$,
\begin{equation*}
\begin{split}
   \begin{pmatrix}
    J_{1,\ell}(\Omega^E_{1,a})\chi_{\sig_1}(\xi) \\
    J_{2,\ell}(\Omega^E_{2,a}) \chi_{\sig_1}(\xi)
\end{pmatrix} &= \begin{pmatrix}
   \Pi_{k} \left[ J_{1,\ell}(\Omega^E_{1,a})\right]\chi_{\sig_1}(\xi) \\
   \Pi_k\left[ J_{2,\ell}(\Omega^E_{2,a})\right] \chi_{\sig_1}(\xi)
\end{pmatrix} + \cO_{\cS_\star}(\ve^{k+1}) \\
& =  \begin{pmatrix}
   \Pi_{k} \left[ J_{1,\ell}(\Omega^E_{1,a})\right] \\
   \Pi_k\left[ J_{2,\ell}(\Omega^E_{2,a})\right]
\end{pmatrix} + \cO_{\cS_\star}(\ve^{k+1}),
\end{split}
\end{equation*}
which implies that
\begin{equation*}
    \begin{split}
\cR^{o}_{N+1}
&=  \Lam^{E,\star}_{0} \left[a_{N+1}\xi_2 \et + \varrho^o_{N+1}\right]  +\Lam^{E,\star}_{2} \left[a_{N-1}\xi_2 \et + \varrho^o_{N-1}\right] + \cdots +\Lam^{E,\star}_{N+1} \left[\xi_2 Y_2\right] + \Vec{C} \\
& \quad - \lam^e_{N+1} \xi_1 Y_2 -\lam^e_{N}\left( \Xi_0\left[\frac{\Ga }{2\ala}|\xi|^2 Y_1\right] + \Xi_0\left[\begin{pmatrix}
    J_{1,1}(\Omega^E_{1,a}) \\
    J_{2,1}(\Omega^E_{2,a})
\end{pmatrix}\right]\right)  - \cdots \\
&\quad -\lam^e_2 \bigg( \Xi_{N-2}\left[\frac{\Ga }{2\ala}|\xi|^2 Y_1\right]   + \Xi_{N-2}\left[\begin{pmatrix}
    J_{1,1}(\Omega^E_{1,a}) \\
    J_{2,1}(\Omega^E_{2,a})
\end{pmatrix}\right]  +\cdots + \Xi_0\left[\begin{pmatrix}
    J_{1,N-1}(\Omega^E_{1,a}) \\
    J_{2,N-1}(\Omega^E_{2,a})
\end{pmatrix}\right]\bigg)\\
&= \lamet_0[\varrho^o_{N+1}] + \lamet_2\left[
         a_{N-1} \xi_2\et \right]  - \lam^e_{N+1}\xi_1 Y_2\\
    &\quad - \frac{\Ga}{\pi} \begin{pmatrix}
        J_{1,N-1}(G)  \\ J_{2,N-1}(G)
    \end{pmatrix} +  \begin{pmatrix}
        S_1(\cI_N, \Omega^E_a, \al^E_a, \dot{\theta}^E_a)\\ S_2 (\cI_N, \Omega^E_a, \al^E_a, \dot{\theta}^E_a)
    \end{pmatrix} +\Vec{C}
    \end{split}
\end{equation*}
for some $S_1,S_2\in \cS_\star$ depending only on $\cI_N, \Omega^E_a, \al_a, \dot{\theta}^E_a$. And it follows from the $\xi_2$-odevity of $\Omega^E_a,\Psi^E_a, \xi_2 .\varrho^o_\ell, |\xi|^2$ and the definition of $\Lam^{E,\star}, \Lam^{E,\star}_\ell$ that $S_1,S_2$ are $\xi_2$-even.

Therefore by virtue   of \eqref{eq 6.2}, we find
\begin{equation*}
    \Lam^{E,\star}_2[ \xi_2 \et]= \frac{\xi_1}{2\pi} (-\Ga_2 , 2\Ga_1+\Ga_2)^T.
\end{equation*}
By choosing $a_{N-1}, \lam^e_{N+1}, J_{1,N-1}, J_{2,N-1}$ so that
\begin{equation*}
\left\{
\begin{aligned}
    &\frac{-\Ga_2}{2\pi}a_{N-1} - \lam^e_{N+1} \Ga_2 = \int_{\R^2} S_1  \p_1G\, d\xi, \\
    &\frac{2\Ga_1 + \Ga_2}{2\pi}a_{N-1} +  \lam^e_{N+1} \Ga_1 = \int_{\R^2} S_2  \p_1G\, d\xi,\\
    & \frac{\Ga}{\pi} J_{i,N-1}(G) = \cP_0 S_i  ,
\end{aligned}
\right.
\end{equation*}
and then letting $\varrho^o_{N+1}$ solve
\begin{equation*}
    \begin{split}
       &  \lamet_0[\varrho^o_{N+1}] =- \frac{a_{N-1}\xi_1}{2\pi}\begin{pmatrix}
        -\Ga_2 \\ 2\Ga_1 + \Ga_2
    \end{pmatrix}  + \lam^e_{N+1} \begin{pmatrix}
        \Ga_2 \xi_1 \\ -\Ga_1 \xi_1
    \end{pmatrix} - (1-\cP_0)  \begin{pmatrix}
        S_1 \\ S_2
    \end{pmatrix},
    \end{split}
\end{equation*}
we deduce from Proposition \ref{Prop 3.2} that there exists a unique $\varrho_{N+1}\in\cS_\star^2$ so that $\varrho^o_{N+1}$ is $\xi_2$-odd and $\lla \left(\varrho^o_{N+1}\right)_i, \p_2 G \rra=0 $. With such choice of $(\varrho^o_{N+1}, \lam^e_{N+1}, a_{N-1}, J_{1,N-1},J_{2,N-1})$, we have $\cR^o_{N+1}=\Vec{C}$,  which  completes the proof of Proposition \ref{Prop 6.2}.
\end{proof}

\begin{Proposition}\label{Prop 6.3a}
{\sl Let $M\geq 3$. For $\varrho^e_a, \varrho^o_a, \lam^e_a, {\rm R}_a^{e,z}, {\rm R}_a^{e,s}, {\rm R}_a^{o,s}$ constructed in Proposition \ref{Prop 6.2}, we denote
\begin{equation}\label{eq 6.18a}
\begin{split}
f^o_a\eqdefa \lbk\varrho^o_a , \Omega^E_a\rbk_V, \quad f^e_a \eqdefa \lbk\varrho^e_a , \Omega^E_a\rbk_V = X \Omega^E_{a}+ {\rm R}_a^{e,z}.
\end{split}
\end{equation}
Then we have
        \begin{subequations}
\begin{align}
& \Lam^E f^e_a = X {\rm R}^E_a + \Lam^E\left[ {\rm R}_a^{e,z}\right]=\cO_{\cZ}(\ve^{M+1}),\label{eq 6.18ba}\\
&\Lam^E f^o_a =  -\lam^e_a f^e_a - \lbk{\rm R}_a^{o,s}, \Omega^E_a\rbk_V +\{\varrho^o_a,{\rm R}^E_a \}_V =-\lam^e_a f^e_a +\cO_{\cZ}(\ve^{M+1}),\label{eq 6.18bb}\\
& \Lam^{E,\star} \varrho^e_a =  {\rm R}_a^{e,s},\quad \Lam^{E,\star} \varrho^o_a = \lam^e_a \varrho^e_a + {\rm R}_a^{o,s},\label{eq 6.18bc}
\end{align}
    \end{subequations}}
\end{Proposition}

\begin{Remark}
    The definition of $f^o_a,f^e_a$ is motivated by relation \eqref{relation G} and \eqref{eq 6.7a}, which reveals the fact that, the  eigenfunctions of $\Lam^{E,\star}$ and $\Lam^E$  are  connected by the operator $\{\cdot,\Omega^E_a\}_V$.
\end{Remark}

\begin{proof}
     It suffices to prove  \eqref{eq 6.18ba} and \eqref{eq 6.18bb}.  We first get, by
     using \eqref{eq 6.12} and \eqref{eq 6.11c}, that
     \begin{equation*}
    \Lam^Ef^e_a = \Lam^E \left[X\Omega^E_a \right] + \Lam^E\left[{\rm R}_a^{e,z}\right] = X{\rm R}^E_a +  \Lam^E\left[{\rm R}_a^{e,z}\right]  =\cO_{\cZ}(\ve^{M+1}),
     \end{equation*}
     which leads to \eqref{eq 6.18ba}.

     While by using  Jacobi identity, we obtain
    \begin{equation*}
        \begin{split}
   \Lam^E[f^o_a]&= \Lam^E[\lbk\varrho^o_a, \Omega^E_a \rbk_V] = \left\{\Psi^E_a,\  \lbk\varrho^o_a, \Omega^E_a \rbk_V  \right\}_V + \left\{ \cB_a\lbk\varrho^o_a, \Omega^E_a \rbk_V, \ \Omega^E_{a} \right\}_V \\
    & = \left\{\Psi^E_a,\  \lbk\varrho^o_a, \Omega^E_a \rbk_V  \right\}_V + \left\{ -\lamet[\varrho^o_a]-  \lbk\Psi^E_a, \varrho^o_a \rbk_V, \ \Omega^E_{a} \right\}_V \\
    &= - \lam^e_a \lbk\varrho^e_a , \Omega^E_a\rbk_V - \lbk{\rm R}_a^{o,s}, \Omega^E_a\rbk_V + \big( \left\{\Psi^E_a,\  \lbk\varrho^o_a, \Omega^E_a \rbk_V  \right\}_V +  \left\{  \Omega^E_{a} ,  \lbk\Psi^E_a,  \varrho^o_a \rbk_V \right\}_V \big) \\
    &= - \lam^e_a f^e_a - \lbk{\rm R}_a^{o,s}, \Omega^E_a\rbk_V - \lbk\varrho^o_a, \{\Omega^E_a, \Psi^E_a\}_V \rbk_V \\
    &= - \lam^e_a f^e_a - \lbk{\rm R}_a^{o,s}, \Omega^E_a\rbk_V +\lbk\varrho^o_a,{\rm R}^E_a \rbk_V,
        \end{split}
    \end{equation*}
which results in  \eqref{eq 6.18bb} and we thus completes the proof of Proposition \ref{Prop 6.3a}.
\end{proof}

\subsection{Inner product structure.}
%In what follows, we will  follow the previous definition of $f^{\star}_a , \varrho^\star_a$ for $\star=e,o,te,to$ and  ${\rm R}^E_a, {\rm R}_a^{e,z},{\rm R}_a^{e,s},{\rm R}_a^{o,s} $ in \eqref{eq 6.5}, \eqref{eq 6.7a}, \eqref{eq 6.11b}, \eqref{eq 6.11a} and \eqref{eq 6.18a}.

By summarizing Propositions \ref{Prop 6.1}- \ref{Prop 6.3a}, we achieve
\begin{equation}\label{eq 6.18}
\left\{
\begin{aligned}
& \varrho^e_a = \xi_1Y_2 +\cO_{\cS_\star}(\ve),\quad \varrho^o_a = \xi_2Y_2 +\cO_{\cS_\star}(\ve),\quad \varrho^{te}_a = \xi_1Y_1,\quad \varrho^{to}_a = \xi_2Y_1,\quad\\
& f^{\star}_a  = \lbk\varrho_a^\star,\Omega^E_a\rbk_V, \quad \text{for }\ \star=e,o,te,to, \quad \lam^e_a=\ve^2 \Ga/\pi + \cO(\ve^3),\\
     & \Lam^E f^e_a = X {\rm R}^E_a + \Lam^E\left[ {\rm R}_a^{e,z}\right],\quad \Lam^E f^o_a =  -\lam^e_a f^e_a - \lbk{\rm R}_a^{o,s}, \Omega^E_a\rbk_V + \lbk\varrho^o_a,{\rm R}^E_a \rbk_V,\\
& \Lam^E f^{te}_a = -\ve^2 \dot{\theta}^E_a f^{to}_a + \p_2 {\rm R}_a^{E}, \quad \Lam^E f^{to}_a = \ve^2 \dot{\theta}^E_a f^{te}_a - \p_1 {\rm R}_a^{E},  \\
& \Lam^{E,\star} \varrho^e_a =  {\rm R}_a^{e,s},\quad \Lam^{E,\star} \varrho^o_a = \lam^e_a \varrho^e_a + {\rm R}_a^{o,s} , \\
 &\Lam^{E,\star} \varrho^{te}_a=\ve^2  \dot{\theta}^E_a \varrho^{to}_a ,\quad  \Lam^{E,\star} \varrho^{to}_a = -\ve^2  \dot{\theta}^E_a \varrho^{te}_a - \frac{\ve}{\Ga}\dot{\theta}^E_a \al_a Y_2 ,
\end{aligned}
\right.
\end{equation}
where  ${\rm R}_a^{e,z},{\rm R}_a^{E} = \cO_{\cZ}(\ve^{M+1})$,
 ${\rm R}_a^{e,s} = \cO_{\cS_\star}(\ve^{M+1})$ and ${\rm R}_a^{o,s} =\cO_{\cS_\star}(\ve^{M+1}) +\Vec{C}(t)$ for some  vector $\Vec{C}(t)$.
The superscripts  ``e" or ``o" over $``\varrho"$  denote whether the function is $\xi_2$-even or $\xi_2$-odd, respectively,  whereas for $``f"$ and $``{\rm R}"$, it is the opposite. And the superscripts $``s",``z"$ over $``{\rm R}"$ indicate $``\cS_\star"$ and $``\cZ"$, respectively, to which ${\rm R}$ belongs.  The following proposition concerns the inner product structure of $f^{\star}_a$ for $\star=e,o,te,to$.

\begin{Proposition}\label{Prop 6.3}
{\sl There hold
\begin{equation}\label{eq 6.19}
\cI\eqdefa  \begin{pmatrix}
 \lla f^{te}_a, \ \varrho^{te}_a \rra_V, & \lla f^{to}_a, \ \varrho^{te}_a \rra_V, &   \lla f^{e}_a, \ \varrho^{te}_a \rra_V, & \lla f^{o}_a, \ \varrho^{te}_a \rra_V \\
\lla f^{te}_a, \ \varrho^{to}_a \rra_V, & \lla f^{to}_a, \ \varrho^{to}_a \rra_V,  &  \lla f^{e}_a, \ \varrho^{to}_a \rra_V, & \lla f^{o}_a, \ \varrho^{to}_a \rra_V \\
\lla f^{te}_a, \ \varrho^{e}_a \rra_V, & \lla f^{to}_a, \ \varrho^{e}_a \rra_V,   & \lla f^{e}_a, \ \varrho^{e}_a \rra_V, & \lla f^{o}_a, \ \varrho^{e}_a \rra_V \\
\lla f^{te}_a, \ \varrho^{o}_a \rra_V, &\lla f^{to}_a, \ \varrho^{o}_a \rra_V,  &  \lla f^{e}_a, \ \varrho^{o}_a \rra_V, & \lla f^{o}_a, \ \varrho^{o}_a \rra_V \\
  \end{pmatrix}= \begin{pmatrix}
      0 , &\Ga, &0,& \Bar{a}\\
      -\Ga,& 0,& \Bar{c}, &0\\
      0,&-\Bar{c}, &0, &\Bar{b} \\
      -\Bar{a},&0,&-\Bar{b},& 0
  \end{pmatrix},
\end{equation}
where
\begin{equation}
    \begin{split}
    \Bar{a}&\eqdefa\lla f^{o}_a, \ \varrho^{te}_a \rra_V = - \lla f^{te}_a, \ \varrho^{o}_a \rra_V= \cO(\ve^{M-1}),\\
    \Bar{b}&\eqdefa \lla f^{o}_a, \ \varrho^{e}_a \rra_V =-\lla f^{e}_a, \ \varrho^{o}_a \rra_V = \Ga_1 \Ga_2 \Ga+ \cO(\ve),\\
    \Bar{c}&\eqdefa\lla f^{e}_a, \ \varrho^{to}_a \rra_V = - \lla f^{to}_a, \ \varrho^{e}_a \rra_V =\lla {\rm R}_a^{e,z}, \varrho^{to}_a \rra_V = \cO(\ve^{M+1}).
    \end{split}
\end{equation}}
\end{Proposition}
\begin{proof}
First of all, the zero elements in the $\cI$ simply follow the $\xi_2$-odevity of functions.
And by using
\begin{equation*}
    \lla \{a,\Omega^E_a\}_V, \ b\rra_V = - \lla \{b,\Omega^E_a\}_V, \ a\rra_V
\end{equation*}
we deduce that $\cI$ is a skew-symmetric matrix, thus it suffices to consider the upper triangular part of $\cI$. By virtue of  the properties of $\Omega^E_a$, we achieve
\begin{equation*}
\begin{split}
 \lla f^{to}_a, \varrho^{te}_o \rra_V &= \lla  - \p_1 \Omega^E_a, \xi_1 Y_1 \rra_V= \lla   \Omega^E_a, Y_1 \rra_V=\Ga,
\end{split}\end{equation*}
and
\begin{equation*}\begin{split}
\Bar{b} &= \lla f^{o}_a, \ \varrho^{e}_a \rra_V = \lla \lbk \xi_2 Y_2 + \cO_{\cS_\star}(\ve), \Omega^E_a\rbk_V , \ \xi_1 Y_2 + \cO_{\cS_\star}(\ve) \rra_V \\
   &= - \lla \lbk \xi_2 Y_2 ,\xi_1 Y_2 \rbk_V + \cO_{\cS_\star}(\ve), \ \Omega^E_a  \rra_V
   = \Ga_1 \Ga_2 \Ga +\cO(\ve), \\
\Bar{c} &=  \lla f^e_a, \varrho^{to}_a \rra_V = \lla X \Omega^E_a + {\rm R}_a^{e,z}, \varrho^{to}_a \rra_V = - \lla  \Omega^E_a, X\varrho^{to}_a \rra_V +\lla {\rm R}_a^{e,z}, \varrho^{to}_a \rra_V\\
    &= - \lla  \Omega^E_a, \ve \Ga/\ala\xi_1 Y_1 +Y_2 \rra_V +\lla {\rm R}_a^{e,z}, \varrho^{to}_a \rra_V
    =\lla {\rm R}_a^{e,z}, \varrho^{to}_a \rra_V = \cO(\ve^{M+1}),\\
\Bar{a}&= \lla f^o_a, \varrho^{te}_a \rra_V = (\ve^2  \dot{\theta}^E_a)^{-1}\big\langle f^o_a,\  -\lamet\varrho^{to}_a - \ve/\Ga\dot{\theta}^E_a \al_a Y_2  \big\rangle_V =  -(\ve^2  \dot{\theta}^E_a)^{-1}\lla \Lam^E f^o_a,\ \varrho^{to}_a \rra_V  \\
&=-(\ve^2  \dot{\theta}^E_a)^{-1} \lla -\lam^e_a f^e_a - \{{\rm R}_a^{o,s},\Omega^E_a\}_V+ \{\varrho^o_a,{\rm R}^E_a\}_V , \varrho^{to}_a \rra_V =\cO(\ve^{M-1}),\\
\end{split}
\end{equation*}
which complete the proof of Proposition \ref{Prop 6.3}.

\end{proof}

\renewcommand{\theequation}{\thesection.\arabic{equation}}
\setcounter{equation}{0}

\section{Estimation of the error}\label{sect6}

\subsection{Decomposition of $\om$.}
With the ``pseudo-momenta'' constructed in Section \ref{sect5}, we obtain the following structural decomposition.

\begin{Lemma}\label{Lem 7.1}
{\sl Let $(\om,\dot{\theta}_p)$ solve \eqref{eq 5.3}. Then we have
\begin{equation}\label{eq 7.1}
    \om = \mu^o\big(f^o_a-\frac{\Bar{a}}{\Ga} f^{to}_a\big) + \mu^e \big(f^e_a +\frac{\Bar{c}}{\Ga} f^{te}_a\big)+ \om^R,
\end{equation}
with $ \mu^o=\left(\Bar{b}+\Bar{a}\Bar{c}/\Ga\right)^{-1}  \lla \om, \varrho^{e} \rra_V , \mu^e = - \left(\Bar{b}+\Bar{a}\Bar{c}/\Ga\right)^{-1}  \lla \om, \varrho^{o} \rra_V $ and
\begin{equation}\label{eq 7.2}
   m[\om^R_i]=0, \quad \text{for} \ i=1,2 \andf \lla\om^R, \varrho^\star_a\rra_V=0 \quad \text{for}\  \star=e,o,te,to.
\end{equation}}
\end{Lemma}
\begin{Remark}
  Since $\Gamma\neq 0$, the vectors $Y_1,Y_2$ are linearly independent. Together with \eqref{eq 7.2}, this yields
\begin{equation}\label{eq 7.3a}
   \left|M[\om^R_1]\right|+\left|M[\om^R_2]\right|\lesssim \ve \left\|\om^R\right\|_{\cX_\ve} .
\end{equation}
\end{Remark}

\begin{proof}
Thanks to \eqref{eq 5.1a} and  $m[Z_i]=0$ for $Z =f^{te}_a,f^{to}_a,f^{o}_a, f^{e}_a$, $i=1,2$,
we can decompose $\om$ into
\begin{equation*}
    \begin{split}
\om =\mu^{te}  f^{te}_a + \mu^{to} f^{to}_a+  \mu^e f^e_a + \mu^o f^o_a + \om^R,
    \end{split}
\end{equation*}
with $\om^R$ satisfying \eqref{eq 7.2}, which is equivalent to
\begin{equation}\label{eq 7.3}
    \begin{split}
\cI \begin{pmatrix}
    \mu^{te} \\
    \mu^{to}\\
    \mu^e\\
    \mu^o
\end{pmatrix} = \begin{pmatrix}
    \lla \om, \varrho^{te} \rra_V \\
    \lla \om, \varrho^{to} \rra_V \\
    \lla \om, \varrho^{e} \rra_V\\
    \lla \om, \varrho^{o} \rra_V
\end{pmatrix} .
%= \begin{pmatrix}
 %   0 \\
 %   0 \\
 %   \lla \om, \varrho^{e} \rra_V\\
 %   \lla \om, \varrho^{o} \rra_V
%\end{pmatrix}
    \end{split}
\end{equation}
Since $(\om,\dot{\theta}_p)$ solves \eqref{eq 5.3}, we deduce from \eqref{eq 5.1a} that
$\lla \om, \varrho^{te} \rra_V = \lla \om, \varrho^{to} \rra_V=0$,
which together with \eqref{eq 7.3} and \eqref{eq 6.19}, gives rise to
\begin{equation*}
    \mu^{te}=\frac{\Bar{c}}{\Ga}\mu^{e}, \quad \mu^{to} = -\frac{\Bar{a}}{\Ga} \mu^o  ,  \quad \mu^o= \left(\Bar{b}+\frac{\Bar{a}\Bar{c}}{\Ga}\right)^{-1} \lla \om, \varrho^{e} \rra_V, \quad \mu^e = -\left(\Bar{b}+\frac{\Bar{a}\Bar{c}}{\Ga}\right)^{-1}   \lla \om, \varrho^{o} \rra_V.
\end{equation*}
This completes the proof of Lemma \ref{Lem 7.1}.
\end{proof}

\subsection{Choice of $\dot{\theta}_p$ and total energy.}
We next specify the modulation parameter $\dot{\theta}_p$, intrinsically linked to the linear structure and the pseudo–momenta of Section \ref{sect6}.

\begin{Lemma}\label{Lem 7.3}
  {\sl  Let $\om$ of the form \eqref{eq 7.1}  be the solution of \eqref{eq 5.3} with $ \dot{\theta}_p = \frac{\Ga \lam^e_a}{\ve  \ala} \mu^o$, we have
    \begin{equation}\label{eq 7.4}
    \begin{split}
&(t\p_t+1/2 )\mu^o \times \big(f^o_a - \frac{\Bar{a}}{\Ga}f^{to}_a\big)+   (t\p_t+1/2 )\mu^e \times \big(f^e_a +\frac{\Bar{c}}{\Ga} f^{te}_a\big)\\
&\qquad + (t\p_t-\cL +\nu^{-1}\Lam^E + \Lam^{NS}) \om^R  + \cT = -{\rm R}_a - \cN,
    \end{split}
\end{equation}
where we denote
\begin{equation}\label{eq 7.5}
    \begin{split}
\cT&\eqdefa \mu^o (t\p_t-\cL-1/2 ) \big(f^o_a - \frac{\Bar{a}}{\Ga}f^{to}_a\big) + \mu^e (t\p_t-\cL-1/2) \big(f^e_a +\frac{\Bar{c}}{\Ga} f^{te}_a\big) \\
&\quad-\frac{\mu^o}{\nu}  \left( \lbk{\rm R}_a^{o,s},\Omega^E_a \rbk_V-\lbk\varrho^o_a ,{\rm R}^E_a\rbk_V\right) -\frac{\Bar{a}\mu^o}{\nu\Ga}(\ve^2  \dot{\theta}^E_a f^{te}_a - \p_1 {\rm R}^E_a) \\
&\quad + \frac{\mu^e}{\nu} \left(X{\rm R}^E_a + \Lam^E\left[ {\rm R}_a^{e,z} \right]\right) + \frac{\Bar{c}\mu^e}{\nu \Ga}\bigl(-\ve^2\dot{\theta}^E_a f^{to}_a+\p_2{\rm R}^E_a\bigr)- \frac{\lam^e_a \mu^o}{\nu} {\rm R}_a^{e,z} \\
&\quad+\Lam^{NS}\big[\mu^o\big(f^o_a - \frac{\Bar{a}}{\Ga}f^{to}_a\big) + \mu^e \big(f^e_a +\frac{\Bar{c}}{\Ga} f^{te}_a\big)\big]  +  \lam^e_a \mu^o X\Omega^{NS}_a,
    \end{split}
\end{equation}
and
\begin{equation}\label{eq 7.6}
    \begin{split}
        \cN\eqdefa \nu^{-1}\{\cB_a \om, \om \}_V + \nu^{-1}\lam^e_a\mu^o \{\hat{X}, \om\}_V.
    \end{split}
\end{equation}}
\end{Lemma}

\begin{proof} We first, get,
by   direct calculations, that
 \begin{equation}\label{eq 7.7}
     \begin{split}
(t\p_t-\cL)\om
&= t\p_t\mu^o \times \big(f^o_a - \frac{\Bar{a}}{\Ga}f^{to}_a\big)+ t\p_t \mu^e \times \big(f^e_a +\frac{\Bar{c}}{\Ga} f^{te}_a\big) + (t\p_t-\cL) \om^R  \\
& \quad + \mu^o (t\p_t-\cL) \big(f^o_a - \frac{\Bar{a}}{\Ga}f^{to}_a\big) + \mu^e (t\p_t-\cL) \big(f^e_a +\frac{\Bar{c}}{\Ga} f^{te}_a\big) \\
& = (t\p_t+1/2 )\mu^o \times \big(f^o_a - \frac{\Bar{a}}{\Ga}f^{to}_a\big)+   (t\p_t+1/2 )\mu^e \times \big(f^e_a +\frac{\Bar{c}}{\Ga} f^{te}_a\big) + (t\p_t-\cL) \om^R   \\
& \quad +\mu^o (t\p_t-\cL-1/2 ) \big(f^o_a - \frac{\Bar{a}}{\Ga}f^{to}_a\big) + \mu^e (t\p_t-\cL-1/2) \big(f^e_a +\frac{\Bar{c}}{\Ga} f^{te}_a\big) .
     \end{split}
 \end{equation}

And by employing \eqref{eq 6.18} and choosing $ \dot{\theta}_p = \frac{\Ga \lam^e_a}{\ve  \ala} \mu^o$, one has
\begin{equation}
   \begin{split}
&\frac{1}{\nu} \Lam^E\om + \frac{1}{\nu}\frac{\ve }{\Ga}\ala \dot{\theta}_p X\Omega^{E}_a \\
&= \frac{1}{\nu} \Lam^E \big[\mu^o\big(f^o_a-\frac{\Bar{a}}{\Ga} f^{to}_a\big) + \mu^e \big(f^e_a +\frac{\Bar{c}}{\Ga} f^{te}_a\big)+ \om^R\big] + \frac{\lam^e_a \mu^o}{\nu} X\Omega^E_a\\
&= \frac{\mu^o}{\nu}\Lam^E f^o_a - \frac{\Bar{a}\mu^o}{\nu \Ga} \Lam^Ef^{to}_a +\frac{\mu^e}{\nu}\Lam^Ef^e_a+ \frac{\Bar{c}\mu^e}{\nu \Ga} \Lam^E f^{te}_a + \frac{1}{\nu}\Lam^E\om^R  + \frac{\lam^e_a \mu^o}{\nu} X\Omega^E_a\\
& = \frac{1}{\nu} \mu^o \left(-\lam^e_a f^e_a -\lbk{\rm R}_a^{o,s},\Omega^E_a \rbk_V+\lbk\varrho^o_a ,{\rm R}^E_a\rbk_V\right) - \frac{\Bar{a}\mu^o}{\nu\Ga}\bigl(\ve^2  \dot{\theta}^E_a f^{te}_a - \p_1 {\rm R}^E_a\bigr) \\
&\quad + \frac{\mu^e}{\nu} \left(X{\rm R}^E_a + \Lam^E\left[ {\rm R}_a^{e,z} \right]\right) + \frac{\Bar{c}\mu^e}{\nu \Ga}\bigl(-\ve^2\dot{\theta}^E_a f^{to}_a+\p_2{\rm R}^E_a\bigr)\\
&\quad + \frac{1}{\nu}\Lam^E \om^R +\frac{\lam^e_a \mu^o}{\nu} \left(f^e_a-{\rm R}_a^{e,z}\right) \\
&= \frac{1}{\nu} \Lam^E \om^R  -\frac{\mu^o}{\nu}  \left( \lbk{\rm R}_a^{o,s},\Omega^E_a \rbk_V-\lbk\varrho^o_a ,{\rm R}^E_a\rbk_V\right) -\frac{\Bar{a}\mu^o}{\nu\Ga}(\ve^2  \dot{\theta}^E_a f^{te}_a - \p_1 {\rm R}^E_a) \\
&\quad + \frac{\mu^e}{\nu} \left(X{\rm R}^E_a + \Lam^E\left[ {\rm R}_a^{e,z} \right]\right) + \frac{\Bar{c}\mu^e}{\nu \Ga}\bigl(-\ve^2\dot{\theta}^E_a f^{to}_a+\p_2{\rm R}^E_a\bigr)- \frac{\lam^e_a \mu^o}{\nu} {\rm R}_a^{e,z} .
\end{split}
\end{equation}
Finally notice that
\begin{equation}
\begin{split}
    &\Lam^{NS} \om + \frac{\ve }{\Ga}\ala \dot{\theta}_p X\Omega^{NS}_a \\
    &= \Lam^{NS}\big[\mu^o\big(f^o_a - \frac{\Bar{a}}{\Ga}f^{to}_a \big) + \mu^e \big(f^e_a +\frac{\Bar{c}}{\Ga} f^{te}_a\big)+\om^R\big]  +  \lam^e_a \mu^o X\Omega^{NS}_a,
\end{split}
\end{equation}
and
\begin{equation}\label{eq 7.10}
    \begin{split}
& \frac{1}{\nu}\{\cB_a \om, \om \}_V + \frac{1}{\nu}\frac{\ve }{\Ga} \al_a \thpp \{\hat{X}, \om\}_V =  \frac{1}{\nu}\{\cB_a \om, \om \}_V + \frac{\lam^e_a\mu^o}{\nu} \{\hat{X}, \om\}_V,
%&=  \frac{1}{\nu}\{\cB_a \om, \om \}_V + \frac{\lam^e_a \mu^o}{\nu} \lbk\hat{X}, \mu^o\left(f^o_a -\frac{\Bar{a}}{\Ga} f^{to}_a \right) + \mu^e f^e_a\rbk_V  + \frac{\lam^e_a \mu^o}{\nu} \{\hat{X}, \om^R\}_V,
    \end{split}
\end{equation}
by summarizing \eqref{eq 7.7}-\eqref{eq 7.10}, we achieve \eqref{eq 7.4} and thus complete the proof of Lemma \ref{Lem 7.3}.
\end{proof}

Now with the choice $ \dot{\theta}_p = \frac{\Ga \lam^e_a}{\ve  \ala} \mu^o$, the unknowns remain  $(\om^R, \mu^o, \mu^e)$, which is analogous to the toy model in Subsection \ref{subsec. toy model}.  We  recall energy functional and dissipation functional:
 \begin{equation*}
 \begin{split}
     &E_\ve[\om]= \f12 \lla\om, \om \dot\otimes W_\ve + \cB_a \om \rra_V, \\
     & D_\ve[\om]= -\f12 \lla  \om, \ \om \dot\otimes (t\p_t W_{\ve})\rra_V - \lla\cL \om, \ \om \dot\otimes W_\ve + \cB_a \om \rra_V,
    \end{split}
 \end{equation*}
 and then define
 \begin{equation}
    \cE \eqdefa (\mu^o)^2+ (\mu^e)^2+ E_\ve[\om^R] \andf \cD\eqdefa (\mu^o)^2+ (\mu^e)^2 + D_\ve[\om^R].
\end{equation}
\begin{Proposition}\label{Prop 7.4}
  {\sl  There exist $c_1, c_2,c_3>0$, so that
\begin{subequations}\begin{align}
  &   c_1 \|\om^R\|_{\cX_{\ve}}^2\leq  E_\ve[\om^R] \leq c_2 \|\om^R\|_{\cX_{\ve}}^2, \label{eq 7.12a} \\
   &   c_1 \|\om^R\|_{\cD_{\ve}}^2  \leq D_\ve[\om^R] \leq c_2 \|\om^R\|_{\cD_{\ve}}^2, \label{eq 7.12b}\\
  &  c_3 \|\om\|_{\cX_{\ve}}^2 \leq  \cE ,
   \qquad  c_3 \|\om\|_{\cD_{\ve}}^2  \leq \cD . \label{eq 7.12c}
\end{align}\end{subequations}}
\end{Proposition}
\begin{proof}
The upper bound part inequality in \eqref{eq 7.12a}-\eqref{eq 7.12b} is immediate, thus it suffices to consider the lower bounds. We first get by  Proposition \ref{Prop 5.5} and \eqref{eq 7.3a}
   \begin{equation}\label{eq 7.14}
 \sum_{i=1}^2\|\om^R_i\|_{\cX_{i,\ve}}^2  \lesssim \sum_{i=1}^2 \left( \|\om^R_i\|_{\cX_{i,\ve}}^2 + \lla  \phi^R_i, \om^R_i \rra  \right),
    \end{equation}
and
\begin{equation}
  \sum_{i=1}^2\|\om^R_i\|_{\cD_{i,\ve}}^2  \lesssim   \sum_{i=1}^2  \Big(- \f12 \lla (t\p_t W_{i,\ve}) \om^R_i, \om^R_i \rra - \lla \cL\om^R_i, W_{i,\ve} \om^R_i + \phi^R_i \rra \Big).
\end{equation}
Along similar arguments in \eqref{eq 5.21} and \eqref{eq 5.22}, one obtains, by using Proposition \ref{Prop 5.4} and \eqref{eq 7.3a}, that
\begin{equation}
\begin{split}
    &\sum_{i=1,2} \big|\big\langle \om^R_i,   \phi^R_\is\big(\cdot+ \ka_i\ve^{-1} \ala\eo\big) \big\rangle  \big|   \\
    &\leq\sum_{i=1}^2 \big\|\mathbbm{1}_{\ive} (1+|\cdot|)^{-3}\phi^R_\is\big(\cdot+ \ka_i\ve^{-1} \ala\eo\big)  \big\|_{L^4_\xi} \left\|(1+|\cdot|)^3 \om^R_i \right\|_{L^\f43_\xi}  \\
    &  \quad + \ve^{-1}\sum_{i=1,2}  \left\| (1+|\cdot|)^{-1} \phi^R_\is  \right\|_{L^4_\xi} \big\|\mathbbm{1}_{\mathrm{I}_{i,\ve}^c}(1+|\cdot|) \om^R_i\big\|_{L^\f43_\xi}\\
    &\lesssim  \sum_{i=1}^2\left( \ve^2\|\om^R_i\|_{\cX_{i,\ve}} +\ve\left|M[\om^R_i]\right|\right)\sum_{i=1,2}\|\om^R_i\|_{\cX_{i,\ve}} + \ve^{-1}\exp(-c\ve^{-2\sig_1})\sum_{i=1}^2\|\om^R_i\|_{\cX_{i,\ve}}^2 \\
    &\lesssim \ve^2\sum_{i=1}^2\|\om^R_i\|_{\cX_{i,\ve}}^2,
\end{split}
\end{equation}
and similarly
\begin{equation}\label{eq 7.15}
    \begin{split}
       \sum_{i=1}^2 \big|\big\langle \cL \om^R_i,    \phi^R_\is\big(\cdot+ \ka_i\ve^{-1} \ala\eo\big)\big\rangle   \big| \lesssim \ve^3\sum_{i=1}^2\|\om^R_i\|_{\cX_{i,\ve}}^2.
    \end{split}
\end{equation}
By summarizing \eqref{eq 7.14}-\eqref{eq 7.15}, we achieve \eqref{eq 7.12a} and  \eqref{eq 7.12b}. Finally thanks to $f^o_a,f^e_a, f^{to}_a,f^{te}_a\in\cZ \subset \cX_{i,\ve}$, it's easy to deduce \eqref{eq 7.12c} from \eqref{eq 7.12a} and \eqref{eq 7.12b}. This completes the proof of Proposition \ref{Prop 7.4}.
\end{proof}

\subsection{Estimation of $\cT$ and $\cN$.}

\begin{Proposition}\label{Prop 7.5}
  Let $\ve ,\sig_1 \ll1$.  There hold
  \begin{subequations} \label{eq 7.16a}
\begin{align}
   &\|\cT\|_{\cX_{\ve}}   \lesssim \left( \ve + \nu^{-1}\ve^{M+1}\right)\cE^\f12, \label{eq 7.16}\\
   & \left|\lla \cN,  \varrho^o_a \rra_V \right| +  \left|\lla \cN,  \varrho^e_a \rra_V \right| \lesssim \nu^{-1}  \cD, \label{eq 7.17}\\
   &\left| \lla \cN, \om^R \dot\otimes W_\ve +\cB_a \om^R \rra_V \right| \lesssim \nu^{-1} \cE^{\f12} \cD.\label{eq 7.18}
   \end{align}
\end{subequations}
\end{Proposition}
\begin{proof}
    Notice that
\begin{equation}\label{eq 7.21a}
\begin{split}
    &t\p_t\Bar{a}, \Bar{a}=\cO(\ve^{M-1}), \quad  t\p_t\Bar{c}, \Bar{c}=\cO(\ve^{M+1}) ,\quad  \Omega^{NS}_a= \cO_\cZ(\ve^2), \quad \Psi^{NS}_a = \cO_{\cS_\star}(\ve^2), \\
    &{\rm R}_a^{e,z},{\rm R}_a^{E} = \cO_\cZ(\ve^{M+1}),\quad {\rm R}_a^{o,s}= \cO_{\cS_\star}(\ve^{M+1}) +\Vec{C}(t), \quad {\rm R}_a^{e,s} = \cO_{\cS_\star}(\ve^{M+1}), \\
    &f^o_a = - \p_1 G Y_2  +\cO_\cZ(\ve), \quad f^e_a =
         \p_2 G Y_2 +\cO_\cZ(\ve), \quad \cL[\p_j G]= -1/2 \p_j G, \\
    &\varrho^o_a = \xi_2 Y_2+ \cO_{\cS_\star}(\ve), \quad \varrho^e_a = \xi_1 Y_2 +\cO_{\cS_\star}(\ve) ,\quad  \cL^\star [\xi_j] = - \xi_j/2,
\end{split}
\end{equation}
we deduce from \eqref{eq 7.5} that
\begin{equation*}
\begin{split}
  \|\cT\|_{\cX_{\ve}}  \lesssim \left( \ve + \nu^{-1}\ve^{M+1}\right)(|\mu^o|+|\mu^e|),
  \end{split}
\end{equation*}
which gives \eqref{eq 7.16}.

While by using Proposition \ref{Prop 5.2} and \eqref{eq 7.12c}, we find
\begin{equation}\label{eq 7.19}
    \|\grad_\xi \cB_a \om\|_{L^\oo_\xi} \lesssim \|\om\|_{\cX_\ve}^\f12 (\|\om\|_{\cX_\ve}^\f12 + \|\grad_\xi\om\|_{\cX_\ve}^\f12) \lesssim \cD^\f12,
\end{equation}
which implies
\begin{equation*}
    \begin{split}
\left|\lla \cN,  \varrho^o_a \rra_V \right| +  \left|\lla \cN,  \varrho^e_a \rra_V \right| \lesssim \nu^{-1}\|\grad_\xi \cB_a\om\|_{L^\oo_\xi} \|\grad_\xi\om\|_{\cX_\ve} + \nu^{-1}\ve^2 |\mu^o| \|\grad_\xi\om\|_{\cX_\ve} \lesssim \nu^{-1}\cD.
    \end{split}
\end{equation*}
This gives \eqref{eq 7.17}.

 Now for \eqref{eq 7.18}, we decompose $\cN$ into $\cN= \cN^{\mathrm{I}}+\cN^{\mathrm{II}} $  with
\begin{align*}
    &\cN^{\mathrm{I}}\eqdefa   \nu^{-1}\{\cB_a \om, \om \}_V + \nu^{-1}\lam^e_a \mu^o \big\{\hat{X}, \mu^o\big(f^o_a -\frac{\Bar{a}}{\Ga} f^{to}_a \big) + \mu^e \big(f^e_a +\frac{\Bar{c}}{\Ga} f^{te}_a \big)\big\}_V,   \\
    &\cN^{\mathrm{II}}\eqdefa \nu^{-1}\lam^e_a \mu^o \{\hat{X}, \om^R\}_V.
\end{align*}
Then by using  \eqref{eq 7.19}, \eqref{eq 7.12c}, integration by parts and H\"older inequality, we obtain
\begin{equation}\label{eq 7.21}
\begin{split}
    &\nu \left| \lla \cN^{\mathrm{I}}, \om^R \dot\otimes W_\ve +\cB_a \om^R \rra_V \right|
    \\
    &\leq \left\|\grad_\xi \cB_a \om\right\|_{L^\oo_\xi} \left\|\grad_\xi \om\right\|_{\cX_\ve} \left\|\om^R\right\|_{\cX_\ve} +\left\|\grad_\xi \cB_a \om\right\|_{L^\oo_\xi} \left\|\grad_\xi \cB_a \om^R\right\|_{L^\oo_\xi} \|\om\|_{\cX_\ve}\\
    &\quad + \ve^2\left(|\mu^o|^2+ |\mu^e|^2\right) \bigl(\left\|\om^R\right\|_{\cX_\ve} + \left\|\grad_\xi \cB_a \om^R \right\|_{L^\oo_\xi}\bigr)\\
    &\lesssim \cE^\f12 \cD.
    \end{split}
\end{equation}
And thanks to Proposition \ref{Prop 5.1} and $\sig_1 \ll1$,  we have
\begin{equation*}
    \begin{split}
& \nu \left| \lla \cN^{\mathrm{II}}, \om^R \dot\otimes W_\ve +\cB_a \om^R \rra_V \right|\\
&=   |\lam^e_a \mu^o|\big| \big\langle \big\{ \xi_1 Y_2 + (2\ala)^{-1}\ve\Ga|\xi|^2 Y_1 ,\om^R\big\}_V,\ \om^R\dot\otimes W_\ve  + \cB_a\om^R \big\rangle_V  \big| \\
&\leq   |\lam^e_a \mu^o|\big| \big\langle \big\{\xi_1 Y_2 + (2\ala)^{-1}\ve\Ga|\xi|^2 Y_1 ,\cB_a\om^R\big\}_V,\  \om^R \big\rangle_V  \big| \\
&\quad + |\lam^e_a \mu^o|\big|  \lla \lbk\xi_1 Y_2,\om^R\rbk_V,\  \om^R \dot\otimes W_\ve \rra_V  \big| +(4\ala)^{-1}\ve|\Ga \lam^e_a \mu^o| \big|  \lla \lbk|\xi|^2 Y_1 ,W_\ve\rbk_V,\  \om^R\dot\otimes \om^R \rra_V  \big|\\
&\lesssim  \ve^2 |\mu^o| \left\|\om^R\right\|_{\cX_\ve} \big( \left\|\grad_\xi \cB_a \om^R\right\|_{L^\oo_\xi} + \left\|\grad_\xi \om^R\right\|_{\cX_\ve} \big) +  \ve^3|\mu^o|\sum_{i=1}^2 \left|\lla \xi^\perp\cdot\grad (W_{i,\ve} - W_0) \mathbbm{1}_{\ive} , |\om^R_i|^2 \rra\right|   \\
&\lesssim  \ve^2 |\mu^o|\cD +  \ve^{3+0.5-\sig_1}|\mu^o| \cE \lesssim  \ve^2\cE^\f12\cD.
    \end{split}
\end{equation*}
By summarizing above inequality and \eqref{eq 7.21}, we achieve \eqref{eq 7.18}. This  completes the proof of Proposition \ref{Prop 7.5}
\end{proof}
\begin{Remark}
    The estimation of $\cN$ is  slightly different from previous literature \cite{DG24}. Here we have an extra  term $|\xi|^2$ in frame stream function, whose derivative is not bounded. We overcome this technical difficulty by noticing that $W_{i,\ve}$ is radial outside $\ive$ and thus $\{|\xi|^2, W_{i,\ve}\}=2\xi^\perp\cdot 
    \grad(W_{i,\ve}-W_0) \mathbbm{1}_{\ive}$. Similar idea is used in the proof of Proposition \ref{Prop 7.6}.
\end{Remark}

\subsection{Estimate of $\om^R$.}
\begin{Proposition}\label{Prop 7.6}
 {\sl   Let $\ve,\sig_1\ll1$. There hold
   \begin{subequations} \label{eq 7.23a}
\begin{align}
\label{eq 7.23}
&\left| \lla (t\p_t-\cL +\nu^{-1}\Lam^E + \Lam^{NS}) \om^R, \varrho^o_a \rra_V\right|\lesssim (\ve + \nu^{-1}\ve^{M+1}) \cE^\f12,\\
   \label{eq 7.24}
&\left| \lla (t\p_t-\cL +\nu^{-1}\Lam^E + \Lam^{NS}) \om^R, \varrho^e_a \rra_V\right|\lesssim (\ve + \nu^{-1}\ve^{M+1}) \cE^\f12,\\
   \label{eq 7.25}
& \lla (t\p_t-\cL ) \om^R , \om^R \dot\otimes W_\ve +\cB_a \om^R \rra_V \geq t\p_t E_\ve[\om^R] +  D_\ve[\om^R]-C\ve^2 \left\|\om^R\right\|_{\cX_\ve}^2,\\
   \label{eq 7.26}
  & \nu^{-1}\left| \lla \Lam^E \om^R , \om^R \dot\otimes W_\ve +\cB_a \om^R \rra_V \right| \lesssim \nu^{-1} (\ve^{M}+\ve^{\sig_2}) \cD,\\
   \label{eq 7.27}
&\left| \lla \Lam^{NS} \om^R , \om^R \dot\otimes W_\ve +\cB_a \om^R \rra_V \right| \lesssim  \ve^{2} \cD.
  \end{align}
\end{subequations}}
\end{Proposition}
\begin{proof} We first get,
by using \eqref{eq 7.2}, \eqref{eq 6.18}, \eqref{eq 7.21a}, that
\begin{equation*}
\begin{split}
   &\left| \lla (t\p_t-\cL + \nu^{-1}\Lam^E + \Lam^{NS}) \om^R, \varrho^o_a  \rra_V\right|  \\
   &\leq \left| \lla  \om^R, t\p_t \varrho^o_a  \rra_V\right| + \left| \lla  \om^R, \cL^\star \varrho^o_a  \rra_V\right| +\nu^{-1}\left| \lla  \om^R, \Lam^{E,\star} \varrho^o_a  \rra_V\right| + \left| \lla  \om^R, \Lam^{NS,\star} \varrho^o_a  \rra_V\right|\\
   &\lesssim \left(\ve+ \nu^{-1}\ve^{M+1}\right) \left\|\om^R\right\|_{\cX_\ve} \lesssim \left(\ve+ \nu^{-1}\ve^{M+1}\right) \cE^\f12,
\end{split}
\end{equation*}
which gives rises to \eqref{eq 7.23}, and  \eqref{eq 7.24} follows along the same line.

For \eqref{eq 7.25},  we observe by direct calculations that
\begin{equation}\label{eq 7.28}
    \begin{split}
&\lla (t\p_t-\cL ) \om^R , \om^R \dot\otimes W_\ve +\cB_a \om^R \rra_V \\
&= t\p_t E_\ve[\om^R] +  D_\ve[\om^R] + 1/2\lla t\p_t \om^R, \cB_a \om^R \rra_V - 1/2 \lla \om^R, t\p_t(\cB_a \om^R) \rra_V\\
&= t\p_t E_\ve[\om^R] + D_\ve[\om^R]   - 1/2 \lla \om^R, [t\p_t,\cB_a] \om^R \rra_V.
    \end{split}
\end{equation}
While one has, by using Propositions \ref{Prop 5.2}, \ref{Prop 5.4} and \eqref{eq 7.3a}, that
\begin{equation*}
    \begin{split}
&\left|\lla \om^R, [t\p_t,\cB_a] \om^R \rra_V\right| \leq \left|t\p_t\left(\ala/\ve\right)\right|\sum_{i=1}^2\big| \big\langle \p_1 \phi^R_\is\left(\xi+\ka_i\ve^{-1} \ala\eo\right) , \om^R_i  \big\rangle\big| \\
&\lesssim \ve^{-1}\sum_{i=1,2} \Big( \big| \big\langle \p_1 \phi^R_\is\left(\xi+\ka_i\ve^{-1} \ala\eo\right) , \om^R_i \mathbbm{1}_{\ive}  \big\rangle\big| +\big| \big\langle \p_1 \phi^R_\is\left(\xi+\ka_i\ve^{-1} \ala\eo\right) , \om^R_i \mathbbm{1}_{\mathrm{I}^c_{i,\ve}}  \big\rangle\big|   \Big)\\
&\lesssim \ve^{-1}\sum_{i=1,2} \Big( \big\|\mathbbm{1}_{\ive}(1+|\cdot|)^{-3} \grad_\xi \phi^R_\is\left(\xi+\ka_i\ve^{-1} \ala\eo\right)\big\|_{L^4_\xi} \left\|(1+|\cdot|)^{3} \om^R_i\right\|_{L^{4/3}_\xi} \\
&\quad \qquad \qquad \quad  + \left\| \grad_\xi \phi^R_\is\right\|_{L^4_\xi} \big\|\mathbbm{1}_{\mathrm{I}^c_{i,\ve}} \om^R_i \big\|_{L^{4/3}_\xi}    \Big) \\
& \lesssim \ve^{-1}\big(\ve^3 \left\|\om^R\right\|_{\cX_\ve} + \ve^2\left|M[\om^R_1]\right|+ \ve^2\left|M[\om^R_2]\right|\big)  \left\|\om^R\right\|_{\cX_\ve} + \ve^{-1} \exp\left(-c\ve^{-2\sig_1}\right)\left\|\om^R\right\|_{\cX_\ve}^2 \\
&\lesssim \ve^2 \left\|\om^R\right\|_{\cX_\ve}^2.
\end{split}
\end{equation*}
Therefore by inserting the above inequality into \eqref{eq 7.28}, we achieve \eqref{eq 7.25}.

For \eqref{eq 7.26}, we split
\begin{equation}\begin{split}\label{eq 7.30}
    & \lla \Lam^E \om^R , \om^R \dot\otimes W_\ve +\cB_a \om^R \rra_V = \lla \lbk \Psi^E_a, \om^R \rbk_V + \lbk\cB_a\om^R, \Omega^E_a\rbk_V , \om^R \dot\otimes W_\ve +\cB_a \om^R \rra_V\\
     & =- \f12 \lla \lbk \Psi^E_a, W_\ve\rbk_V, \om^R\dot\otimes\om^R \rra_V + \lla \lbk\cB_a \om^R , \Omega^E_a \rbk_V\dot\otimes W_\ve + \lbk \cB_a \om^R, \Psi^E_a \rbk_V    , \om^R\rra_V\\
     &\eqdefa A_1+A_2.
\end{split}
\end{equation}
We can bound $A_1$ by using Propositions \ref{Prop 4.7}, \ref{Prop 5.2} and \ref{Prop 5.4} that for some $N\in\N$,
\begin{equation}\label{eq 7.31}
\begin{split}
    &|A_1| \lesssim\sum_{i=1}^2 \left|\lla \{\Psi^E_{i,a}, W_{i,\ve}\} , \mathbbm{1}_{\ive} |\om^R_i|^2\rra\right| +\sum_{i=1}^2 \left|\lla \{\Psi^E_{i,a}, W_{i,\ve}\} , \mathbbm{1}_{\iiive} |\om^R_i|^2\rra\right| \\
    &\lesssim \sum_{i=1}^2 \left|\lla \{\Theta_{i}, W_{i,\ve}\} , \mathbbm{1}_{\ive} |\om^R_i|^2\rra\right|  \\
    &\quad +\sum_{i=1}^2 \big| \big\langle \big\{ \left(\cB_a\Omega^E_a\right)_i + \ve  \ka_i \Ga_i\dot{\theta}^E_a \ala \xi_1/\Ga, \ \exp(|\xi|^{2\ga}/4)\big\} , \mathbbm{1}_{|\xi|\geq \ve^{-\sig_2}} |\om^R_i|^2\big\rangle\big|\\
    &\lesssim \ve^{M+1-N\sig_1} \left\|\om^R\right\|_{\cX_\ve}^2  + \ve^{\sig_2}\big(\ve +\sum_{i=1}^2\left\|\grad_\xi \cB_a\Omega^E_a  \right\|_{L^\oo_\xi}   \big) \left\|\om^R\right\|_{\cD_\ve}^2\\
    &\lesssim \left(\ve^{M+1-N\sig_1}+\ve^{\sig_2}\right) \cD .
\end{split}
\end{equation}
While we can similarly bound $A_2$ by
\begin{equation}\label{eq 7.32}
    \begin{split}
|A_2|&\lesssim \sum_{i=1}^2 \left| \lla \lbk \left(\cB_a \om^R\right)_i, \Omega^E_{i,a}  \rbk F'_i(\Omega^E_{i,a}) + \lbk\left(\cB_a \om^R\right)_i,\Psi^E_{i,a}\rbk , \mathbbm{1}_{\ive} \om^R_i  \rra  \right| \\
& \quad +\sum_{i=1}^2 \big| \big\langle \lbk \left(\cB_a \om^R\right)_i, \Omega^E_{i,a}  \rbk W_{i,\ve} + \lbk\left(\cB_a \om^R\right)_i,\Psi^E_{i,a}\rbk , \mathbbm{1}_{\mathrm{I}_{i,\ve}^c} \om^R_i  \big\rangle  \big|\\
&\lesssim \sum_{i=1}^2 \left| \lla \lbk \left(\cB_a \om^R\right)_i, \Theta_i  \rbk , \mathbbm{1}_{\ive} \om^R_i  \rra  \right| \\
& \quad +\sum_{i=1}^2 \big| \big\langle \lbk \left(\cB_a \om^R\right)_i, \Omega^E_{i,a}  \rbk W_{i,\ve} + \lbk\left(\cB_a \om^R\right)_i,\Psi^E_{i,a}\rbk , \mathbbm{1}_{\mathrm{I}_{i,\ve}^c} \om^R_i  \big\rangle  \big|\\
&\lesssim \left(\ve^{M+1}+ \exp\left(-c\ve^{-2\sig_1}\right) \right) \left\|\grad_\xi\cB_a\om^R\right\|_{L^\oo_\xi}   \left\|\om^R\right\|_{\cX_\ve} \\
&\lesssim \left(\ve^{M+1}+ \exp\left(-c\ve^{-2\sig_1}\right) \right) \cD.
    \end{split}
\end{equation}
By inserting \eqref{eq 7.31}, \eqref{eq 7.32} into \eqref{eq 7.30}, and then letting $\ve, \sig_1\ll 1$, we achieve \eqref{eq 7.26}.

Now we turn to the proof of \eqref{eq 7.27}. We decompose $\Psi^{NS}_a$ into $\Psi^{NS}_a = \Psi^{NS,\mathrm{I}}_a+ \Psi^{NS,\mathrm{II}}_a$ with
\begin{equation*}
\begin{split}
   &\Psi^{NS,\mathrm{I}}_a\eqdefa  \cB_a \Omega^{NS}_a + \frac{\ve }{\Ga}(\dot{\theta}^{NS}_a \al_a \xi_1 + \dot{\al}^{NS}_a\xi_2) Y_2, \quad \Psi^{NS,\mathrm{II}}_a\eqdefa  \frac{\ve^2}{2} \dot{\theta}^{NS}_a |\xi|^2 Y_1,
   \end{split}
\end{equation*}
and then split
\begin{equation}\label{eq 7.33}
    \begin{split}
 &\lla \Lam^{NS} \om^R , \om^R \dot\otimes W_\ve +\cB_a \om^R \rra_V \\
 &=  \lla \lbk\Psi^{NS,\mathrm{I}}_a, \om^R \rbk_V , \om^R \dot\otimes W_\ve +\cB_a \om^R \rra_V + \lla \lbk\Psi^{NS,\mathrm{II}}_a, \om^R \rbk_V , \om^R \dot\otimes W_\ve +\cB_a \om^R \rra_V  \\
 &\quad +  \lla \lbk\cB_a\om^R, \Omega^{NS}_a \rbk_V, \om^R \dot\otimes W_\ve +\cB_a \om^R \rra_V \eqdefa A_3+A_4+A_5.
    \end{split}
\end{equation}
Using integration by parts, $\dot{\theta}^{NS}_0=\dot{\al}^{NS}_0=0$ and $\Omega^{NS}_a =\cO_\cZ(\ve^2)$, we get
\begin{equation}\label{eq 7.33a}
    \begin{split}
|A_3|&\leq   \big|\big\langle \big\{\Psi^{NS,\mathrm{I}}_a, \om^R \big\}_V ,  \om^R \dot\otimes W_\ve \big\rangle_V  \big| + \big|\big\langle \big\{\Psi^{NS,\mathrm{I}}_a, \cB_a \om^R \big\}_V , \om^R   \big\rangle_V\big|\\
&\lesssim \left\|\grad_\xi \Psi^{NS,\mathrm{I}}_a\right\|_{L^\oo_\xi}  \big(\left\|\grad_\xi\om^R\right\|_{\cX_\ve} + \left\|\grad_\xi \cB_a\om^R\right\|_{L^\oo_\xi}  \big) \left\|\om^R\right\|_{\cX_\ve} \lesssim \ve^2 \cD,
    \end{split}
\end{equation}
and similarly
\begin{equation}
    \begin{split}
|A_5|\lesssim \left\|\grad_\xi \cB_a\om^R\right\|_{L^\oo_\xi}\left\|\grad_\xi\Omega^{NS}_a\right\|_{\cX_\ve} \left\|\om^R\right\|_{\cX_\ve}\lesssim \ve^2 \cD.
    \end{split}
\end{equation}
Finally for $A_4$, one obtains, by using Proposition \ref{Prop 5.1} and taking $\sig_1$ to be sufficiently small, that
\begin{equation}\label{eq 7.35}
    \begin{split}
|A_4|&\leq 1/2 \times \big|\lla \lbk\Psi^{NS,\mathrm{II}}_a, W_\ve \rbk_V , \om^R \dot\otimes \om^R \rra_V \big| + \big|\lla \lbk\Psi^{NS,\mathrm{II}}_a, \cB_a\om^R \rbk_V , \om^R \rra_V \big|\\
&\lesssim\sum_{i=1}^2 \big|\big\langle \big\{\Psi^{NS,\mathrm{II}}_{i,a}, W_{i,\ve} \big\} , \mathbbm{1}_{\ive} |\om^R_i|^2 \big\rangle_V \big|  + \ve^2   \left\|\grad_\xi \cB_a\om^R\right\|_{L^\oo_\xi} \left\|\om^R\right\|_{\cX_\ve} \\
&\lesssim \ve^2\sum_{i=1}^2 \left|\lla  \xi^\perp \cdot\grad_\xi (W_{i,\ve}-W_0),   \mathbbm{1}_{\ive}|\om^R_i|^2 \rra_V \right| + \ve^2 \cD\\
&\lesssim \ve^{2.5-\sig_1} \cE^2+ \ve^2 \cD \lesssim \ve^2 \cD.
    \end{split}
\end{equation}
By inserting \eqref{eq 7.33a}-\eqref{eq 7.35} into \eqref{eq 7.33}, we achieve \eqref{eq 7.27}. This completes the proof of Proposition \ref{Prop 7.6}.

\end{proof}

\subsection{Estimation of $\mu^o$ and $\mu^e$.}

\begin{Proposition}\label{Prop 7.7}
 {\sl Let $\ve,\sig_1\ll1$ and
  \begin{equation}
      \cF\eqdefa (t\p_t+1/2 )\mu^o \times \big(f^o_a - \frac{\Bar{a}}{\Ga}f^{to}_a\big)+   (t\p_t+1/2 )\mu^e \times \big(f^e_a +\frac{\Bar{c}}{\Ga} f^{te}_a\big).
  \end{equation}
  There hold
\begin{subequations}  \label{eq 7.36a}
\begin{gather}
\lla \cF , \   -\varrho^o_a \rra_V  = \left(\Bar{b} + \frac{\Bar{a}\Bar{c}}{\Ga}\right) (t\p_t + 1/2)\mu^e ,\quad  \lla \cF, \  \varrho^e_a \rra_V  = \left(\Bar{b} + \frac{\Bar{a}\Bar{c}}{\Ga}\right)(t\p_t + 1/2)\mu^o, \label{eq 7.36}\\
\left| \lla \cF, \   \om^R \dot\otimes W_\ve +\cB_a \om^R \rra_V\right|  \lesssim \ve^{M+1} \left(|t\p_t \mu^o| + |\mu^o| +|t\p_t \mu^e| + |\mu^e|\right) \cE^\f12.\label{eq 7.38}
\end{gather}
\end{subequations}}
\end{Proposition}
\begin{proof}
    \eqref{eq 7.36} simply follows \eqref{eq 6.19}. Now for \eqref{eq 7.38}, notice that
    \begin{equation*}
\begin{split}
    &\lla f^o_a, \om^R \dot\otimes W_\ve +\cB_a \om^R\rra_V =  \lla \lbk\varrho^o_a, \Omega^E_a\rbk_V, \om^R \dot\otimes W_\ve +\cB_a \om^R\rra_V \\
    &= \lla \lbk\varrho^o_a, \Omega^E_a\rbk_V \dot\otimes W_\ve +\cB_a\lbk\varrho^o_a, \Omega^E_a\rbk_V, \om^R \rra_V\\
    &=- \lla \Lam^{E,\star}\varrho^o_a, \om^R\rra_V + \lla \lbk\varrho^o_a, \Omega^E_a\rbk_V \dot\otimes W_\ve - \lbk \Psi^E_a,\varrho^o_a\rbk_V, \om^R\rra_V \\
    &= - \lla {\rm R}_a^{o,s}, \om^R\rra_V +\sum_{i=1}^2 \lla \lbk\varrho^o_{i,a}, \Omega^E_{i,a}\rbk  F'_i(\Omega^E_{i,a}) - \lbk \Psi^E_{i,a},\varrho^o_{i,a}\rbk, \om^R_i \mathbbm{1}_{\ive} \rra \\
    &\quad +\sum_{i=1}^2 \big\langle \lbk\varrho^o_{i,a}, \Omega^E_{i,a}\rbk  W_{i,\ve} - \lbk \Psi^E_{i,a},\varrho^o_{i,a}\rbk, \om^R_i \mathbbm{1}_{\mathrm{I}_{i,\ve}^c} \big\rangle\\
    &= - \lla {\rm R}_a^{o,s}, \om^R \rra_V +\sum_{i=1,2} \lla \lbk\varrho^o_{i,a}, \Theta_i\rbk , \om^R_i \mathbbm{1}_{\ive} \rra \\
    &\quad +\sum_{i=1}^2 \big\langle \lbk\varrho^o_{i,a}, \Omega^E_{i,a}\rbk  W_{i,\ve} - \lbk \Psi^E_{i,a},\varrho^o_{i,a}\rbk, \om^R_i \mathbbm{1}_{\mathrm{I}_{i,\ve}^c} \big\rangle,
\end{split}
    \end{equation*}
we can bound $ \lla f^o_a, \om^R \dot\otimes W_\ve +\cB_a \om^R\rra_V$ by using Proposition \ref{Prop 4.7} and \eqref{eq 7.21a}:
\begin{equation*}
    \begin{split}
     \left|\lla f^o_a, \om^R \dot\otimes W_\ve +\cB_a \om^R\rra_V\right| \lesssim \left\|\om^R\right\|_{\cX_\ve}\left(\ve^{M+1}   +\exp\left(-c\ve^{-2\sig_1}\right)\right) \lesssim   \ve^{M+1} \cE^\f12.
    \end{split}
\end{equation*}
And other term can be estimated along  similar lines, we thus complete the proof of \eqref{eq 7.38}.
\end{proof}

\subsection{proof of Theorem \ref{Thm 1.1}}
We are now in a position to present the proof of Theorem \ref{Thm 1.1}.
\begin{proof}
   By taking $\lla \cdot, \cdot\rra_V$ inner product of \eqref{eq 7.4} with $\varrho^e$ and utilizing  ${\rm R}_a=\cO_\cZ\left(\nu^{-1}\ve^{M+1}+\nu\ve^2\right)$,  we get, by using Propositions \ref{Prop 7.5}-\ref{Prop 7.7}, that
\begin{equation}\label{eq 7.39}
    \begin{split}
   \left|\Bar{b} + \frac{\Bar{a}\Bar{c}}{\Ga}\right| \left|(t\p_t + 1/2)\mu^o\right| \lesssim   \nu^{-1}\ve^{M+1} +\nu \ve^2 + \left( \ve + \nu^{-1}\ve^{M+1}\right) \cE^\f12  + \nu^{-1}  \cD,
    \end{split}
\end{equation}
which  implies that
\begin{equation}\label{eq 4.65}
\begin{split}
    t\p_t|\mu^o|^2 + |\mu^o|^2 \lesssim \left(\nu^{-1}\ve^{M+1} +\nu \ve^2\right)\cE^\f12 + \left( \ve + \nu^{-1}\ve^{M+1}\right) \cE + \nu^{-1}\cE^\f12 \cD.
\end{split}
\end{equation}
Similarly, we have
\begin{equation}\label{eq 7.41}
    \begin{split}
    \left|\Bar{b} + \frac{\Bar{a}\Bar{c}}{\Ga}\right| \left|(t\p_t + 1/2)\mu^e\right| \lesssim   \nu^{-1}\ve^{M+1} +\nu \ve^2 + \left( \ve + \nu^{-1}\ve^{M+1}\right) \cE^\f12  + \nu^{-1}  \cD,
    \end{split}
\end{equation}
and
\begin{equation}\label{eq 7.42}
\begin{split}
   t\p_t|\mu^e|^2 + |\mu^e|^2 \lesssim \left(\nu^{-1}\ve^{M+1} +\nu \ve^2\right)\cE^\f12 + \left( \ve + \nu^{-1}\ve^{M+1}\right) \cE + \nu^{-1}\cE^\f12 \cD.
\end{split}
\end{equation}

  While by taking $\lla \cdot, \cdot\rra_V$ inner product of \eqref{eq 7.4} with $\om^R \dot\otimes W_\ve + \cB_a\om^R$,  we get, by
  using Propositions \ref{Prop 7.5}-\ref{Prop 7.7}, that
\begin{equation*}
    \begin{split}
&t\p_t E_\ve[\om^R] + D_\ve[\om^R]- C\ve^2 \left\|\om^R\right\|_{\cX_\ve}^2 \\
&\lesssim  \left(\nu^{-1}\ve^{M+1} +\nu \ve^2\right)\cE^\f12 + \left( \ve + \nu^{-1}\ve^{M+1}\right) \cE + \nu^{-1} \cE^\f12 \cD\\
&\quad + \nu^{-1}(\ve^M+\ve^{\sig_2})\cD+\ve^2 \cD +\ve^{M+1}(|t\p_t \mu^o| +|t\p_t \mu^e|  + \cE^\f12) \cE^\f12 ,
    \end{split}
\end{equation*}
which together with \eqref{eq 7.39} and \eqref{eq 7.41} implies
\begin{equation}\label{eq 7.43}
    \begin{split}
&t\p_t E_\ve[\om^R] + D_\ve[\om^R] \\
&\lesssim \left(\nu^{-1}\ve^{M+1} +\nu \ve^2\right)\cE^\f12 + \left( \ve + \nu^{-1}\ve^{M+1}\right) \cE + \nu^{-1} \cE^\f12 \cD +\nu^{-1}\left(\ve^M+\ve^{\sig_2}\right)\cD+\ve^2 \cD \\
&\quad +\ve^{M+1}\left(\nu^{-1}\ve^{M+1} +\nu \ve^2 + \left( \ve + \nu^{-1}\ve^{M+1}\right) \cE^\f12  + \nu^{-1}  \cD + \cE^\f12\right) \cE^\f12 \\
&\lesssim \left(\nu^{-1}\ve^{M+1} +\nu \ve^2\right)\cE^\f12 + \left( \ve + \nu^{-1}\ve^M + \nu^{-1}\ve^{\sig_2}\right) \cD +\nu^{-1} \cE^\f12 \cD.
    \end{split}
\end{equation}

 By summarizing \eqref{eq 4.65}, \eqref{eq 7.42} and \eqref{eq 7.43}, we find
 \begin{equation}\label{eq 7.44}
     \begin{split}
&t\p_t \cE + \cD  \lesssim\left(\nu^{-1}\ve^{M+1} +\nu \ve^2\right)\cE^\f12 + \left( \ve + \nu^{-1}\ve^M + \nu^{-1}\ve^{\sig_2}\right) \cD  +\nu^{-1} \cE^\f12 \cD .
     \end{split}
 \end{equation}
 For any given $s\in(0,1)$, we take $\sig_2 = M > \max(\frac{5-s}{1-s},\frac{2}{1-s}) = \frac{5-s}{1-s}$, which ensures that
 \begin{equation*}
    \nu^{-1}\ve^M + \nu^{-1}\ve^{\sig_2} \ll 1 \andf \nu^{-1}\ve^{M+1} \ll \nu \ve^2 .
 \end{equation*}
 Then by virtue of \eqref{eq 7.44}, $\cE(0)=0$ and continuity argument, we deduce that for $t\leq \nu^{-s}$,
 \begin{equation*}
     \cE + \int_0^t \frac{1}{\tau} \cD(\tau) d\tau  \lesssim \left(\nu^{-1}\ve^{M+1} +\nu \ve^2\right)^2 \ll (\nu \ve^2)^2,
 \end{equation*}
 which together with Proposition \ref{Prop 7.4}, $\dot{\theta}_p=\frac{\Ga\lam^e_a}{\ve \ala}\mu^o$ gives rise to
 \begin{equation*}
  \sum_{i=1}^2 \left\| \Omega_i - \Omega_{i,a} \right\|_{\cX_{i,\ve}} \lesssim \nu \ve^2 \andf |\dot{\theta} - \dot{\theta}_a|\lesssim \nu \ve^3.
 \end{equation*}
In particular, this gives
\begin{equation}\label{eq 7.46}
  \sum_{i=1}^2 \|\Omega_i - \Ga_i G\|_{L^1_\xi} \lesssim \ve^2 \andf \left|\dot{\theta} -\dot{\theta}^E_a \right| \lesssim \nu \ve^2 \ \Rightarrow \ \left|\theta -\theta^E_a \right| \lesssim  \ve^4.
\end{equation}

Now recalling \eqref{eq 2.2} and  then setting
$ \xi= \frac{1}{\sqrt{\nu t}} \left(\cR(\theta) x - r_i \eo \right)$, \eqref{eq 7.46} results in
\begin{equation}\label{eq 7.47}
 \Bigl\|w_i(t) -\frac{\Ga_i}{\nu t} G\Bigl(\frac{ x - \cR^{-1}(\theta)[r_i \eo]}{\sqrt{\nu t}}   \Bigr) \Bigr\|_{L^1_x}=    \Bigl\|w_i(t) -\frac{\Ga_i}{\nu t} G\Bigl(\frac{\cR(\theta) x - r_i \eo}{\sqrt{\nu t}}   \Bigr) \Bigr\|_{L^1_x} \lesssim \ve^2.
\end{equation}

Finally, by using \eqref{eq 7.46} and $r_i = \frac{\ka_i \Ga_\is}{\Ga}\ala = \ell_i \ala = \ell_i + \cO(\ve^4)$, we obtain
\begin{equation*}
  \left|\cR^{-1}(\theta)[r_i \eo] - \cR^{-1}(\theta^E_a) [\ell_i \eo] \right|\lesssim \ve^4,
\end{equation*}
from which and \eqref{eq 7.47}, we deduce that for $i=1,2$,
\begin{equation*}
    \Bigl\|w_i(t) -\frac{\Ga_i}{\nu t} G\Bigl(\frac{ x - \cR^{-1}(\theta^E_a)[\ell_i \eo]}{\sqrt{\nu t}}   \Bigr) \Bigr\|_{L^1_x} \lesssim \ve^2  \lesssim \nu t.
\end{equation*}
Thus by summarizing the above inequality and then setting $ \beta_k \eqdefa - \dot{\theta}^E_k$, we complete the proof of Theorem \ref{Thm 1.1}.

\end{proof}

\subsection*{acknowledgements} P. Zhang is partially  supported by National Key R$\&$D Program of China under grant 2021YFa1000800 and by National Natural Science Foundation of China under Grant 12421001, 12494542 and 12288201. Y. Zhang wants to thank Ning Liu for stimulating discussions.

\end{document}